\documentclass[12pt]{amsart}
\usepackage{srcltx}
\usepackage{varioref}
\usepackage[                
   pdftex,                      
   pdfstartview=FitH,             
   bookmarks=true,          
	 bookmarksdepth=subsection, 
   bookmarksopen=true,     
   bookmarksnumbered=true]{hyperref}
\hypersetup{                         
   colorlinks=true,           
   allcolors=[rgb]{0,0.3,0.6},                  
}

\allowdisplaybreaks

\theoremstyle{plain}

\newtheorem{theorem}{Theorem}
\newtheorem{lemma}[theorem]{Lemma}
\newtheorem{proposition}[theorem]{Proposition}
\newtheorem{corollary}[theorem]{Corollary}
\newtheorem{conjecture}[theorem]{Conjecture}

\theoremstyle{definition}

\newtheorem{remark}[theorem]{Remark}
\newtheorem{definition}[theorem]{Definition}

 \numberwithin{equation}{section}
 \numberwithin{theorem}{section}

\newcommand*{\refh}[2]{\hyperref[#2]{#1~\ref{#2}}} % ie \hyref{Section}{sec:RepS-RepQ} gives: Section 3

\newcounter{ourcount}
\advance\textheight\topskip
\usepackage{amsmath,amsfonts,amsthm,amssymb,amscd}
\usepackage{mathtools} 
\allowdisplaybreaks

 \usepackage{graphicx}
 \usepackage[curve,matrix,arrow,color]{xy}

 \usepackage[usenames,dvipsnames]{color}
 \xyoption{line}

\usepackage{xr}
\newcommand{\np}{\mathrm N}  

\usepackage[mathscr]{eucal}
\usepackage[OT2,OT1]{fontenc}

\DeclareMathOperator{\Gr}{Gr}

\newcommand{\half}{%
  \mathchoice{\ffrac{1}{2}}{\frac{1}{2}}{\frac{1}{2}}{\frac{1}{2}}}
\newcommand{\rmi}{\mathrm{i}}
\newcommand{\id}{\mathrm{id}}

\newcommand{\oN}{\mathbb{N}}

\newcommand{\SLiiZ}{SL(2,\oZ)}

\newcommand{\modS}{\mathscr{S}}
\newcommand{\modT}{\mathscr{T}}

\newcommand{\Q}{\mathsf{Q}}
\newcommand{\QQ}{\hat{\Q}}
\newcommand{\Qhat}{\hat{\mathsf{Q}}}
\newcommand{\Salg}{\mathsf{S}}

\newcommand{\B}{\mathsf{B}}
\newcommand{\Bel}{\mathsf{b}}

\newcommand{\z}{\mathsf{z}}
\newcommand{\copS}{\Delta^{\Salg}}
\newcommand\Zc{Z}

\newcommand{\brS}{\psi}               %%% for braiding in RepS
\newcommand{\smult}{\,\hat.\,}
\newcommand{\RS}{r}

\newcommand{\cat}{\mathcal{C}}
\newcommand{\catD}{\mathcal{D}}
\newcommand{\catSF}{\mathcal{S\hspace{-1.2pt}F}}
\newcommand{\svect}{\mathrm{{\bf Svect}}}
\newcommand{\vect}{\mathrm{{\bf vect}}}

\newcommand{\rep}{\mathrm{\bf Rep}\,}

\newcommand{\repsv}{\mathrm{\bf Rep}_{\mathrm{s.v.}}}
\newcommand{\repQ}{\rep\Q}
\newcommand{\repQQ}{\rep\Qhat}
\newcommand{\repS}{\rep\Salg}

\newcommand{\tensC}{\tensor_{\cat}}
\newcommand{\assC}{\alpha^{\cat}}
\newcommand{\tensD}{\tensor_{\catD}}
\newcommand{\assD}{\alpha^{\catD}}

\newcommand{\fun}{\mathcal{F}}    %%% a functor
\newcommand{\funD}{\mathcal{D}}   %%% From C to RepS
\newcommand{\funCS}{\funD}        %%% From C to RepS
\newcommand{\funSQ}{\mathcal{G}}  %%% From RepS to RepQhat
\newcommand{\funQS}{\mathcal{H}}  %%% From RepQhat to RepS
\newcommand{\funQSF}{\mathcal{F}} %%% From Q to SF

\newcommand{\isoD}{\Delta}
\newcommand{\isoG}{\Gamma}
\newcommand{\tensor}{\otimes}
\newcommand{\ot}{\otimes}

\newcommand{\svtensor}{\tensor_{\svect}}
\newcommand{\Ctensor}{*}

\newcommand{\Stensor}{\tensor_{\repS}}

\newcommand{\as}{\Phi}            %%% in Q
\newcommand{\asQQ}{\hat\Phi}      %%% in Q hat
\newcommand{\Sas}{\Lambda}        %%% in S

\newcommand{\assocQ}{\alpha^{\repQ}} %%% for associators in RepQ
\newcommand{\assocQQ}{\alpha^{\repQQ}} %%% for associators in RepQ
\newcommand{\flip}{\tau}                    %%% in Vect
\newcommand{\sflip}{\tau^{\mathrm{s.v.}}}   %%% in Svect

\newcommand{\tid}{\one}      %%% Tensor identity

\newcommand{\ev}{\mathrm{ev}}
\newcommand{\coev}{\mathrm{coev}}

\newcommand{\Salpha}{\boldsymbol{\alpha}}
\newcommand{\Sbeta}{\boldsymbol{\beta}}

\newcommand\hLie{\mathfrak{h}}

\newcommand{\tr}{\mathrm{Tr}}
\newcommand{\bal}{\varkappa}
\newcommand{\CM}{\mathsf{C}}

\newcommand{\bchi}{\pmb{\chi}}
\newcommand{\bphi}{\pmb{\phi}}

\newcommand{\xx}{\mathsf{x}}
\newcommand{\xp}{\mathsf{x}^+}
\newcommand{\xm}{\mathsf{x}^-}
\newcommand{\xpm}{\mathsf{x}^{\pm}}
\newcommand{\xmp}{\mathsf{x}^{\mp}}
\newcommand{\LL}{\mathsf{L}}

\newcommand{\oneS}{\one_{\mathsf{S}}}

\newcommand{\Za}{\Zc_\algGr}
\newcommand{\Zb}{\Zc_P}
\newcommand{\Sa}{\modS_{\Za}}
\newcommand{\Sb}{\modS_{\Zb}}
\newcommand{\Ta}{\modT_{\Za}}
\newcommand{\Tb}{\modT_{\Zb}}

\newcommand{\ffrac}[2]{\mbox{\footnotesize$\displaystyle\frac{#1}{#2}$}}

\newcommand{\one}{\boldsymbol{1}}%{1\kern-4pt 1}

\newcommand{\leftact}{\triangleright}

\newcommand{\oC}{\mathbb{C}}
\newcommand{\End}{\mathrm{End}}

\newcommand{\Hom}{\mathrm{Hom}}

\newcommand{\K}{\mathsf{K}}

\newcommand{\ff}{\mathsf{f}}
\newcommand{\fpm}{\ff^\pm}
\newcommand{\fp}{\ff^+}
\newcommand{\fm}{\ff^-}

\newcommand{\eps}{\varepsilon}

\newcommand{\algGr}{\mathbb{\Lambda}}
\newcommand{\algCl}{\mathsf{Cl}}

\newcommand{\coend}{\mathcal L}

\newcommand{\IntL}{\Lambda_\coend}
\newcommand{\Hpair}{\omega_\coend}
\newcommand{\muc}{\mu_\coend}

\newcommand\eqpics[4]{\begin{eqnarray*}
                   \begin{picture}(#2,#3){}\end{picture}\nonumber\\
                   \raisebox{-#3pt}{ \begin{picture}(#2,#3) #4 \end{picture} }
                   \label{#1} \\~\nonumber \end{eqnarray*} }

\newcommand{\XX}{\mathsf{X}} %<---- irrep of qalg

\newcommand{\PP}{\mathsf{P}}

\newcommand{\intQ}{{\mathbf{c}}}
\newcommand{\coint}{\Lambda^{\rm co}}

\newcommand{\balance}{{\boldsymbol{g}}}
\newcommand{\ribbon}{{\boldsymbol{v}}}
\newcommand{\sqs}{{\boldsymbol{u}}}
\newcommand{\Dt}{{\boldsymbol{f}}}

\newcommand{\idem}{\boldsymbol{e}}

\newcommand{\oZ}{\mathbb{Z}}

\newcommand{\gcg}{\ , \ } % gap-comma-gap 
\newcommand{\gp}{\ .}
\newcommand{\gc}{\ ,}
\newcommand{\qcq}{\quad , \quad}

\newcommand{\qc}{\quad ,}

\newcommand{\qp}{\quad .}

\newcommand\be{\begin{equation}}
\newcommand\ee{\end{equation}}

\newcommand{\al}[1]{\begin{align}#1\end{align}}

\usepackage[bbgreekl]{mathbbol}

\providecommand{\keyword}[1]{\textit{{keywords:\;}} #1}

\begin{document}

%\begin{flushright}
%ZMP-HH/17-21\\
%Hamburger Beitr\"age zur Mathematik 668
%\end{flushright}
%\vspace{2em}

\title[Symplectic fermion ribbon q-Hopf algebra and $\SLiiZ$-action on its centre]
{The symplectic fermion ribbon quasi-Hopf algebra
\\
 and the $\SLiiZ$-action on its centre}

\address{VF: Fachbereich Mathematik, Universit\"at Hamburg, Bundesstra\ss e 55,
20146 Hamburg, Germany}
\email{vanda.farsad@uni-hamburg.de}

\author{V.~Farsad, A.M.~Gainutdinov, I.~Runkel}
\address{AMG:
Institut Denis Poisson, CNRS, Universit\'e de Tours, Universit\'e d'Orl\'eans,
Parc de Grandmont, 37200 Tours, 
France}
\email{azat.gainutdinov@lmpt.univ-tours.fr}

\address{IR: Fachbereich Mathematik, Universit\"at Hamburg, Bundesstra\ss e 55,
20146 Hamburg, Germany}
\email{ingo.runkel@uni-hamburg.de}

\begin{abstract}
We introduce a family of factorisable ribbon quasi-Hopf algebras $\Q(\np)$ for $\np$ a positive integer:
as an algebra, $\Q(\np)$ is the semidirect product of $\oC\oZ_2$ with the direct sum of a Gra\ss{}mann and a Clifford algebra in $2\np$ generators.
  We show that $\rep \Q(\np)$ is ribbon equivalent to the symplectic fermion category $\catSF(\np)$ that was computed in~\cite{Runkel:2012cf} from conformal blocks of the corresponding logarithmic conformal field theory. The latter category in turn is conjecturally ribbon equivalent to representations of $\mathcal{V}_\mathrm{ev}$, the even part of the symplectic fermion vertex operator super algebra. 

Using the formalism developed in \cite{FGR} we compute the projective $\SLiiZ$-action on the centre of $\Q(\np)$ as obtained from Lyubashenko's general theory of mapping class group actions for factorisable finite ribbon categories.
This allows us to test a conjectural non-semisimple version of the modular Verlinde formula: we verify that the $\SLiiZ$-action computed from $\Q(\np)$ agrees projectively with that
	on
pseudo trace functions of $\mathcal{V}_\mathrm{ev}$.
\end{abstract}

\maketitle

\keyword{quasi-Hopf algebras, ribbon categories, vertex operator algebras, $\SLiiZ$ representations, Verlinde formula.}

\vspace{0.2em}

\setcounter{tocdepth}{1}
\tableofcontents
    
\thispagestyle{empty}
\newpage

\section{Introduction}
We introduce a new family of quasi-Hopf algebras $\Q(\np)$ whose definition is motivated by representation theory of vertex operator algebras, specifically by a family arising from logarithmic conformal field theories called symplectic fermions. Applying the general formalism developed in \cite{FGR} to $\Q(\np)$ allows us to test a conjecture 
in \cite{Gainutdinov:2016qhz} on a non-semisimple version of the Verlinde formula
for this class of examples. Before we give more details on the results in this paper, let us provide some background for this problem.

\medskip

In mathematical physics, 
	vertex operator algebras (VOAs)
provide an axiomatisation of the properties of holomorphic fields of a two-dimensional conformal quantum field theory. In mathematics, the original motivation for their study was to establish a link between algebra and modular forms and in particular, famously, between representations of the Monster group and the modular invariant $j$-function \cite{Borcherds:1983sq,FLM-book}.

The focus of this work is a different link between representation theory and modular forms provided by VOAs: the Verlinde formula, originally formulated in \cite{Verlinde:1988sn} in the context of rational conformal field theory. From the point of view of VOAs, the key statement is that two projective $SL(2,\mathbb{Z})$-actions agree, one algebraic and one from modular forms. 
As a corollary, one then obtains the Verlinde formula which is a statement about fusion rules expressed via 
	the action of the $S$-generator.	
To be more precise, let $V$ be a rational VOA, so that in particular the category $\rep V$ of $V$-modules is finitely semisimple. The two $SL(2,\mathbb{Z})$-actions are defined as follows (rationality of $V$ is crucial for both actions). 
\begin{itemize}
	\item Denote by $\mathrm{Irr}(V)$ a choice of representatives of the isomorphism classes of simple $V$-modules. Each $M \in \mathrm{Irr}(V)$ is $\mathbb{N}$-graded (with an overall shift). The resulting graded dimension converges and defines a function $\chi_M(\tau)$ on the upper half plane, called character of $M$. Remarkably, under modular transformations of $\tau$, these characters transform into each other \cite{Zhu1996}, giving a matrix-valued representation of $SL(2,\mathbb{Z})$. To achieve linear independence one has to use so-called torus one-point-functions instead of characters.
	The space $C_1(V)$ of torus-one-point functions has dimension $|\mathrm{Irr}(V)|$ \cite{Zhu1996}, resulting in an $SL(2,\mathbb{Z})$-representation in terms of matrices of size $|\mathrm{Irr}(V)|$.

	\item
	$\rep V$ is a semisimple 
	factorisable ribbon category (a modular fusion category)
	\cite{Huang:2004bn}. Here, factorisability is a non-degeneracy condition for the braiding.
The three-di\-mensional topological field theory constructed from a modular fusion category \cite{RT,tur} provides projective actions of surface mapping class groups. For $\rep V$ and in the case of the torus one obtains a projective $SL(2,\mathbb{Z})$-action on the vector space $E = \End_V(G)$ where  $G = \bigoplus_{M \in \mathrm{Irr}(V)} M$. 
\end{itemize}
Both $C_1(V)$ and $E$ have a canonical basis indexed by $\mathrm{Irr}(V)$, and the main step in Huang's proof \cite{Huang:2008} of the Verlinde formula for rational VOAs is to show that under the resulting linear isomorphism the two $SL(2,\mathbb{Z})$-actions agree projectively.

The multiplicities in the decomposition of the tensor product of two simple objects into a direct sum of simple objects are called fusion rules.
It is a relatively straightforward algebraic result that the $SL(2,\mathbb{Z})$-action on $E$ determines the fusion rules of the modular fusion category  \cite[Thm.\,4.5.2]{tur} (one actually only needs the action of the $S$-generator). One can summarise this deep result in the theory of VOAs as follows:
\begin{quote}
	\textsl{For $V$ rational, the modular properties of the torus-one-point functions of $V$ determine the fusion rules of $\rep V$.}
\end{quote}

In Zhu's theory of torus-one-point functions \cite{Zhu1996}, in order for $C_1(V)$ to be finite-dimensional and carry a $SL(2,\mathbb{Z})$-action, it is enough to require that $V$ satisfies a technical condition called $C_2$-cofiniteness, rather than imposing the stronger requirement of rationality. 
In this case $\mathrm{Irr}(V)$ is still finite, but $\rep V$ need not be semisimple.

It is now natural to ask if a relation similar to the Verlinde formula can be established in this more general setting. Indeed, this question has a long history in conformal field theory, see \cite{Flohr:2001zs,Fuchs:2003yu} and e.g.\ the works \cite{Fuchs:2006nx,Flohr:2007jy,Gaberdiel:2007jv,Gainutdinov:2007tc,Pearce:2009pg,Creutzig:2016fms}. 
For the specific version of the Verlinde formula we are concerned with in this paper, in addition to $C_2$-cofiniteness we require that $V$ is simple and self-dual,
	as well as non-negatively graded with one-dimensional component at grade zero. 
Then $\rep V$ is known to be braided monoidal \cite{HLZ}, and it is conjectured in \cite{Gainutdinov:2016qhz} that:
\begin{enumerate}
	\item 
	$\rep V$ is a factorisable finite ribbon category, see e.g.~\cite[Sec.\,\ref*{I-sec:univ-Hopf-finite}]{FGR} for conventions and references. 
	\item
	The projective $SL(2,\oZ)$-action on the endomorphisms of the identity functor of $\rep V$ computed from \cite{Lyubashenko:1995} agrees with that on the torus one-point functions $C_1(V)$ of $V$. 
\end{enumerate}
One can then use the algebraic result in \cite[Thm.\,3.9]{Gainutdinov:2016qhz} to compute the structure constants of the Grothendieck ring of $\rep V$ (which encodes the composition factors in the tensor product of two simple modules) in terms of the action of the $S$-generator. 
If $V$ is in addition rational, these conjectures turn into the theorems we quoted above, and one recovers the semisimple Verlinde formula.

The precise formulation of part~2 is that the two $SL(2,\oZ)$-actions (one projective, one not) on
$\End(Id_{\rep V})$ constructed in (a) and (b) below agree projectively:
\begin{itemize}
	\item[(a)] {\em VOA side:}
	Pick a projective generator $G$ of $\rep V$ and let $E = \End_V(G)$. Write $C(E) = \{ \, \varphi \in E^* \,|\, \varphi(fg) = \varphi(gf) \text{ for all } g,f \in E \,\}$ for the central forms on $E$. An element $\varphi \in C(E)$ defines a pseudo-trace function $\xi_G^\varphi$ \cite{Miyamoto:2002ar,Arike:2011ab}, providing a linear map 
	\be\label{eq:intro-xiG}
	\xi_G \colon\; C(E) \to C_1(V) \gc
	\ee
	see \cite[Sec.\,4]{Gainutdinov:2016qhz} and Section~\ref{sec:modprop-SF} for more details and references. 
	Conjecturally, $\xi_G$ is an isomorphism \cite[Conj.\,5.8]{Gainutdinov:2016qhz}. There is a linear (i.e.\ non-projective)  action of $SL(2,\oZ)$ on $C_1(V)$ \cite{Zhu1996} which by pullback along $\xi_G$ gives an $SL(2,\oZ)$-action on $C(E)$.

	Denote the modular $S$-transformation on torus one-point functions by $S_V : C_1(V) \to C_1(V)$ and
	write $\delta \in C(E)$ for the central form that gets mapped to the modular $S$-trans\-for\-ma\-ti\-on of the vacuum character, $\xi_G^\delta = S_V(\chi_V)$.
	Conjecturally, $\delta$ provides a non-degenerate central form on $E$, i.e.\ it turns $E$ into a symmetric Frobenius algebra \cite[Conj.\,5.10]{Gainutdinov:2016qhz}.
	One obtains an isomorphism between central forms and the centre of $E$ via $C(E) \to Z(E)$, $z \mapsto \delta(z \cdot -)$, see e.g.\ \cite[Lem.\,2.5]{Broue:2009}. Since $Z(E) \cong \End(Id_{\rep V})$, altogether we get an isomorphism $\hat\delta : \End(Id_{\rep V}) \xrightarrow{\sim} C(E)$. Via pullback, we finally obtain a 
	linear 
	$SL(2,\oZ)$-action on $\End(Id_{\rep V})$.
	\item[(b)] {\em Categorical side:}
	Abbreviate $\cat := \rep V$, a factorisable finite ribbon category by part~1 of the 
	above
	conjecture. 
	Let $\coend$ be the universal Hopf algebra in $\cat$, which can be defined as the coend $\coend = \int^{U \in \cat} U^* \otimes U$ \cite{Majid:1993,Lyubashenko:1995}, see also \cite[Sec.\,\ref*{I-sec:univ-hopf-braided}\,\&\,\ref*{I-sec:univ-hopf-via-coends}]{FGR}  for more details. One can equip $\coend$ with a normalised integral $\Lambda_\coend$, which is unique up to a sign.
	In \cite{Lyubashenko:1995} a projective $SL(2,\oZ)$-action 
	is constructed on the 
	Hom-space
	$\cat(\one,\coend)$ 
	using the non-degenerate Hopf pairing on $\coend$ and the ribbon twist to define the action of the 
	$S$- and $T$-generator, respectively. There is an isomorphism $\cat(\one,\coend) \cong \End(Id_{\cat})$, and one thus obtains a projective $SL(2,\oZ)$-action on $\End(Id_{\cat})$, see e.g.\ \cite[Sec.\,\ref*{I-sec:SL2Z}]{FGR}. 
	This projective action only depends on $\cat$ and the choice of sign for the integral $\Lambda_\coend$ \cite[Prop.\,\ref*{I-prop:ST-on-EndId}]{FGR}, and different sign choices result in actions that agree projectively.
\end{itemize}

%\medskip

In this paper we provide the first comparsion of the actions in (a) and (b) in a non-semisimple family of examples. The examples we consider are the so-called symplectic fermions. These are the simplest logarithmic conformal field theories \cite{Kausch:1995py}, consisting of the untwisted and $\oZ_2$-twisted sector of a free theory (namely purely odd free superbosons, see \cite{Runkel:2012cf}). 
They are parametrised by the natural number $\np$ (that counts the number of pairs of the fermionic fields). With regard to the actions in (a) and (b) we have:
\begin{itemize}
	\item[(a)] {\em VOA side for symplectic fermions:}
	Denote the even part of the vertex operator super algebra of $\np$ pairs of symplectic fermions by $\mathcal{V}_\mathrm{ev}$ \cite{Abe:2005}. Its central charge is 
	$c=-2\np$. In \cite{Davydov:2012xg,Runkel:2012cf} a factorisable finite ribbon category $\catSF(\np)$ was obtained via a conformal block calculation, and it is known that there is a faithful $\oC$-linear functor 
	$\catSF(\np) \to \rep \mathcal{V}_\mathrm{ev}$ (Proposition~\ref{prop:induction-evenpart-functor}, due to \cite{Abe:2005}), which, conjecturally, is a ribbon equivalence, see below. We review the category $\catSF(\np)$ in Section~\ref{sec:SFdef}. One can choose a projective generator $G$ in 
	$\catSF(\np)$ and use the above faithful functor to compute pseudo-trace functions, their behaviour under modular transformations, and the pullback to $C(E)$ and
	to $\End(Id_{\catSF})$ (\cite[Cor.\,6.9]{Gainutdinov:2016qhz}
	and Theorem~\ref{thm:SL2Z-from-pseudotrace}).

	We stress that the modular forms on which the $SL(2,\mathbb{Z})$-action is realised are given completely explicitly, 
	see Remark~\ref{rem:explicit-modular}.

	\item[(b)] {\em Categorical side for symplectic fermions:}
	General expressions for the universal Hopf algebra $\coend$ and its structure maps are known in the semisimple case, for representations of Hopf algebras \cite{Lyubashenko:1994tm,Kerler:1996}, and of quasi-Hopf algebras \cite[Sec.\,\ref*{I-sec:coend-repA}]{FGR}.
	In this paper, we first generalise the
	construction of~\cite{Gainutdinov:2015lja} to $\np>1$ 
	and give a ribbon quasi-Hopf algebra $\Q(\np)$ such that 
	$\catSF(\np) \simeq \rep\Q(\np)$
	as ribbon categories.
	Then we  apply the results in \cite[Sec.\,\ref*{I-sec:SL2Z-quasiHopf}]{FGR} to compute the projective 
	$SL(2,\oZ)$-action on $\End(Id_{\rep\Q})$, and hence on $\End(Id_{\catSF})$.
\end{itemize} 
The main result of this paper is (Theorem \ref{thm:compare-Q-SF}):

\begin{theorem}\label{thm:intro-main}
	Under the assumption that $\xi_G$
	from~\eqref{eq:intro-xiG} is an isomorphism,
	the $SL(2,\oZ)$-actions computed in part (a) and (b) for symplectic fermions agree projectively.
\end{theorem}

This does not yet prove the modular Verlinde formula for symplectic fermions. One still needs to show that $\catSF(\np) \simeq \rep \mathcal{V}_\mathrm{ev}$ as ribbon categories \cite[Conj.\,7.4]{Davydov:2016euo} and that the pseudo-trace functions of \cite{Arike:2011ab} indeed provide a bijection between $C(E)$ and $C_1(\mathcal{V}_\mathrm{ev})$ \cite[Conj.\,5.8]{Gainutdinov:2016qhz}.

%\medskip

After stating the main result, we now describe the contents of this paper in more detail. The present paper is the sequel to \cite{FGR}, where the construction of the projective action of the mapping class group of the torus for factorisable finite ribbon categories from \cite{Lyubashenko:1995} is made explicit for quasi-Hopf algebras. This generalises the treatment for Hopf algebras
in~\cite{Lyubashenko:1994ma}. 
Below, we specialise the general results in \cite{FGR} to the symplectic fermion example. 
To do so, we first need to realise the category $\catSF(\np)$ as 
the  representation category of a quasi-Hopf algebra.

From the construction of $\catSF$ it is natural to make it depend on another discrete parameter~$\beta$ which satisfies $\beta^4 = (-1)^\np$, that is, we have factorisable finite ribbon categories $\catSF(\np,\beta)$ which for 
the choice  $\beta = e^{- \rmi \pi \np / 4}$
are conjecturally ribbon-equivalent to $\rep \mathcal{V}_\mathrm{ev}$, see Sections~\ref{sec:SFdef} and~\ref{sec:modprop-SF}. 
The categories $\catSF(\np,\beta)$ all have integral Perron-Frobenius dimensions. By \cite[Prop.\,6.1.14]{EGNO-book}, such categories can be realised as representations of a quasi-Hopf algebra. 
The main technical contribution of this paper is to introduce a new family of factorisable ribbon quasi-Hopf algebras $\Q(\np,\beta)$ such that the following holds (Theorem~\ref{SF-repQ-rib-eq}):

\begin{theorem}
$\rep \Q(\np,\beta) \simeq \catSF(\np,\beta)$ as $\oC$-linear ribbon categories.
\end{theorem}

The proof of this theorem goes via two intermediate quasi-Hopf algebras, one in $\svect$ and one in $\vect$, and takes up the bulk of this paper.

The detailed definition of $\Q(\np,\beta)$ is given in Section~\ref{sec:Q-def}. 
We refer to~\cite[Sec.\,\ref*{I-subsec:conventions-ribbon-qHopf}]{FGR} for our conventions on quasi-Hopf algebras.
As an algebra over $\oC$, $\Q(\np,\beta)$ is generated by $\K$ and $\fpm_i$, $i=1,\ldots,\np$, subject to the relations
\be 
	\{\ff^\pm_i,\K\} = 0  ~~,\quad 
	\{\ff^+_i,\ff^-_j\}=\delta_{i,j}\,\tfrac12(\one-\K^2) ~~,\quad 
	\{\ff^\pm_i,\ff^\pm_j\}=0 ~~,\quad 
	\K^4=\one \ ,
\ee
where $\{x,y\}=xy+yx$ is the anticommutator. 
Define the central idempotents $\idem_0\coloneqq\tfrac12(\one+\K^2)$ and $\idem_1 \coloneqq \tfrac12(\one-\K^2)$. Then the coproduct is given by
\begin{align}
\Delta (\K) &= \K\tensor\K - (1+(-1)^\np)\idem_1\tensor\idem_1\cdot\K\tensor\K \gc 
\\ \nonumber
\Delta (\fpm_i) &= \fpm_i\tensor\one + (\idem_0 \pm \rmi\idem_1)\K\tensor\fpm_i  \ .
\end{align} 
The counit is defined by
$\eps(\K)=1$ and $\eps(\ff^{\pm}_i)=0$.
From the coproduct one sees that $\K$ is group-like only for odd $\np$. 
The coassociator lives in the Hopf-subalgebra generated by powers of $\K$ and takes the simple form
\be
\Phi = \one\tensor\one\tensor\one  
+ \idem_1\tensor\idem_1\tensor
\big\{ (\K^{\np}-\one)\idem_0+(\beta^2 (\rmi\K)^{\np}-\one)\idem_1 \big\}  \gp
\ee
This should be contrasted with the much more complicated coassociator given in \cite{Gainutdinov:2015lja}, where -- for a different quasi-Hopf algebra $\hat\Q(\np,\beta)$ and for $\np=1$ only -- the ribbon equivalence $\rep \hat\Q(1,\beta) \simeq \catSF(1,\beta)$
  was first established. In Appendix
~\ref{sec:twist-Q-Qhat}
 we show 
that $\hat\Q(\np,\beta)$ is twist-equivalent to $\Q(\np,\beta)$.

\begin{remark}
\mbox{}
\begin{enumerate}
\item We have not made explicit what the qualifier `symplectic' refers to. But let us at least point out that in \cite{Davydov:2012xg}, 
$\catSF(\np,\beta)$ is constructed from 
a $2\np$-dimensional symplectic vector space and from the constant $\beta$. It is furthermore shown in \cite[Rem.\,5.4]{Davydov:2012xg} that $\catSF(\np,\beta)$ carries a faithful action of $Sp(2\np,\mathbb{C})$ by ribbon equivalences.

\item
For $\np=1$, $\mathcal{V}_\mathrm{ev}$ agrees with $\mathcal{W}(p)$, the VOA of the $\mathcal{W}_p$-triplet models, in the case $p=2$,
	see~\cite{Kausch:1995py}. 
For $p \ge 2$ a factorisable ribbon quasi-Hopf algebra is constructed in \cite{CGRprep}
whose representation category is conjecturally ribbon-equivalent to that of~$\mathcal{W}(p)$.
The quasi-Hopf algebra in question is a modification of the restricted quantum group for $sl(2)$ which was used in the study of $\mathcal{W}_p$-models in \cite{[FGST]}.
For $p=2$, this quasi-Hopf algebra indeed agrees with our $\Q$ for $\np=1$.

\item
One can construct examples of non-semisimple theories from bosonic free-field theories by taking kernels of screening operators \cite{Fuchs:2003yu,AM,FT}. 
In particular $\mathcal{V}_\mathrm{ev}$ for $\np=1$ appears as an example there.
For general $\np$ it was shown in~\cite{Flandoli:2017} that $\mathcal{V}_\mathrm{ev}$ is the $W$-algebra constructed as the kernel of (short) screenings in a lattice VOA of type~$B_{\np}$, 
and with the lattice rescaled by $\sqrt{p}$ for $p=2$.
In this context, the symplectic symmetry $\mathfrak{sp}(2\np)$ comes from the action of long screenings.
A similar interplay between $B_\np$ and $C_\np$ root
lattices exists  on the quantum group side due to the Frobenius homomorphism, studied for this case in~\cite{Lentner:Frob}.

\item 
The categories $\rep \Q(\np,\beta) \simeq \catSF(\np,\beta)$ provide four of the 16 possible cases in the classification of factorisable finite tensor categories containing a Lagrangian subcategory of a certain form \cite{Gelaki:2017}.
\end{enumerate}
\end{remark}

	In Section~\ref{sec:SFdef} we review the ribbon category $\catSF(\np,\beta)$.
In Sections \ref{sec:RqHQ} and~\ref{coendQ}
we study properties of $\Q(\np,\beta)$ and the coend $\coend$ in $\rep\Q(\np,\beta)$. 
For $\np$ even and $\beta^2=1$, $\Q(\np,\beta)$ is a Hopf algebra, not just a quasi-Hopf algebra. If in addition $\beta=1$, $\Q(\np,\beta)$ is 
	the
Drinfeld double 
of a generalisation of the Sweedler's Hopf algebra,
 see Proposition~\ref{DH:main-prop}.
 The structure maps on the coend $\coend = \Q(\np,\beta)^*$ are given in Proposition~\ref{prop:coend-Q}.

In Section~\ref{sec:sl2Z-Q-SF} we compute the $SL(2,\oZ)$-action on
the centre
 $Z(\Q(\np,\beta))$, with the result stated in Theorem~\ref{thm:ST-from-Q}, and we also describe the decomposition of this representation of $SL(2,\oZ)$ in Remark~\ref{rem:SL2Z-decomp}.
 In Section~\ref{sec:SL2Z-compare} we carry out the comparison to the modular properties of pseudo-trace functions as already described above. 
 The three appendices take up almost half of the paper and contain the more technical calculations and proofs.

\subsection*{Acknowledgements:} 
We thank Christian Blanchet, Nathan Geer, Simon Lentner, Ehud Meir, and Hubert Saleur for helpful discussions. 
We also grateful to the anonymous referee for the important comments that improved parts of the paper.
The work of AMG was  supported by  CNRS. AMG also thanks Mathematics Department of 
Hamburg University for kind hospitality during the period 2016-2017 when the main results were obtained.

\subsection*{Convention:} We will make frequent references to the prequel \cite{FGR} of the present paper, and we will abbreviate such references by ``I:\dots''. E.g.\ we write ``Section I:\ref*{I-sec:univ-hopf-braided}'' instead of \cite[Sec.\,2]{FGR}.

\section{The ribbon category \texorpdfstring{$\catSF$}{SF}}\label{sec:SFdef}

In this section we review the definition of the categories $\catSF(\np,\beta)$ introduced in \cite{Davydov:2012xg,Runkel:2012cf}, see also \cite{Davydov:2013ty} for a summary. These are finite ribbon categories (see Sections I:\ref*{I-subsec:conventions-mono-cats} and I:\ref*{I-subsec:finite-tens-cats} for notation and conventions) which depend on two parameters:
\begin{equation}\label{eq:beta-param}
	\np\in\oN = \{1,2,\dots\} \qquad\text{and}\qquad \beta\in\oC \quad\text{such that}\quad \beta^4=(-1)^\np \gp 
\end{equation} 
We will refer to $\catSF(\np,\beta)$ as the {\em category of $\np$ pairs of symplectic fermions}, and we will abbreviate $\catSF := \catSF(\np,\beta)$.

\subsection{ \texorpdfstring{$\catSF$}{SF} as an abelian category}\label{sec:SF-abelian}

Let $\algGr$ be the $2^{2\np}$-dimensional Gra\ss{}mann algebra over $\oC$ generated by $a_{i}, b_{i}$ with defining relations
\begin{align}\label{eq:aibi-relation}
		\{a_{i} , a_{j}\} = \{b_{i} , b_{j}\} = 
	\{a_{i} , b_{j}\} = 0 \gc \qquad i,j=1,\ldots \gc\np
	\gp
\end{align}
By giving $c \in \{a_1,b_1, \ldots ,a_\np,b_\np\}$ odd parity and defining $\Delta (c)= c\tensor \one + \one\tensor c$, $\eps(c)=0$ and $S(c)=-c$ 
we get a Hopf algebra in $\svect$.
Let $\repsv\algGr$ be the category of finite dimensional super-vector spaces with a $\algGr$-action. 
The category $\catSF = \catSF(\np,\beta)$ is defined as
\begin{equation}\label{eq:catSF-dec}
\catSF := \catSF_0 \oplus  \catSF_1 \quad \text{where} \quad  \catSF_0\coloneqq \repsv\algGr \gc \  \catSF_1 \coloneqq  \svect .
\end{equation}
Note that $\catSF_0$ is not semi-simple. 
A choice of representatives for the isomorphism classes of simple objects in $\catSF$ is
\begin{align}\label{eq:SF-simple}
	\one &:= \oC^{1|0} \in \catSF_0
	\ , 
	&
	T &:= \oC^{1|0} \in \catSF_1
	\ ,
	\\ \nonumber
	\Pi\one &:= \oC^{0|1} \in \catSF_0
	\ , 
	&
	\Pi T &:= \oC^{0|1} \in \catSF_1 \ ,
\end{align}
where $\Pi$ denotes the parity exchange endofunctor on $\svect$, and the $\algGr$-action on $\one$ and $\Pi\one$ is trivial.
The projective covers of the simple objects are
\begin{align}\label{eq:SF-proj}
	P_{\one} &= \algGr
	\ , 
	&
	P_T &=T
	\ ,
	\\ \nonumber
	P_{\Pi\one} &=\Pi\algGr
	\ , 
	&
	P_{\Pi T} &=\Pi T \ .
\end{align}

We will now endow $\catSF$ step by step with the structure of a ribbon category.

\subsection{Tensor product}

Given two objects $X,Y\in\catSF$ we define a tensor product functor $\ast\colon \catSF\times\catSF\to\catSF$ as follows:
\begin{equation}\label{eq:*-tensor}
X \ast Y ~=~
\left\{\rule{0pt}{2.8em}\right.
\hspace{-.5em}\raisebox{.7em}{
\begin{tabular}{ccll}
   $X$ & $Y$ & $X \ast Y$ &
\\
$\catSF_0$ & $\catSF_0$ & $X \otimes_{\repsv\algGr} Y$ & $\in~\catSF_0$
\\
 $\catSF_0$ & $\catSF_1$ & $F(X) \svtensor Y$ & $\in~\catSF_1$
\\
 $\catSF_1$ & $\catSF_0$ & $X \svtensor F(Y)$ & $\in~\catSF_1$
\\
$\catSF_1$ & $\catSF_1$ & $\algGr \svtensor X \svtensor Y$ & $\in~\catSF_0$
\end{tabular}}
\end{equation}
Here, $F : \repsv\algGr \to \svect$ stands for the forgetful functor. 
By $X \otimes_{\repsv\algGr} Y$ we mean the tensor product in $\svect$ with $\algGr$-action via the coproduct. In more detail, denote by 
\begin{equation}\label{sflip}
\sflip_{X,Y} : X \svtensor Y \longrightarrow Y \svtensor X
\quad , \quad
\sflip_{X,Y}(x \tensor y) = (-1)^{|x||y|} y \tensor x \gc
\end{equation}
the symmetric braiding in $\svect$, where $x,y$ are homogeneous elements. 
Then the action of an element $g\in\algGr$ on $x\tensor y\in X\svtensor Y$ is then given by
\begin{align}
	g.(x\otimes y) &= 
	(\rho^X \otimes \rho^Y) \circ (\id \otimes \sflip_{\algGr,X} \otimes \id) \circ (\Delta(g) \otimes x \otimes y)
\\ \nonumber &=
	\sum_{(g)} (-1)^{|g''||x|} (g'.x) \otimes (g''.y)
\quad \text{where} \quad \Delta(g) = \sum_{(g)} g' \otimes g'' \ , 
\end{align}
where $\rho^X$ and $\rho^Y$ give the action of $\algGr$ on $X$ and $Y$.

On morphism, we define the tensor product in all cases beside the last one to be $f\ast g=f\tensor g$. If $f,g\in\catSF_1$ we set $f *g = \id_\algGr \tensor f \tensor g$. 

For the remainder of this section we will drop the subscripts from the tensor products $\svtensor$ and $\otimes_{\repsv\algGr}$ for brevity.

\subsection{Associator}\label{sec:assoc-SF}

Both, the associator and the braiding depend on $\beta$ from \eqref{eq:beta-param}
and on a copairing $C$ on $\algGr$
given by \cite[Eqn.\,(5.5)]{Davydov:2012xg}
\begin{equation} \label{def:copair}
	C := \sum_{i=1}^\np b_i \tensor a_i - a_i \tensor b_i
     \quad \in~ \algGr \tensor \algGr \gp
\end{equation}

The associator is a natural family of isomorphisms
\begin{equation*}
\alpha^{\catSF}_{X,Y,Z} :\quad X*(Y*Z) \to (X*Y)*Z \gc
 \end{equation*}
which is defined sector by sector by the following eight expressions (see \cite[Thm.\,6.2]{Runkel:2012cf} and \cite[Sect.\,5.2\,\&\,Thm.\,2.5]{Davydov:2012xg}):
\begin{align*}
X\,Y\,Z \qquad & X \ast (Y \ast Z) && (X \ast Y) \ast Z && \alpha^{\catSF}_{X,Y,Z} : X \ast (Y \ast Z) \to (X \ast Y) \ast Z
\nonumber\\[.3em]
0\,~0\,~0\, \qquad &
\underline{X} \tensor \underline{Y} \tensor \underline{Z} &&
\underline{X} \tensor \underline{Y} \tensor \underline{Z} &&
\id_{X \otimes Y \otimes Z}
\nonumber\\[.3em]
0\,~0\,~1\, \qquad &
X \tensor Y \tensor Z &&
X \tensor Y \tensor Z &&
\id_{X \otimes Y \otimes Z}
\nonumber\\[.3em]
0\,~1\,~0\, \qquad &
X \tensor Y \tensor Z &&
X \tensor Y \tensor Z &&
\exp\!\big(C^{(13)} \big)
\nonumber\\[.3em]
1\,~0\,~0\, \qquad &
X \tensor Y \tensor Z &&
X \tensor Y \tensor Z &&
 \id_{X \otimes Y \otimes Z}
\nonumber\\[.3em]
0\,~1\,~1\, \qquad &
\underline{X} \tensor \underline{\algGr}\tensor Y \tensor Z &&
\underline{\algGr} \tensor X\tensor Y \tensor Z &&
\Big[\big\{ \id_{\algGr} \otimes (\rho^X \circ (S \otimes \id_X))  \big\}
\nonumber\\[-.1em]
&&&&&\qquad \circ \big\{ 
	\Delta
\otimes \id_X \big\} \circ \sflip_{X,\algGr} \Big] \otimes \id_{Y \otimes Z}
\nonumber\\[.3em]
1\,~0\,~1\, \qquad &
\underline{\algGr} \tensor X\tensor Y \tensor Z &&
\underline{\algGr} \tensor X\tensor Y \tensor Z &&
\exp\!\big( C^{(13)} \big)
\nonumber\\[.3em]
1\,~1\,~0\, \qquad &
\underline{\algGr} \tensor X\tensor Y \tensor Z &&
\underline{\algGr} \tensor X \tensor Y \tensor \underline{Z} &&
\big\{ \id_{\algGr \otimes X \otimes Y} \otimes \rho^Z \big\}
  \circ
  \big\{ \id_{\algGr} \otimes \sflip_{\algGr,X \otimes Y} \otimes \id_Z  \big\}
\nonumber\\[-.1em]
&&&&&\qquad
  \circ
  \big\{ \Delta \otimes \id_{X \otimes Y \otimes Z} \big\}
\nonumber\\[.3em]
1\,~1\,~1\, \qquad &
X \tensor \algGr \tensor Y \tensor Z &&
\algGr \tensor X \tensor Y \tensor Z &&
\big\{ \phi \otimes \id_{X \otimes Y \otimes Z} \big\} \circ \big\{\sflip_{X,\algGr} \otimes \id_{Y \otimes Z}\big\}
\end{align*}
The underlines mark on which tensor factors $\algGr$ acts (hence they appear only when the triple tensor product lies in $\catSF_0$, i.e.\ when an even number of sectors `1' appear).  
With $C^{(13)}$ denote $C_1\tensor\one\tensor C_2$ where $C=\sum_{(C)}C_1\tensor C_2$. Hence, the action of $C^{(13)}$ is given by
\begin{equation}
	C^{(13)}(x \tensor y \tensor z) = (-1)^{(|x|+|y|)}\sum_{i=1}^\np b_i .x  \tensor y \tensor a_i . z - a_i .x  \tensor y \tensor b_i .z \gc
\end{equation}
for homogeneous $x,y,z$. The linear map $\phi : \algGr \to \algGr$ is given by
\begin{equation}\label{eq:SF-phi-def}
	\phi = \left(\id\tensor (\Lambda^\mathrm{co}_\algGr \circ\mu_\algGr)\right) \circ \left( \exp(-C) \tensor \id\right)
\gc
\end{equation}
where $\mu_\algGr$ is the multiplication in $\algGr$ and $\Lambda^\mathrm{co}_\algGr \in \algGr ^\ast$ is a specific cointegral for $\algGr$ \cite[Eqn.\,(5.16)]{Davydov:2012xg}. Namely, $\Lambda^\mathrm{co}_\algGr$ is non-vanishing only in the top degree of $\algGr$, and there it takes the value
\begin{equation}\label{eq:SF-coint-def}
	\Lambda^\mathrm{co}_\algGr(a_1b_1\cdots a_\np b_\np) = \beta^{-2} \ .
\end{equation}

\subsection{Braiding}\label{sec:SF-braiding} 
For $X \in \svect$ denote by
\begin{equation}\label{eq:svect-parityinv-def}
	\omega_X : X \xrightarrow{~\sim~} X
\quad,\quad
	\omega_X(x) = (-1)^{|x|} \, x \gc
\end{equation}
the parity involution on $X$. The family $X \mapsto \omega_X$ is a natural monoidal isomorphism of the identity functor on $\svect$.

The braiding on $\catSF$ is given -- again sector by sector -- by the following family of natural isomorphisms $c_{X,Y}$ (see \cite[Thm.\,6.4]{Runkel:2012cf} and \cite[Sect.\,5.2\,\&\,Thm.\,2.8]{Davydov:2012xg}):
\begin{equation}\label{eq:SF-braiding-sectors}
\raisebox{.7em}{
\begin{tabular}{ccl}
 $X$ & $Y$ & $c_{X,Y} ~:~ X \ast Y \to Y \ast X$
\\[.3em]
 $0$ & $0$ & $ \sflip_{X,Y} \circ \exp(- C) $
\\[.3em]
 $0$ & $1$ & $\sflip_{X,Y} \circ \big\{ \kappa \tensor \id_Y \big\}$
\\[.3em]
 $1$ & $0$ & $\sflip_{X,Y}  \circ \big\{ \id_X \tensor \kappa \big\} \circ \{ \id_X \otimes \omega_Y \}$
\\[.3em]
 $1$ & $1$ & $\beta \cdot
  \big(\id_{\algGr} \otimes \sflip_{X,Y}\big)\circ \big\{ R_{\kappa^{-1}} \otimes \id_X \tensor \omega_Y \big\}$
\end{tabular}}
\end{equation}
Here, $\kappa :=\exp\!\big(\tfrac12\hat C\big)$ where $\hat C = \mu_\algGr(C) =  \sum_{i=1}^\np -2 a_i b_i$ and $R_a$ is the right multiplication with $a\in\algGr$:
\begin{align} 
R_a\colon \algGr\to \algGr , \qquad R_a=\mu^{\algGr} \circ (id_{\algGr}\otimes a) \gp
\end{align}
If $a$ is parity-even, this is indeed a morphism in $\svect$, and, since it commutes with the $\algGr$-action, it is also a morphism in $\catSF$. 

Recall the definition of a factorisable braided tensor category from Definition~I:\ref*{I-def:fact-cat}. 
A key property of $\catSF$ is:

\begin{proposition}\label{prop:SF-fact}
	The finite braided tensor category $\catSF$ is factorisable.  
\end{proposition}

\begin{proof}
	In \cite[Proposition 5.3]{Davydov:2012xg} it is shown that the full subcategory of transparent objects in $\catSF$ is $\vect$. 
	Thus $\catSF$ fulfils one of equivalent factorisability conditions in \cite{Shimizu:2016} (see Section~I:\ref*{I-sec:factorisable-cat} for a summary in our notation). 
\end{proof}

\subsection{Left duality}\label{eq:SF-dual}

Our conventions for duality morphisms are given in Section I:\ref*{I-subsec:conventions-mono-cats}. By convention, for $X \in \svect$ we choose the left-dual $X^*$ and the associated duality maps $\ev_X^{\svect} $, $\coev_X^{\svect}$ to be the same as for the underlying vector space.

For $X \in \catSF_i$, $i=0,1$, we define the left dual object $X^* \in \catSF_i$ to be the dual in $\svect$ as super-vector space. For $X \in \catSF_0$, $X^*$ is furthermore equipped with the $\algGr$-action 
\be \rho_{X^\ast}\colon \algGr\tensor X^\ast \to X^\ast \gcg g\tensor \varphi \mapsto 
(x\mapsto (-1)^{|g|(|\varphi|+1)} \varphi(g.x)) 
\gp 
\ee 
The evaluation and coevaluation maps in $\catSF$ are induced from those in $\svect$ as follows \cite[Sec.\,3.6]{Davydov:2012xg} (note that for $X \in \catSF_1$ we have $X^\ast \ast X = \algGr \tensor X^\ast \tensor X$): 
\begin{equation}
\raisebox{.7em}{
\begin{tabular}{ccc}
 \phantom{.}$X$\phantom{.} & \phantom{.}$\ev_X^{\catSF}: X^\ast \ast X \to \tid$\phantom{.} & \phantom{.}$\coev_X^{\catSF}: \tid \to X \ast X^\ast$\phantom{.}
\\[.3em]
 $0$ & $ \ev_X^{\svect} $ & $\coev_X^{\svect} $
\\[.3em]
 $1$ & $\eps_\algGr \tensor \ev_X^{\svect}$ & $\Lambda_\algGr \tensor \coev_X^{\svect}$
\end{tabular}} 
\end{equation}
Here, $\eps_\algGr$ is the counit for $\algGr$, and $\Lambda_\algGr =\beta^2 a_1b_1\cdots a_\np b_\np$ is the integral for $\algGr$ normalised with respect to the cointegral in \eqref{eq:SF-coint-def} such that $\Lambda^\mathrm{co}_\algGr( \Lambda_\algGr ) = 1$.

\subsection{Ribbon twist}\label{eq:SF-twist}
The ribbon twist isomorphisms $\theta_X$ are given in \cite[Prop.\,4.17]{Davydov:2012xg}:
\begin{equation}
\raisebox{.7em}{
\begin{tabular}{ccl}
 $X$ & $\theta_X : X \to X$
\\[.3em]
 $0$ & $ \exp(-\hat C) $
\\[.3em]
 $1$ & $\beta^{-1} \cdot \omega_X$
\end{tabular}}
\end{equation}
The twist isomorphisms satisfy
\begin{equation}
	\theta_{X * Y} = (\theta_X \tensor \theta_Y) \circ c_{Y,X} \circ c_{X,Y} \quad , \quad \theta_{X^\ast} = \theta_{X}^\ast \gp
\end{equation}
For later reference we note that on the four simple objects in \eqref{eq:SF-simple}, the twist is given by
\be\label{eq:twistSFsimple}
	\theta_{\one} = \id_{\one} 
	~~,\quad
	\theta_{\Pi\one} = \id_{\Pi\one}
	~~,\quad
	\theta_{T} = \beta^{-1} \,\id_{T}
	~~,\quad
	\theta_{\Pi T} = -\beta^{-1} \,\id_{\Pi T}
	\ .
\ee

We summarise Proposition~\ref{prop:SF-fact} and the ribbon structure on $\catSF$ reviewed in this section by the following theorem.

\begin{theorem}\label{SF-fact}
$\catSF(\np,\beta)$ is a factorisable finite ribbon category for all choices of $\np$ and $\beta$ as in \eqref{eq:beta-param}. 
\end{theorem}

\section{The ribbon quasi-Hopf algebra \texorpdfstring{$\Q$}{Q}}
\label{sec:RqHQ}

In this section we define the central object of this paper, the family of ribbon quasi-Hopf algebras $\Q(\np,\beta)$, where the parameters $\np$ and $\beta$ are constrained as in \eqref{eq:beta-param}. The main result of this section is a ribbon equivalence between 
$\repQ(\np,\beta)$ and $\catSF(\np,\beta)$.

Our conventions on quasi-Hopf algebras, universal R-matrices, etc., are given in Section~I:\ref*{I-sec:ribbon-qHopf}. In this section we will abbreviate $\Q := \Q(\np,\beta)$.

\subsection{Definition of \texorpdfstring{$\Q$}{Q}}	\label{sec:Q-def}

In defining the ribbon quasi-Hopf algebra $\Q$, we will first list all of its data -- starting with the product and ending with the ribbon element -- and only afterwards we will prove that it satisfies the necessary properties.

We start by giving $\Q$ as an associative unital algebra over $\oC$ via generators and relations. The generators are
\be
	\K \quad \text{and} \quad \fpm_i,\qquad i=1,\ldots,\np \ .
\ee
Define the elements
\begin{align} \label{def:idem Q}
\idem_0\coloneqq\tfrac12(\one+\K^2) \qcq \idem_1 \coloneqq \tfrac12(\one-\K^2)  \ .
\end{align}
Using these, the defining relations of $\Q$ can be written as, for $i,j = 1,\dots,\np$,
\be \label{def:Q}
	\{\ff^\pm_i,\K\} = 0  ~~,\quad 
	\{\ff^+_i,\ff^-_j\}=\delta_{i,j}\idem_1 ~~,\quad 
	\{\ff^\pm_i,\ff^\pm_j\}=0 ~~,\quad 
	\K^4=\one \ ,
\ee
where $\{x,y\}=xy+yx$ is the anticommutator. With these relations, $\idem_0$ and $\idem_1$ become central idempotents in $\Q$. The corresponding decomposition of $\Q$ into ideals is
\begin{equation}\label{Q-decomp}
\Q =  \Q_0\oplus \Q_1 \qquad \text{where}\qquad \Q_i\coloneqq\idem_i\Q \gp
\end{equation}
It is easy to check that restricted to each ideal, the algebra structure of $\Q$ becomes
\begin{equation}\label{Q0Q1}
\Q_0 =  \algGr_{2\np}\rtimes \oC\oZ_2 \gc \qquad  \Q_1 = \algCl_{2\np}\rtimes \oC\oZ_2 \ .
\end{equation}
Here,  $\algGr_{2\np}$ is the Gra\ss{}mann algebra of the $2\np$ generators $\ff^{\pm}_i\idem_0$ while  $\oC\oZ_2$ stands for the group algebra of $\oZ_2$ generated by $\K\idem_0$. Similarly, $\algCl_{2\np}$ is  the Clifford algebra generated by $\ff^{\pm}_i\idem_1$, while the generator of $\oZ_2$ is $\rmi \K\idem_1$. In particular, we see that
\be 
\dim_\oC\Q=2^{2\np+2} \ 
\ee
	and that a basis of $\Q$ is
\be\label{eq:Qbasis}
\Q = \mathrm{span}_\oC \left\{(\ff^{+}_{1})^{n_1^+}(\ff^{-}_{1})^{n_1^-}\cdots (\ff^{+}_{N})^{n_N^+}(\ff^{-}_{N})^{n_N^-} \K^n\,|\, 
n_i^{\pm} \in \{0,1\},\,
n\in \mathbb{Z}_4 \right\}\ .
\ee

Next we define the quasi-bialgebra structure of $\Q$. 
The coproduct  on generators is
\begin{align}\label{def:delta}
\Delta (\K) &= \K\tensor\K - (1+(-1)^\np)\idem_1\tensor\idem_1\cdot\K\tensor\K \gc 
\\ \nonumber
\Delta (\fpm_i) &= \fpm_i\tensor\one + \omega_\pm\tensor\fpm_i  \ ,
& \omega_\pm\coloneqq (\idem_0 \pm \rmi\idem_1)\K \gp
\end{align} 
We will show in Lemma~\ref{lem:D-alg-map} below that $\Delta$ is 
	well-defined and
an algebra map. The counit is 
\begin{equation}
\eps(\K)=1\gc\qquad
\eps(\ff^{\pm}_i)=0 \gp
\end{equation}
We remark that $\K$ itself is group-like only for odd $\np$. However, one quickly verifies that
\be
	\K^2 ~,~
	\K^\np
	\quad
	\text{are group-like for all $\np$} \ .
\ee

The co-associator and its inverse are
\be \label{Phi+Phi-inv}
\Phi^{\pm 1} = \one\tensor\one\tensor\one  
+ \idem_1\tensor\idem_1\tensor
\big\{ (\K^{\np}-\one)\idem_0+(\beta^2 (\pm \rmi\K)^{\np}-\one)\idem_1 \big\}  \gp
\ee

For the quasi-Hopf algebra structure on $\Q$ we still need to specify the antipode $S$ and the evaluation and coevaluation elements $\Salpha$ and $\Sbeta$. They are: 
\begin{align} \label{eq:Q-antipode-def}
	S(\K) &= \K^{(-1)^\np} = (\idem_0+(-1)^\np\idem_1)\K \ ,
	& \Salpha&=\one \ ,
	\\ \nonumber
	 S(\fpm_k) &= \fpm_k (\idem_0 \pm (-1)^\np \, \rmi \idem_1)\K \ ,
   & \Sbeta &=\idem_0+ \beta^2(\rmi\K)^\np \idem_1 \ .
\end{align} 

\begin{remark} \label{rem:Q-Hopf-alg}
  For even $\np$ and $\beta^2=1$ we have $\Salpha=\Sbeta=\one$ and $\Phi=\one\tensor\one\tensor\one$, so that (pending the verification that the above data verifies the axioms) in these cases $\Q$ is a Hopf algebra.
For odd $\np$ on the other hand, the coproduct fails to be coassociative.   
For example, 
\begin{align}
(\Delta\tensor\id)\circ\Delta(\fm_i) 
=~& \fm_i\tensor\one\tensor\one + \omega_-\tensor\fm_i\tensor\one  \\ \nonumber
&~+ \K\tensor\K\cdot \left( \idem_0\tensor\idem_0 -\rmi \idem_0\tensor\idem_1 -\rmi \idem_1\tensor\idem_0 -(-1)^\np\idem_1\tensor\idem_1\right)\tensor\fm_i \gc \\\nonumber
(\id\tensor\Delta)\circ\Delta(\fm_i) 
=~& \fm_i\tensor\one\tensor\one + \omega_-\tensor\fm_i\tensor\one  \\ \nonumber
&~+ \K\tensor\K\cdot \left( \idem_0\tensor\idem_0 -\rmi \idem_0\tensor\idem_1 -\rmi \idem_1\tensor\idem_0 -\idem_1\tensor\idem_1\right)\tensor\fm_i \gp
\end{align}
\end{remark}

Next we introduce a quasi-triangular structure on $\Q$.
Define the Cartan factor $\rho_{n,m}$ as
\begin{equation}\label{eq:car-fac}
\rho_{n,m} = \half \sum_{i,j=0}^1 (-1)^{ij}\rmi^{-in+jm} \K^i\tensor\K^j \gc
\qquad n,m \in \{0,1\}\gp
\end{equation}
The universal $R$-matrix and its inverse are
defined as
\begin{align} 
 R  &=  \Big( \sum_{n,m \in \{0,1\} } \beta^{nm} \rho_{n,m} \idem_n\tensor\idem_m \Big) \cdot \prod_{k=1}^{\np}(\one\tensor\one - 2 \ff^-_k \omega_-\tensor \ff^+_k) \gc 
 \label{R+Riv} 
 \\ \nonumber
 R^{-1} &= \prod_{k=1}^{\np}(\one\tensor\one+2\ff^-_k \omega_-\tensor \ff^+_k)\cdot
 \Big( \sum_{n,m \in \{0,1\} } \beta^{-nm} \rho_{n,m}
\idem_n\tensor\idem_m \Big) \ .
\end{align}

Finally, the ribbon element of $\Q$ and its inverse are
\begin{align}\label{ribbon+ribinv}
\ribbon &=  (\idem_0-\beta\rmi\K\idem_1)\cdot\prod_{k=1}^\np(\one-2\fp_k\fm_k) \gc \\ \label{eq:ribinv}
\ribbon^{-1} &= (\idem_0-\beta^{-1}\rmi\K\idem_1)\cdot\prod_{k=1}^\np(\one+2\fp_k\fm_k\K^2) \gp
\end{align}  

\begin{proposition} \label{prop-Q-qHopf}
	The data $(\Q,\cdot,\one,\Delta,\eps,
	\Phi,
S,\Salpha,\Sbeta, R, \ribbon)$ defines a ribbon quasi-Hopf algebra.
\end{proposition}

The proof of this proposition will be given after some preparation. 
The following lemma is straightforward check on the generators of $\Q$.

\begin{lemma} \label{lem:D-alg-map}
The map $\Delta$ defined in \eqref{def:delta} is an algebra map.
\end{lemma}

Since $\Delta$ is an algebra map, we can use it to define a tensor product functor 
\be\label{eq:tensor-Q}
\tensor\colon\; \repQ\times\repQ\to\repQ\ .
\ee
The central idempotents $\idem_0$ and $\idem_1$ behave under the coproduct as
\be
\Delta(\idem_0)=\idem_0\tensor\idem_0+\idem_1\tensor\idem_1
\quad , \quad 
\Delta(\idem_1)=\idem_0\tensor\idem_1+\idem_1\tensor\idem_0 \ ,
\ee
so that the tensor product~\eqref{eq:tensor-Q} respects the $\oZ_2$-grading $\rep\Q = \rep\Q_0 \oplus \rep\Q_1$.

Denote by
\begin{equation}\label{vecflip}
\flip_{X,Y}\colon\; X \tensor Y \longrightarrow Y \tensor X
\quad , \quad
\flip_{X,Y}(x \tensor y) =  y \tensor x \gc
\end{equation}
the symmetric braiding in $\vect$.
By acting with $\Phi$, $R$
 nd $\ribbon^{-1}$ we get 
families of isomorphisms in $\repQ$
\begin{align} 
\assocQ_{M,N,K} &\colon \quad M\tensor(N\tensor K) \to (M\tensor N)\tensor K \gcg  \quad	m\tensor n\tensor k \mapsto \Phi. (m\tensor n\tensor k) \gc\label{assocQ} \\ 
  c^{\repQ}_{M,N} &\colon \quad M\tensor N \to N\tensor M \gcg \qquad m\tensor n \mapsto 
  \flip_{M,N}
  (R.(m\tensor n)) \gc  \label{cQ} \\ 
	\theta^{\repQ}_M &\colon  \quad M \to M  \gcg \qquad m \mapsto \ribbon^{-1}.m  \gp  \label{thetaQ}
\end{align} 
We will show in Lemma~\ref{lem-trans-QSF} below that the linear maps~\eqref{assocQ}--\eqref{thetaQ} are indeed morphisms in $\repQ$, that they are natural and  that $\assocQ$ gives an associator, $c^{\repQ}$ a braiding and $\theta^{\repQ}$ a twist in $\repQ$.

\begin{definition} \label{def:transport}
Let $\cat$ be a monoidal category and $\catD$ a category with a tensor product functor $\tensor_\catD\colon \catD\times\catD\to\catD$ and tensor unit $\one_\catD$.
A functor $\fun\colon \cat\to\catD$ is called \emph{multiplicative}
if there exists a family of natural isomorphisms 
$\Theta_{U,V}\colon \fun(U\tensor_{\cat} V)\to \fun(U)\tensor_{\catD} \fun(V)$ and an isomorphism $\Theta_{\one}\colon \fun(\one)\to\one$. 
\end{definition}

If $\catD$ in the above definition is a monoidal category as well (i.e.\ equipped with associator and unit isomorphisms), a multiplicative functor is also called a {\em quasi-tensor functor}, see e.g.\ \cite[Def.\,4.2.5]{EGNO-book},
i.e.\ we dropped the coherence conditions with the associators from the definition of the tensor functor.

Let $\fun\colon \cat\to\catD$ be a multiplicative equivalence.
 We now describe how $\fun$ can be used to transport structure from $\cat$ to $\catD$.
\begin{itemize}
\item
By \emph{transporting the associator from $\cat$ to $\catD$ along $\fun$} we mean  seeking a natural family $\alpha_{X,Y,Z}^\catD$, for $X,Y,Z \in \catD$, 
such that for all $U,V,W \in \cat$, the diagram
\begin{equation}\label{eq:transport-assoc-diag}
\xymatrix@R=32pt@C=72pt{
\fun \bigl(U\tensC (V\tensC W)\bigr)\ar[d]_{\Theta_{U,V\tensC W}}\ar[r]^{\fun(\assC_{U,V,W})}
&\fun \bigl((U\tensC V)\tensC W\bigr)\ar[d]^{\Theta_{U\tensC V,W}}\\
\fun (U)\tensD \fun(V\tensC W)\ar[d]^{\id\tensor\Theta_{V,W}}
&\fun (U\tensC V)\tensD \fun(W)\ar[d]^{\Theta_{U,V}\tensor\id}\\
\fun (U)\tensD\bigl(\fun(V)\tensD\fun(W)\bigr)\ar[r]^{\assD_{\fun(U),\fun(V),\fun(W)}}
&\bigl(\fun (U)\tensD \fun(V)\bigr)\tensD\fun(W)
}
\end{equation}
commutes. 
Since $\fun$ is an equivalence such an $\alpha^\catD$ exists. Moreover, because $\alpha^\catD$ is natural and $\fun$ is essentially surjective, $\alpha^\catD$ is unique. 
By construction $\alpha^\catD$ satisfies the pentagon condition. 
The tensor unit structure (i.e.\ left and right unit morphisms) can be transported in the same way which makes $\catD$ a monoidal category. 
\item Suppose $\cat$ is in addition braided.
\emph{Transporting the braiding} from $\cat$ to $\catD$ means determining a family $c^{\catD}_{X,Y}$ such that the following diagram commutes for all $U,V\in\cat$:
\begin{equation}\label{eq:transport-braiding-via-functorequiv}
\xymatrix@R=22pt@C=42pt{
&\fun(U\tensor_{\cat} V)\ar[r]^{\fun(c^{\cat}_{U,V})}\ar[d]^{\Theta_{U,V}}&\fun(V\tensor_{\cat} U)\ar[d]^{\Theta_{V,U}}&\\
&\fun(U)\tensor_{\catD}\fun(V)\ar[r]^{c^{\catD}_{\fun(U),\fun(V)}}&\fun(V)\tensor_{\catD}\fun(U)&}  
\end{equation}
For the same reasons as above this diagram determines the $c^{\catD}_{X,Y}$ uniquely and turns $\catD$ into a braided category.
\item Suppose $\cat$ is in addition ribbon. 
By \emph{transporting the ribbon twist} we mean giving a natural family $\theta^\catD_X$ such that for all $U \in \cat$, $\theta^\catD_{\fun(U)} = \fun(\theta^\cat_U)$.
Again, the family $\theta^\catD_X$ is unique and turns $\catD$ into a ribbon category (with braided monoidal structure as above).
\end{itemize}

We will use the above transport of structure in the following way: Suppose we are given natural isomorphisms $\alpha^\catD$, $c^\catD$, $\theta^\catD$, but we do not know if they satisfy the coherence conditions required to make $\catD$ into a ribbon category.
One way to verify these conditions is to show that $\alpha^\catD$, $c^\catD$, $\theta^\catD$ are obtained by transport of structure from a ribbon category $\cat$ as above. By construction the multiplicative equivalence $F$ is then a ribbon functor. However, because the standard notion of a ribbon functor requires source and target to already be ribbon categories, we cannot start the proof that $\alpha^\catD$, $c^\catD$, $\theta^\catD$ satisfy the required coherence conditions with ``let $F$ be a ribbon functor''. Hence the above somewhat roundabout approach.

\begin{lemma}\label{lem-trans-QSF}
The linear maps $\assocQ$, $c^{\repQ}$, $\theta^{\repQ}$ in~\eqref{assocQ}--\eqref{thetaQ} are natural isomorphisms in $\repQ$.
There exists a multiplicative equivalence $\funQSF\colon\catSF\to\repQ$ which transports
\begin{itemize}
	\item the associator \eqref{sec:assoc-SF} of $\catSF$ to $\assocQ$, and the unit isomorphisms of $\catSF$ to those of the underlying vector spaces in $\rep\Q$,
	\item the braiding \eqref{sec:SF-braiding} of $\catSF$  to $c^{\repQ}$,
	\item the ribbon twist \eqref{eq:SF-twist} of $\catSF$ to $\theta^{\repQ}$.
\end{itemize}
\end{lemma}

The proof of this lemma is lengthy and tedious (and fills half of this paper). It is spread across the Appendices~\ref{app:SF-S} and \ref{app:Q-S}. In Appendix~\ref{app:SF-S} we transport the structure morphisms of $\catSF$ to an intermediate category of representations of a quasi-bialgebra
in $\svect$. In Appendix~\ref{app:Q-S} we transport the structure morphisms further to a quasi-bialgebra
in $\vect$ which, finally, we exhibit to be a twisting of $\Q$.

\begin{proof}[Proof of Proposition~\ref{prop-Q-qHopf}]

By Lemma~\ref{lem-trans-QSF}, $\assocQ$ fulfils the pentagon identity (since $\alpha^{\catSF}$ does and $\funQSF$ is multiplicative) and $c^{\repQ}$ the hexagon identities. Hence, $\repQ$ is braided monoidal. 
We conclude that $(\Q,\cdot,\one,\Delta,\eps,R)$ is a quasi-triangular quasi-bialgebra.

We will now show that $S,\Salpha,\Sbeta$ define a quasi-Hopf structure on $\Q$, see I:\ref*{I-subsec:conventions-ribbon-qHopf} for definitions. A straightforward calculation shows that $S$, as defined on generators in \eqref{eq:Q-antipode-def}, is compatible with the relations on $\Q$ and hence provides an algebra anti-homomorphism on $\Q$. It remains to show the identities
\begin{equation}\label{eq:Salpha-1}
\sum_{(a)}S(a')\Salpha a'' =
\eps(a)
\Salpha\gc\qquad
\sum_{(a)}a'\Sbeta S(a'') =
\eps(a)
\Sbeta \ ,
\end{equation}
for all $a\in \Salg$ and
\begin{equation}\label{eq:Salpha-2}
\sum_{(\Phi)}S(\Phi_1)\Salpha \Phi_2\Sbeta  S(\Phi_3) = \one\gc\qquad
\sum_{(\Phi^{-1})}(\Phi^{-1})_1\Sbeta S((\Phi^{-1})_2)\Salpha  (\Phi^{-1})_3 = \one \gp
\end{equation}
For example, to see the first equality in \eqref{eq:Salpha-2} one computes
\begin{align}
\sum_{(\Phi)}S(\Phi_1)\Salpha \Phi_2\Sbeta  S(\Phi_3) 
&= \idem_0 + \idem_1 \big(\beta^2(\rmi\K)^\np\big) S\big(\beta^2 (\rmi\K)^{\np}\big)  \\ 
&= \idem_0 + \idem_1 \K^\np   ((-1)^\np\K)^{\np} = \idem_0 + \idem_1 (-1)^{\np^2-\np} = \one \gp \nonumber
\end{align}

To see the first identity in \eqref{eq:Salpha-1} we define the linear map $P\coloneqq\mu\circ(S\tensor\id)\colon\Q\tensor\Q\to\Q$ where $\mu$ denotes the multiplication in $\Q$ and check that
\begin{align}\label{eq:Pf1}
P\big(\Delta(\fpm_i)\big) &= S(\ff^\pm_i) + S\big((\idem_0\pm\rmi \idem_1)\K\big)\ff^\pm_i \\ \nonumber
                  &= \fpm_i\K(\idem_0 \pm (-1)^\np \, \rmi \idem_1) + (\idem_0 \pm (-1)^\np \, \rmi \idem_1)\K\ff^\pm_i = 0 \gc \\ \nonumber
P\big(\Delta(\K)\big)     &=(\idem_0+(-1)^\np \idem_1)\K^2 - (1+(-1)^\np)\idem_1 \K^2 =\one \gp \nonumber
\end{align}
For general basis elements
$\ff_{i_1}^{\eps_1}\cdots \ff_{i_m}^{\eps_m} \K^n$ 
from~\eqref{eq:Qbasis}
and by defining 
$f_1\coloneqq \ff_{i_1}^{\eps_1}$ and $f\coloneqq \ff_{i_2}^{\eps_1}\cdots \ff_{i_m}^{\eps_m} \K^n$ we get 
\begin{align}
P\big(\Delta(\ff_{i_1}^{\eps_1}\cdots \ff_{i_m}^{\eps_m} \K^n)\big) 
&= P\big(\Delta(f_1)\Delta(f)\big)=\sum_{(f_1),(f)} S(f_1'f')f_1''f'' \\  \nonumber
 &=\sum_{(f)} S(f')P\big(\Delta(f_1)\big)f'' 
\stackrel{\eqref{eq:Pf1}}{=}
  0 \gc \\  \nonumber
P\big(\Delta(\K^n)\big)
&= \sum_{(\K^{n-1})} S\big((\K^{n-1})'\big)P\big(\Delta(\K)\big)(\K^{n-1})'' = P\big(\Delta(\K^{n-1})\big) 
= \one \ , \nonumber
\end{align} 
where the last equality follows by induction in $n$.
The second identity in \eqref{eq:Salpha-1} can be shown in a similar way.  

Having a quasi-Hopf structure we can now conclude from Lemma~\ref{lem-trans-QSF} that $\ribbon$ defines a ribbon element in $\Q$. 
In particular, $S(\ribbon)=\ribbon$ follows from the duality property of the twist $\theta_{U^*} = (\theta_U)^*$ 
(one can verify that if this equality holds for one choice of left duality $(-)^*$ on the category in question it holds for all choices of left duality).
\end{proof}

Another important consequence of Lemma~\ref{lem-trans-QSF} is the following theorem.

\begin{theorem}\label{SF-repQ-rib-eq}
$\catSF(\np,\beta)$ is ribbon equivalent to $\rep \Q(\np,\beta)$ for all choices of $\np$ and $\beta$ as in \eqref{eq:beta-param}. 
\end{theorem}

The precise definition of the equivalence functor $\funQSF\colon\catSF\to\repQ$ is given in Section~\ref{app:rib-equiv-QSF}. 

\begin{remark}\label{rem:3-ab-cohomology}
For fixed $\np$, the quasi-triangular quasi-Hopf algebras $\Q(\np,\beta)$ for the four possible choices of $\beta$ differ by abelian 3-cocycles for $\mathbb{Z}_2$. The 3rd 
abelian group cohomology of $\mathbb{Z}_2$  is $H^3_{ab}(\mathbb{Z}_2,\oC^\times) = \mathbb{Z}_4$ and it describes possible braided monoidal structures on $\rep \mathbb{Z}_2$, up to braided monoidal equivalence  (see \cite{Joyal:1993} for details). A generator of $H^3_{ab}(\mathbb{Z}_2,\oC^\times)$ is given by the class of $(\omega,\sigma)$, where $\omega$ is a 3-cocycle for group-cohomology with (writing $\mathbb{Z}_2 = \{0,1\}$ additively) $\omega(1,1,1)=-1$ and $1$ else, and $\sigma$ is a 2-cochain with only non-trivial value $\sigma(1,1)=\rmi$.
Multiplying the coassociator $\as$ by $\sum_{a,b,c \in \mathbb{Z}_2} \omega(a,b,c) \cdot \idem_a \otimes \idem_b \otimes \idem_c$ and the $R$-matrix by
$\sum_{a,b \in \mathbb{Z}_2} \sigma(a,b) \cdot \idem_a \otimes \idem_b$ is equivalent to changing $\beta$ to $\rmi \beta$. Repeating this makes $\beta$ run through its four possibilities.
\end{remark}

\subsection{Factorisability of \texorpdfstring{$\Q$}{Q}} \label{sec:fact-Q}

A finite-dimensional quasi-triangular quasi-Hopf algebra is called \emph{factorisable} if its representation category is factorisable in the sense of Definition~I:\ref*{I-def:fact-qHopf}. 
A direct definition in terms of the data of $\Q$ can be found in Remark I:\ref*{I-rem:fact-qHopf-def}. It is shown in Corollary I:\ref*{I-prop:fact-qHopf=fact-coend} that this definition is equivalent to the one given in \cite[Def.\,2.1]{BT}. We obtain the following corollary to Theorem~\ref{SF-repQ-rib-eq}:

\begin{corollary}\label{cor:Q-fact}
$\Q$ is a factorisable ribbon quasi-Hopf algebra.
\end{corollary}

Below in Lemma~\ref{lem:M-non-deg} we will give an alternative proof by verifying non-degeneracy of the Hopf pairing of the universal Hopf algebra in $\repQ$ by direct calculation.

\subsection{\texorpdfstring{$\Q(2n,1)$}{Q(2n,1)} as a Drinfeld double} \label{sec:DD}

In Remark~\ref{rem:Q-Hopf-alg} we saw that for $\np = 2n$
 and $\beta^2=1$, $\Q(\np,\beta)$ is a Hopf algebra, not only a quasi-Hopf algebra. Combining this with Corollary~\ref{cor:Q-fact} shows that $\Q(2n,\pm 1)$ is a factorisable ribbon Hopf algebra.
	It turns out that 
	as a quasi-triangular Hopf algebra,
$\Q(2n,1)$ is isomorphic to a Drinfeld double,
as we now explain.

Let $H=H(\np)$ be the algebra generated by the elements $k$ and $f_i$, $i=1,\ldots ,\np$, subject to the relations
\begin{align} \label{H-alg1}
\{f_i,f_j\}=0 \gc \qquad \{f_i,k\}=0  \gc \qquad 
	k^2=\one \gp
\end{align}
We define the coproduct, counit and antipode on $H(\np)$ as ($i = 1,\dots,\np$)
\begin{align} \label{H-alg2}
\begin{split}
\Delta(f_i) &= f_i\tensor k + \one\tensor f_i \gc \qquad \Delta(k) = k\tensor k \gc \\
\eps(f_i) &= 0 \gc  \qquad  \eps(k) = 1 \gc  \qquad  S(f_i) = -f_ik \gc \qquad S(k) = k \gp 
\end{split} 
\end{align} 
One can easily verify that we get an injective Hopf-algebra homomorphism 
$H(\np) \to \Q(\np,1)$ 
which is defined on generators by (see Appendix~\ref{app:DD} for details)
\be\label{eq:HN-QN}
	k \longmapsto \omega_-=(\idem_0-i\idem_1)\K
	\  , 
	\qquad
	f_i \longmapsto \fm_i \omega_-
	~, \quad i=1,\ldots ,\np \ .
\ee
This embedding also proves that $H(\np)$ is indeed a Hopf algebra.

\begin{remark}
Above we defined $H(\np)$ for even $\np$, but the definition works just as well for odd $\np$. 
In this case, the map \eqref{eq:HN-QN} defines an embedding $H(\np) \to \Q(\np+1,1)$ 
(or into any $\Q(2n,1)$ with $2n>\np$) 
and therefore $H(\np)$ is a Hopf algebra for all $\np \in \mathbb{N}$.
Note that $H(1)$ is Sweedler's 4-dimensional Hopf algebra, so that 
for $\np>1$ we have a generalisation of Sweedler's Hopf algebra which is different from the Taft Hopf algebra. 
In fact, $H(\np)$ is the Hopf algebra associated to (or bosonisation of)
the super-group algebra $\Lambda \mathbb{C}^\np$, see e.g.\ \cite[Sec.\,3.4]{AEG}.
\end{remark}

\begin{proposition}\label{DH:main-prop}
For  $\np\geq 1$,  the Drinfeld double of $H(\np)$ is isomorphic to $\Q(\np,\beta)$ as a $\oC$-algebra.
For even $\np$ and $\beta=\pm 1$ the Drinfeld double of $H(\np)$ is isomorphic to $\Q(\np,\beta)$ as a Hopf algebra. 
Moreover, if $\beta=1$ this isomorphism is an isomorphism of quasi-triangular Hopf algebras. 
\end{proposition}

The proof of this proposition is given in  Appendix~\ref{app:DD}.

\begin{remark}\label{rem:HN-comments}
~\\[-1.5em]
\begin{enumerate}\setlength{\leftskip}{-1em}
\item
The double of $H(\np)$ has also been constructed in \cite{Bontea:2014}. It  appears in \cite{Gelaki:2017} in the classification of factorisable tensor categories which contain $\rep H(\np)$ as a Lagrangian subcategory (and of which $\catSF(\np,\beta)$ provide four of the 16 possible cases).

\item
Following Remark~\ref{rem:3-ab-cohomology}, the quasi-triangular quasi-Hopf algebras  $\Q(2n,\beta)$ for the other three  choices of $\beta$ are simple modifications of the Drinfeld double of $H(2n)$ by the 3rd abelian cohomology classes of $\mathbb{Z}_2$.

\item
It is shown in \cite[Thm.\,6.11]{Davydov:2016euo} that $\catSF(\np,\beta)$ contains a Lagrangian algebra $L$ iff $\np$ is even and $\beta=1$ (see  \cite{Davydov:2016euo} for definition and references). This implies that precisely in these cases $\catSF(\np,\beta)$ is equivalent as a braided category to the Drinfeld centre $\mathcal{Z}(\catD)$ of some other (non-unique) finite tensor category $\catD$ \cite[Cor.\,4.1]{DMNO}, e.g.\ one may choose $\catD=L$-mod. Proposition~\ref{DH:main-prop} shows that $\catD$ can also be taken to be $\rep H(\np)$. 
\end{enumerate}
\end{remark}

\subsection{Some special elements of \texorpdfstring{$\Q$}{Q}}

The Drinfeld twist $\Dt$ of a quasi-Hopf algebra expresses the 
deviation
of the antipode from being an anti-coalgebra map via  
\begin{equation}\label{eq:anti-coalg}
	\Dt \Delta (S(a))  = (S\tensor S)(\Delta^{\operatorname{op}}(a)) \Dt \gc \qquad a\in A \ .
\end{equation} 
Its expression in terms of quasi-Hopf algebra structure maps is given in 	I:(\ref*{I-def:F})
following~\cite{Dr-quasi}. When evaluated for $\Q$, one quickly checks that the general expression reduces to
\begin{align}\label{def:F-Q}
\Dt = \idem_0\tensor\one + \idem_1\tensor \K^\np\idem_0 + \beta^2(-\rmi\K)^\np \idem_1\tensor \idem_1 \gp  
\end{align} 

As reviewed in Section~I:\ref*{I-sec:Drinfeld-element}, the canonical Drinfeld element $\sqs$ and the 
corresponding
 element $\tilde\sqs$ with inverse braiding defined as
\begin{align}\label{sqs+sqsinvbr}
\sqs &= \sum_{(\as),(R)} S\big(\as_2\Sbeta S(\as_3)\big)S(R_2)\Salpha R_1\as_1 \gc \\ \nonumber
\tilde\sqs &= \sum_{(\as),(R^{-1})} S\big(\as_2\Sbeta S(\as_3)\big)S\big((R^{-1})_2\big)\Salpha (R^{-1})_1\as_1 \gp
\end{align}

\begin{lemma} In $\Q$ the elements $\sqs$, $\tilde\sqs$ and $\sqs^{-1}$ take the form 
\begin{align}\label{sqs+sqsinvbr-Q}  
	\sqs &= \Big(\idem_0 \K + \idem_1 \beta (-\rmi\K)^\np \Big) \cdot \prod_{i=1}^{\np} (\one-2\fp_i\fm_i)  \gc \\ \nonumber
 	\tilde\sqs &= \sqs^{-1} = \big( \idem_0 \K + \idem_1 \beta^{-1} (-\rmi\K)^\np \big) \cdot \prod_{i=1}^{\np} (\one+2\fp_i\fm_i \K^2)  \gp
\end{align}
\end{lemma}
\begin{proof}
	We start with the expression for $\sqs$ in \eqref{sqs+sqsinvbr-Q}. We give the details for sector {\bf 0}, the computation in sector {\bf 1} is similar. 

In sector {\bf 0}, the first equality in \eqref{sqs+sqsinvbr} reduces to 
	$\sqs \cdot \idem_0 = S(R_2)R_1 \cdot \idem_0$.
Define the linear map $T\colon \Q\tensor\Q\to\Q , \ a\tensor b\mapsto S(b)a$, so that
 $\sqs \cdot \idem_0= T(R) \cdot \idem_0$. 
The last factor of $R$ in \eqref{R+Riv} can be written as
\begin{align}
	X\coloneqq\prod_{i=1}^\np (\one\tensor\one
	 - 2\fm_i\omega_-\tensor\fp_i) 
	= \one\tensor\one + \sum_{\substack{1\leq m\leq\np \\ 1\leq i_1 <\ldots<i_m\leq\np}} (-2)^m \fm_{i_1}\omega_-\cdots\fm_{i_m}\omega_-\tensor\fp_{i_1}\cdots\fp_{i_m} \gp
\end{align}
We need to compute
\be
	T(R) \idem_0 = T\big(\tfrac12(\one\tensor\one+\one\tensor\K+\K\tensor\one-\K\tensor\K)\cdot X\big) \cdot \idem_0 \ .
\ee
To do so we use that for $a,b,x,y\in \Q$ such that $S(y)T(a\tensor b)=\pm T(a\tensor b)S(y)$ we have
\be \label{T-quasi-alg-map}
	T(a\tensor b\cdot x\tensor y) = \pm T(a\tensor b)\cdot T(x\tensor y) \ .
\ee
Then, together with $\omega_- \idem_0 = \K \idem_0$,
\begin{align}
	 T\big(\tfrac12(\one\tensor\one-\K\tensor\K)\cdot X\big) \cdot \idem_0 
	 &=
	 (1-\K^2) T(X) \idem_0  = 0
\\ \nonumber
T\big(\tfrac12(\one\tensor\K+\K\tensor\one)\cdot X\big) \cdot \idem_0 
&=
\K \Big(\one + \sum 2^m  \fp_{i_m}\K\cdots\fp_{i_1}\K \cdot \fm_{i_1}\K\cdots\fm_{i_m}\K \Big) \cdot \idem_0
\\ \nonumber
&= 
\K \Big(\one + \sum (-2)^m \fp_{i_1} \fm_{i_1}\cdots\fp_{i_m}\fm_{i_m} 
	\Big) \cdot \idem_0
\\ \nonumber
&= 
\K \prod_{i=1}^\np(\one-2\fp_i\fm_i) \cdot \idem_0 \ ,
\end{align}
as required.

The expression for $\sqs^{-1}$ is an immediate consequence of the following identity, which is easily verified:
\be\label{eq:prod-inverse}
\Big(\prod_{i=1}^{\np} (\one-2\fp_i\fm_i)\Big)^{-1}
=~
\prod_{i=1}^{\np} (\one+2\fp_i\fm_i \K^2) \ .
\ee
Finally, 
$\tilde \sqs$ can be calculated from the identity $\tilde \sqs = S(\sqs^{-1})$,  see I:(\ref*{I-sqs-iv}).
\end{proof}

In a ribbon quasi-Hopf algebra, we define the balancing element as~\cite{[AC]}
\begin{equation}\label{balance-ribbon}
  \balance=\Sbeta S(\Salpha)\ribbon^{-1}\sqs \gc
\end{equation}
with the canonical Drinfeld element $\sqs$ defined in~\eqref{sqs+sqsinvbr}.
The balancing element $\balance$ is group-like.
 We recall 
	from Section~I:\ref*{I-sec:qHopf-pivot} 
that the pivotal structure is given by the action with $\ribbon^{-1}\sqs$ (as in the Hopf case).
The balancing $\balance$ appears in the expression
for the categorical trace $\mathrm{tr}^{\cat}$ of a morphism $f: M\to M$:
 \be
 \mathrm{tr}^{\cat}(f) = \tilde{\ev}_M \circ (f\otimes\id) \circ \coev_M =  
\tr_M(\balance \circ f)\ ,
 \ee
where we used the expressions for evaluation and coevaluation in I:(\ref*{I-eq:ev-coev}) and I:(\ref*{I-eq:qHopf-tilde-evcoev}).
In particular, the quantum dimension of $M$ is 
$\tr_M(\balance)$.

Combining \eqref{eq:Q-antipode-def}, \eqref{ribbon+ribinv} and \eqref{sqs+sqsinvbr-Q}, in $\Q$ we find the  balancing element to be 
\begin{equation}\label{balance-ribbon-Q}
\balance = (\idem_0-(-1)^{\np}\, \rmi\beta^2\idem_1)\K \gp
\end{equation}

\begin{remark}
For even $\np$ the element $\balance$ equals $\omega_\pm$ for $\beta^2=\mp 1$, while for $\np$ odd $\balance= \K^{\pm1}$ for $\beta^2 = \mp\rmi$.
Given a pivotal structure on a monoidal category with left duals, all other pivotal structures are obtained by composing the given one with natural monoidal automorphisms of the identity functor. In $\rep\Q$, these are given by acting with group-like elements in the centre of $\Q$. It follows from
Proposition~\ref{prop:cen-Q} below that these are precisely $\{ \one , \K^2 \}$. Modifying the pivotal structure by $\K^2$ has the effect of replacing $\beta^2$ by $-\beta^2$ in \eqref{balance-ribbon-Q}.
\end{remark}

\subsection{Integrals}
A \emph{two-sided integral} of a quasi-Hopf 
algebra $A$ is an element $\intQ\in A$ such that~\cite[Def.\,4.1]{Hausser:1999}
\be \label{eq:two-sided-int} \intQ a = \eps(a)\intQ = a \intQ \quad , \quad\text{for all }a\in A\gp \ee
In \cite[Sec.\,6]{BT} it is shown that factorisable quasi-Hopf algebras are unimodular. 
Together with \cite[Thm.\,4.3]{Hausser:1999} this shows that the space of two-sided integrals is one-dimensional. 
It is easy to see that every element in $\Q$  of the form 
\begin{align} \label{int-Q}
	\intQ= \nu \, 2^{\np}  \beta^2  \, \fp_1\fm_1\ldots\fp_\np\fm_\np \,\idem_0(1+\K) \qcq \nu\in\oC \ ,
\end{align}
satisfies \eqref{eq:two-sided-int}. 
	Indeed, $\intQ\K  = \intQ = \K \intQ$ and $\intQ\,\fpm_i = 0 = \fpm_i \intQ$.
The prefactor $\nu 2^{\np}\beta^2$ will be convenient later, when a normalisation condition will require the constant $\nu$ to be a sign (see Proposition~\ref{prop:int} below).

\subsection{The centre of \texorpdfstring{$\Q$}{Q}} \label{sec:Z(Q)}

We define 
\be \label{eq:epm}
\idem_1^\pm \coloneqq \ffrac12 \idem_1 \big(\one\mp \rmi\K \prod_{i=1}^\np (\one- 2\fp_i\fm_i)\big) \gp 
\ee

\begin{lemma}\label{lem:e1+-idem}
The elements $\idem_1^\pm$ are central orthogonal idempotents.
Moreover, the ribbon twist acts on $\idem_1^\pm$ by a scalar: 
\be \label{eq:epm2}
 \ribbon^{-1}\cdot \idem_1^\pm = \pm\,\beta^{-1} \idem_1^\pm \gp 
\ee
\end{lemma}
\begin{proof}
It is easy to see the commutativity property of $\idem_1^\pm$. For example, since 
\be  \K(\one-2\fp_i\fm_i)\fp_i \idem_1= -\K\fp_i \idem_1= \fp_i\K(\one-2\fp_i\fm_i) \idem_1  \ee
$\idem_1^\pm$ commutes with $\fp_i$.

The orthogonality and idempotent property follow immediately from $\idem_1 (\one-2\fp_i\fm_i)^2=\idem_1$, see \eqref{eq:prod-inverse}.
In order to prove \eqref{eq:epm2} we express $\idem_1^\pm$ in terms of $\ribbon$ from \eqref{ribbon+ribinv} as
\be  
\idem_1^\pm = \ffrac12\idem_1(\one\pm \beta^{-1}\ribbon) \qp 
\ee
Together with $\beta^{-1} \ribbon \idem_1 = \beta \ribbon^{-1} \idem_1$, this gives
\be
\ribbon^{-1} \idem_1^\pm = \tfrac12\idem_1(\ribbon^{-1} \pm \beta^{-1} \one)=\pm\beta^{-1}\tfrac12\idem_1(\pm \beta \ribbon^{-1} + \one)=\pm\beta^{-1}\idem_1^\pm \qp
\ee
\end{proof}

\begin{proposition}\label{prop:cen-Q}
	The centre of $\Q$ is $\Zc(\Q) = \Zc_0 \oplus \Zc_1$, where
	\begin{align}	\label{eq:cQ}
	\Zc_0 &~\coloneqq~ \mathrm{span}_\oC\Big\{ \idem_0 \prod_{j=1}^{2k} \ff_{i_j}^{\eps_j}  \,\big|\, 
	k\in \{ 0,1,\dots,N \},\,  
	i_j \in \{1, \dots, \np\},\,  \eps_j=\pm \Big\} 
	~\oplus~ 
\oC\,  \K\idem_0 \prod_{i=1}^\np \fp_i\fm_i 
	\gc \\  \nonumber
	\Zc_1 &~\coloneqq~ \mathrm{span}_\oC\big\{ \idem_1^+,\idem_1^- \big\} \ .   
	\end{align}
	It has dimension $3+2^{2\np-1}$.
\end{proposition}
\begin{proof}
{}From \eqref{Q0Q1} we know that
the ideal $\Q_1$ is a direct sum of two matrix algebras, so its centre is two-dimensional and is therefore spanned by the central idempotents $\idem_1^{\pm}$ from Lemma~\ref{lem:e1+-idem}.
It remains to compute the centre $\Zc_0$ of $\Q_0$. For an element of $\Q_0$ to commute with $\K$, it must be a sum of monomials in $\ff^{\pm}_i\idem_0$ of even degree, where each monomial can be multiplied by $\K$, that is,
\be\label{eq:Z0sub}
	\Zc_0 ~\subset~ \mathrm{span}_\oC\Big\{ \K^\delta  \prod_{j=1}^{2k} \ff_{i_j}^{\eps_j} \idem_0 \,\big|\, 
	k\leq \np \,,\, 1\leq i_j \leq\np \,,\, \eps_j=\pm \,,\,
 \delta \in \{0,1\}
\Big\} \ .
\ee
Any monomial from RHS of~\eqref{eq:Z0sub} with $\delta=0$ obviously commutes with $\fpm_i$, for $1\leq i\leq \np$, while
a monomial with $\delta=1$  commutes with all $\fpm_i$ iff it is annihilated by all $\fpm_i$. This gives the expression for $\Zc_0$ in \eqref{eq:cQ}.
Since the $\oC$-linear span of the even degree monomials (multiplied by $\idem_0$) has the dimension $2^{2\np-1}$, the overall dimension of the centre is
\begin{equation}
\dim_\oC \Zc(\Q) = 3 + 2^{2\np-1} \ .
\end{equation}
\end{proof}

\subsection{Simple and projective \texorpdfstring{$\Q$}{Q}-modules}\label{sec:modules}

Recall the decomposition~\eqref{Q-decomp} of $\Q$ onto  the  direct sum of two algebras  $\Q_0 \oplus \Q_1$ where the first  is the Gra\ss{}mann algebra times $\oZ_2$, which is non-semisimple, while the second is 
 the Clifford algebra times $\oZ_2$, which is semisimple. Therefore,
the algebra $\Q$ has up to an isomorphism only four simple modules that we will denote as $\XX^{\pm}_{s}$, with $s=0,1$, and where 
$\XX^{\pm}_{0}\in\rep \Q_0$ while  $\XX^{\pm}_{1}\in\rep \Q_1$. They are of highest-weight type:
 $\XX^{\pm}_{0}$ are one-dimensional of weights $\pm1$ with respect to $\K$ and with zero action of $\fpm_i$, 
 i.e.\ they are spanned by $v^{\pm}_0$ such that
\be\label{eq:v00}
\K.v^{\pm}_0 = \pm v^{\pm}_0\ ,\qquad
\fpm_i .v^{\pm}_0 = 0\ ;
\ee
$\XX^{\pm}_{1}$ are of the highest weights $\pm\rmi$, i.e.\ they are generated by $v^{\pm}_1$ such that
\be\label{eq:v0}
\K.v^{\pm}_1 = \pm\rmi v^{\pm}_1\ ,\qquad
\fm_i .v^{\pm}_1 = 0\ .
\ee
	A basis of $\XX^{\pm}_{1}$ is given by the set
\be\label{eq:X1+-_basis}
	\Big\{ \, v^{\pm}_i := \prod_{k=1}^{\np}
(\fp_k)^{i_k}.v^{\pm}_1 \,\Big|\, i=(i_1,\dots, i_\np) \,,\, 
 i_k\in \{0,1\} \,\Big\} \ .
\ee
In particular,
the dimension of $\XX^{\pm}_{1}$ is $2^\np$. 
In terms of the central idempotents $\idem_1^\pm$ from~\eqref{eq:epm},  the modules $\XX^{\pm}_{1}$ are direct summands in 
$\Q\idem_1^\pm \cong \bigl(\XX^{\pm}_{1}\bigr)^{\oplus 2^\np}$, recall that $\Q_1$ is a semisimple algebra.
The modules $\XX^{\pm}_{1}$ are therefore projective.

We note that 
\be\label{eq:idem0-pm}
\idem^{\pm}_0 = \half(\one\pm \K)\idem_0
\ee
 are primitive (non-central) idempotents and $\K \idem^{\pm}_0 = 
  \idem^{\pm}_0 \K
  = \pm \idem^{\pm}_0$.
The module $\PP^{\pm}_{0}:= \Q \idem^{\pm}_0$ is therefore a projective cover for $\XX^{\pm}_{0}$.
It is indecomposable but reducible and  has the basis 
	\be
\PP^{\pm}_{0}\, :\; \text{span} \Big\{
\Big(\prod_{k=1}^{\np}
(\fp_k)^{i^+_k}(\fm_k)^{i^-_k}\Big) \idem^{\pm}_0 
\,\Big|\, i^+_k, i^-_k\in\{0,1\}  
\  \Big\} \ .
\ee
The dimension of $\PP^{\pm}_{0}$ is thus $2^{2\np}$ and each of them has $2^{2\np-1}$ copies of $\XX^{\pm}_{0}$ 
in its  composition series.
We can finally conclude that the 
 \textit{Cartan matrix}\footnote{Its matrix elements are the multiplicities of the simple $\Q$-module $V$ in the composition series of the projective cover $P_U$ of $U$, i.e.\ $\CM(\Q)_{U,V} = \Hom_\Q(P_V,P_U)$.}  $\CM(\Q)$ 
is
\be\label{eq:CM}
\CM(\Q)
 ~=~ 
	\begin{pmatrix}
	2^{2\np-1} & 2^{2\np-1} & 0 & 0  \\
	2^{2\np-1} & 2^{2\np-1} & 0 & 0  \\
	0 & 0 & 1 & 0 \\
	0 & 0 & 0 & 1
	\end{pmatrix} \ .
\ee
One can check now the dimension of $\Q$ by decomposing it as the left regular representation: $\Q = \PP^{+}_{0} \oplus \PP^{-}_{0} \oplus 2^{\np}\XX^{+}_{1} \oplus  2^{\np}\XX^{-}_{1}$ and this indeed gives $\dim\Q = 2^{2\np+2}$.

\newcommand\Cb            {\mathbb{C}}
\newcommand{\KK}{\kappa}
\subsection{Basic algebra}

The basic algebra of $\Q$ is $E:=\End_{\Q}(G_\Q)$ where $G_\Q$ is the minimal projective generator
\be\label{eq:G_Q}
G_\Q =  \PP^{+}_{0} \oplus  \PP^{-}_{0} \oplus \XX^{+}_{1} \oplus \XX^{-}_{1} \ .
\ee 
In what follows, we will need a description of this algebra. Recall from~\eqref{Q0Q1} that $\algGr:=\algGr_{2\np}$ is the subalgebra in $\Q_0$ generated by $\ff^{\pm}_i\idem_0$.

\begin{lemma}\label{lem:End-of-projgen-Q}
$E^{\mathrm{op}} = (\algGr \oplus \Cb e_{\XX}) \rtimes \Cb\oZ_2$, where
\begin{itemize}
\item $e_\XX$ denotes	the idempotent corresponding to $\XX^{+}_{1} \oplus \XX^{-}_{1}$, i.e.\ it acts as identity on $\XX^{+}_{1} \oplus \XX^{-}_{1}$ and as zero otherwise,

\item 
 $a \in \algGr$ acts 
 on $\PP^{+}_{0} \oplus  \PP^{-}_{0} = \Q_0$ 
 by right multiplication,
\item
the generator  $\KK$
 of the group algebra $\Cb \oZ_2$ acts on $\PP^{+}_{0} \oplus  \PP^{-}_{0}$ by right multiplication with~$\K$ which is  $\pm\id$ on  $\PP^{\pm}_{0}$; 
 on $\XX^{\pm}_{1}$ it equally acts by $\pm\id$,
 
\item the algebra structure 
is that of $\algGr \oplus \Cb e_\XX$ with   $\KK a = - a \KK$ for all 
$a=\ff^{\pm}_i\idem_0$, 
$\KK e_\XX = e_\XX\KK$.
\end{itemize}
\end{lemma}
\begin{proof}
 We have the decomposition $\End(G_\Q) = E_0 \oplus E_1$ where $E_0=\End_\Q( \PP^{+}_{0} \oplus  \PP^{-}_{0})$ and $E_1= \End_\Q(\XX^{+}_{1} \oplus \XX^{-}_{1})$. The latter algebra is  a direct sum of two matrix algebras of dimension~1 each, and is isomorphic to the algebra $\Cb e_\XX\rtimes \oC\oZ_2$ from the statement.
We then note that $\PP^{+}_{0} \oplus  \PP^{-}_{0}$ is 
	equal
to the left regular representation of $\Q_0$, with the action by left multiplication. The algebra centralising this action is given by $\Q_0^{\mathrm{op}}$ which acts by right multiplication. Therefore 
 $E_0^{\mathrm{op}}=\Q_0$.
We then recall from~\eqref{Q0Q1} that 
  $\Q_0 = \algGr\rtimes \oC\oZ_2$, with   the same $\oZ_2$ action as in the statement. The fact that $\K$ acts by $\pm\id$ on the direct summands $\PP^{\pm}_{0}$ follows from the identification $\PP^{\pm}_{0}:= \Q \idem^{\pm}_0$, recall~\eqref{eq:idem0-pm}.  This finally proves the lemma.
\end{proof}

\begin{remark}
In the above lemma we give $E^\mathrm{op}$ instead of $E$ as we describe the endomorphism algebra via a right action. However, 
from the defining relations of $\Q$ it is easy to give an algebra isomorphism $E^\mathrm{op} \to E$, 
e.g.\ via the antipode $S$.
We also note that
$\repQ$ is equivalent to $\rep E$ as abelian categories,
i.e.\ $\Q$ and $E$ are Morita equivalent, and $E$ is the minimal algebra with such a property.
\end{remark}

We recall then  the equivalence stated in Theorem~\ref{SF-repQ-rib-eq} (we need only the equivalence of abelian $\oC$-linear categories) and given by the functor $\funQSF\colon\catSF\to\repQ$. 
Under this functor, the projective covers  are mapped as $\funQSF(P_{\one})=\PP^+_0$, $\funQSF(P_{\Pi\one})=\PP^-_0$, $\funQSF(T)=\XX^+_0$, and  $\funQSF(\Pi T)=\XX^-_0$, recall~\eqref{eq:SF-proj}. In particular, the minimal projective generator in $\catSF$
\be\label{eq:G-def}
G_\catSF := P_{\one} \oplus P_{\Pi \one} \oplus T \oplus \Pi T 
= (\algGr \oplus T) \otimes \oC^{1|1}
\ee
  goes to $\funQSF(G_{\catSF}) \cong G_{\Q}$.

\newcommand{\funSFQ}{\mathcal J}

We denote the functor inverse to $\funQSF$, as a $\oC$-linear functor, by $\mathcal J: \repQ \to \catSF$ and $\mathcal J = \mathcal E \circ \funQS$, where $\mathcal E$ 
is given in the proof of Proposition~\ref{prop:D-is-C-lin-equiv} and $\funQS$ in the proof of Proposition~\ref{prop:G-equiv}.
	To describe $\End_{\catSF}(G_\catSF)$ explicitly, consider
the isomorphism $\psi: \algGr \otimes \oC^{1|1}\xrightarrow{\sim} \funSFQ(\Q_0)$ given by
\be
\psi \colon\; f\tensor v \mapsto f\cdot \bigl(v_0 \idem_0^+ + v_1 \idem_0^-\bigr)\ ,\qquad f\in\algGr, \; v\in\oC^{1|1}\ ,
\ee
where $v_{0/1}\in\oC$ is the even/odd component of $v$. Using this isomorphism, we have that the endomorphism $R_a$ of $\Q_0$ from Lemma~\ref{lem:End-of-projgen-Q} given by the right multiplication with $a\in\algGr$ goes under the functor $\funSFQ$ to
\begin{align}\label{eq:SF-a-acts}
\psi^{-1} \circ \funSFQ(R_a) \circ \psi
\colon \;& \algGr \otimes \oC^{1|1} \to  \algGr \otimes \oC^{1|1}\ ,  \\
& f\otimes v \mapsto f\cdot a\otimes \Pi^{\deg a} v\gp\nonumber
\end{align}
Indeed, the shift of the degree part is due to 
\be
\psi^{-1}(R_a(f\idem^{\pm}_0)) = \psi^{-1}(f\idem^{\pm}_0\cdot a) = \psi^{-1}(f\cdot a\idem^{\mp}_0) = f\cdot a\otimes \Pi (-)\ ,
\ee
for odd $a$. 
Then, as a corollary to Lemma~\ref{lem:End-of-projgen-Q} we	get:

\begin{corollary}\label{lem:End-of-projgen-SF}
For the minimal projective generator $G_\catSF$ the opposite algebra of $\End_{\catSF}(G_\catSF)$ is
$(\algGr \oplus \Cb e_{T}) \rtimes \Cb\oZ_2$ where
\begin{itemize}
\item $e_T$ is	the idempotent corresponding to $T \oplus \Pi T$,

\item the generator $\KK$ acts by $\id$ on
$P_{\one}$ and $T$ and by $-\id$ on $P_{\Pi\one}$ and $\Pi T$,

\item the element $a\in\algGr$ acts  as in~\eqref{eq:SF-a-acts},

\item the algebra structure 
is that of $\algGr \oplus \Cb e_T$ with   $\KK a = - a \KK$ for odd $a$, 
$\KK e_T = e_T\KK$.
\end{itemize}
\end{corollary}

\section{Properties of the coend in \texorpdfstring{$\rep\Q$}{Rep(Q)}}
\label{coendQ}

In this section we investigate the universal Hopf algebra $\coend$ of $\rep\Q$, which can be expressed as a coend of a certain functor. We refer to Sections~I:\ref*{I-sec:univ-hopf-braided} and I:\ref*{I-sec:univ-hopf-via-coends} for general background and references on the universal Hopf algebra of a braided monoidal category with duals, and to Section~I:\ref*{I-sec:coend-repA} for the particular case of categories of representations of quasi-triangular quasi-Hopf algebras.

We give explicitly the Hopf algebra structure and Hopf pairing on the universal Hopf algebra $\coend$ of $\rep\Q$, we verify by direct calculation that the Hopf pairing is non-degenerate, and we describe the integrals and cointegrals of $\coend$.
In section~\ref{sec:internal-ch}, we also calculate the central elements corresponding to internal characters of $\coend$.

\subsection{The universal Hopf algebra \texorpdfstring{$\coend$}{L}}

By Proposition~I:\ref*{I-prop:coend-RepH}, universal Hopf algebra -- described by a coend $\coend$ in $\repQ$ -- can be chosen to be the object $\coend = \Q^\ast$ equipped with the coadjoint action 
\be
	\Q \otimes \Q^* \to \Q^*
	\gcg\qquad
	a \otimes \varphi
	\mapsto
	\sum_{(a)} \varphi\big( S(a') (-) a'' \big) \ , 
\ee
and the dinatural transformation
\be
 \iota_M \colon M^\ast\tensor M \to \Q^\ast \ ,\qquad \varphi\tensor m \mapsto \big(a\mapsto \varphi(a.m)\big) \gcg M\in\repQ 
 \ee
(see Figure~I:\ref*{I-fig:coend-qHopf} for a string diagram representation).

\newcommand{\mm}{\hat \mu_\coend}
\newcommand{\co}{\hat \Delta_\coend}
\newcommand{\un}{\hat \eta_\coend}
\newcommand{\coun}{\hat \eps_\coend}
\newcommand{\anip}{\hat S_\coend}
\newcommand{\Hp}{\hat \omega_\coend}
\newcommand{\Co}{D}

We define the contraction maps
\begin{align}
	\langle-,-\rangle &: \Q^* \otimes \Q \to \oC \ ,
	&&\langle \varphi,a \rangle = \varphi(a) \ ,
	\\ \nonumber
	\langle-,-\rangle &: \Q^* \otimes \Q^* \otimes \Q\otimes \Q \to \oC \ ,
	&&\langle \varphi \otimes\psi,a \otimes b \rangle = \varphi(b)\psi(a)   
	\ .
	\hspace*{4em}
\end{align}
As in Section I:\ref*{I-sec:Hopf-structure-repA}, we will express the Hopf algebra structure maps of $\coend$ in terms of their duals as follows ($f,g \in \Q^*$, $a,b \in \Q$):
\begin{align} \label{eq:coend:structuremaps}
\big\langle \muc (f\tensor g) \,,\, a \big\rangle
 &~=~ \big\langle f\tensor g \,,\, \mm(a) \big\rangle
 &&, \quad \mm : \Q \to \Q \otimes \Q \ ,
 \\ \nonumber 
\big\langle \Delta_\coend (f) \,,\, a \otimes b \big\rangle
 &~=~ \big\langle f \,,\, \co(a \otimes b) \big\rangle 
 &&, \quad \co : \Q \otimes \Q \to \Q\ ,
 \\ \nonumber 
\eta_\coend (1) &~=~ \big(\, a\mapsto \un(a) \, \big)  
 &&, \quad \un : \Q \to \oC \ ,
 \\ \nonumber
\eps_\coend (f) &~=~ f(\coun) 
 &&, \quad \coun \in \Q  \ ,
 \\ \nonumber
\big\langle S_\coend(f) \,,\, a \big\rangle &~=~ 
\big\langle f \,,\, \anip(a) \big\rangle  
 &&, \quad \anip : \Q \to \Q \ ,
 \\ \nonumber
	\omega_\coend (f\tensor g) &~=~  \big\langle f\tensor g \,,\, \Hp \big\rangle
 &&, \quad \Hp \in \Q \otimes \Q \ .
\end{align}

\begin{proposition}\label{prop:coend-Q}
	Via the dualisation in \eqref{eq:coend:structuremaps}, the Hopf algebra structure maps and the Hopf pairing on the coend $\coend = \Q^\ast$ are given by
\begin{align} \label{eq:qHopf-coend-dualstructuremaps}
\mm(a) &~=~ \sum_{\substack{(R),(a),\\ n,m\in\oZ_2}}\left(\K^{\np nm}R_2\K^{\np m}\right)\leftact a'\tensor(\K^{\np n(1-m)}\leftact a'') R_1 \cdot \idem_n\tensor\idem_m  \gc
 \\ \nonumber
\co(a\ot b)&~=~  ba \gcg 
 \\ \nonumber
\un(a)&~=~ \eps( a) \gc 
 \\ \nonumber
\coun&~=~\one \gc 
 \\ \nonumber
\anip(a)&~=~ \sum_{(R)} S(aR_1)\tilde\sqs R_2 \ ,
\\ \nonumber
\Hp&~=~ \sum_{(M)} S(M_2)\tensor M_1\idem_0 + S(\K^{-\np}M_2\K^{\np})\tensor M_1 \idem_1
\ ,
\end{align}
where $h\leftact a := \sum_{(h)} S(h')ah''$ defines an action of $\Q^{\operatorname{op}}$ on $\Q$, and $\tilde\sqs$ was defined in \eqref{sqs+sqsinvbr-Q}. 
\end{proposition}
\begin{proof} 
Applying Theorem~I:\ref*{I-thm:explicit-coend-qHopf} the only non-trivial equalities in \eqref{eq:qHopf-coend-dualstructuremaps} are the first two and last one. 
Note that for the calculation of $\un$ we used $\eps(\Sbeta)=1$.
By the same theorem we know that 
\newcommand{\p}{{\Psi}}
\newcommand{\pp}{\tilde{\Psi}}
\begin{align}\label{M-gen}
\mm(a) &~=~ \sum_{(\as),(\Psi),(\tilde{\Psi}),(R)} 
\left[S(\as_2\Psi_1 R_2'\tilde{\Psi}_3')\tensor S(\as_1\tilde{\Psi}_1)\right] \cdot \Dt
\\[-.8em] \nonumber
& \hspace{8em}
 \cdot \Delta(a \as_3) \cdot \left[(\Psi_2
R_2''\tilde{\Psi}_3'')
\tensor(\Psi_3R_1\tilde{\Psi}_2)\right] \gc
\end{align} 
where $\Psi=\as^{-1}$, $\tilde{\Psi}$ is another copy of $\as^{-1}$, and $\Dt$ is the Drinfeld twist from \eqref{def:F-Q}. 
We compute the result sector by sector. For this, it is useful to expand  $\Dt$ and $\as^{\pm 1}$ sector by sector first:
\be\label{eq:drin-phi-sectordec}
\begin{tabular}{c|c}
& $\Dt$ \\
\hline
00 & $\one \otimes \one$  \\
01 & $\one \otimes \one$  \\
10 & $\one \otimes \K^\np$  \\
11 & $\beta^2 (-\rmi)^\np \one \otimes \K^\np$ 
\end{tabular}
\hspace{5em} 
\begin{tabular}{c|c}
& $\as^{\pm1}$ \\
\hline
000 & $\one \otimes \one \otimes \one$  \\
001 & $\one \otimes \one \otimes \one$  \\
010 & $\one \otimes \one \otimes \one$  \\
100 & $\one \otimes \one \otimes \one$  \\
011 & $\one \otimes \one \otimes \one$ \\
101 & $\one \otimes \one \otimes \one$  \\
110 & $\one \otimes \one \otimes \K^\np$  \\
111 & $\beta^2 (\pm\rmi)^{\np}  \one \otimes \one \otimes \K^{\np}$ 
\end{tabular}
\ee
The elements $\Dt$ and $\as^{\pm 1}$ are recovered from these tables by summing over the idempotents $\idem_a \otimes \idem_b$ (resp.\ $\idem_a \otimes \idem_b \otimes \idem_c$) multiplied with the corresponding entry of the table.

We now give the contribution of each sector to $\mm(a)$. In doing so, we indicate which sectors of $\Dt,\as,\p,\pp,R$ contribute to the expression
(the equalities in sectors ${\bf ij}$ are true up the multiplication with the corresponding idempotents $\idem_i\tensor \idem_j$, here we omit them for brevity):
\begin{description}\setlength{\leftskip}{-1em}
	\item[{\bf 00}] $\quad \Dt: 00 \ , \ \as: 000\ ,\ \p:  000\ ,\ \pp:  000\ ,\ R:00$
			$$
	\sum_{(R)} 
\left[S(R_2')\tensor \one\right] \cdot \Delta(a) \cdot \left[R_2''\tensor R_1\right] = \sum_{(R),(a)} R_2\leftact a' \tensor a'' R_1$$

	\item[{\bf 01}] $\quad \Dt: 01 \ , \ \as: 101\ ,\ \p:  001\ ,\ \pp:  110\ ,\ R:10$
	$$ \sum_{(R)} \left[S(R_2'\K^\np)\tensor \one\right] \cdot \Delta(a) \cdot \left[R_2''\K^\np\tensor R_1\right] = \sum_{(R),(a)} (R_2\K^\np)\leftact a' \tensor a'' R_1 $$ 

	\item[{\bf 10}] $\quad \Dt: 10 \ , \ \as:011\ ,\ \p:110\ ,\ 
	 \pp:000 
	 \ ,\ R:00 $ 
		\begin{align*} &\sum_{(R)} \left[S(R_2')\tensor \one\right] \cdot \one\tensor \K^\np \cdot \Delta(a) \cdot \left[R_2''\tensor (\K^\np R_1)\right] \\ 
						&= \sum_{(R),(a)} R_2\leftact a' \tensor  \K^\np a''  \K^\np R_1  = \sum_{(R),(a)} R_2\leftact a' \tensor  (\K^\np \leftact a'') R_1 
		 \end{align*} 

	\item[{\bf 11}] $\quad \Dt: 11 \ , \ \as:110\ ,\ \p:111\ ,\ \pp:110\ ,\ R:10 $ 
		\begin{align*}
						& \big(\beta^2(-\rmi)^\np\big)^2\sum_{(R)} \left[S(R_2'\K^\np)\tensor \one\right] \cdot \K^\np\tensor\one  \cdot \Delta(a\K^\np) \cdot \left[R_2''\K^\np\tensor \K^\np R_1\right] \\ 
						&= \sum_{(R),(a)} S(R_2'\K^\np)\K^{\np} a' \K^{\np} R_2''\K^\np \tensor   a''  \K^{2\np} R_1 \\ 
						&= \sum_{(R),(a)} S(\K^{\np} R_2'\K^\np) a' \K^{\np} R_2''\K^\np \tensor   a'' R_1  \\ 
						&= \sum_{(R),(a)} (\K^\np R_2\K^\np)\leftact a' \tensor a'' R_1   \gp
		\end{align*} 
\end{description}
Combining the four sectors above results in the expression for $\mm(a)$ given in \eqref{eq:qHopf-coend-dualstructuremaps}.

In order to determine $\co$ we have to calculate 
\begin{equation}\label{eq:hat-delta-aux}
\co(a\tensor b) = \sum_{(\Co)}  S(D_1)bD_2 S(D_3)aD_4 
\end{equation} 
where $\Co=(\id\tensor\id\tensor\Delta)(\as)\cdot\one\tensor\as^{-1}\cdot\one\tensor\Sbeta\tensor\one\tensor\one$,
see I:(\ref*{I-eq:qHopf-coend-dualstructuremaps}) and I:(\ref*{I-eq:coend-maps-elementsDW}).
Note that we need $\Co$ only in sectors {\bf 0000} and {\bf 1111} which are $\one^{\tensor 4}$ and $\one\tensor\K^\np\tensor\K^\np\tensor \one$, respectively. It is then immediate that \eqref{eq:hat-delta-aux} reduces to $\co(a\tensor b) = ba$ as claimed in \eqref{eq:qHopf-coend-dualstructuremaps}.

The Hopf pairing is given by (see I:(\ref*{I-eq:qHopf-coend-dualstructuremaps}) and I:(\ref*{I-eq:coend-maps-elementsDW}))
\begin{align}
	\Hp~=~ \sum_{(W)} S(W_3) W_4 \ot S(W_1) W_2
\end{align}
with
$W = (\one\tensor\Salpha\tensor\one\tensor\Salpha)\cdot
(\one\tensor\as^{-1})\cdot(\one\tensor
 M
\tensor\one)\cdot(\one\tensor\as)
\cdot(\id\tensor\id\tensor\Delta)(\as^{-1})$. 
Recall from \eqref{eq:Q-antipode-def} and \eqref{eq:drin-phi-sectordec}, that $\Salpha=\one$ and that $\as$ is non-trivial only in the third tensor factor. Thus $W$ simplifies to
\be
 W = (\one\tensor M \tensor\one)
\cdot(\id\tensor\id\tensor\Delta)(\as^{-1})  \ .
\ee
To compute $\Hp$ we only need the following four sectors of $W$:
\be
\begin{tabular}{c|l}
& \hspace{2em} $W$ \\
\hline
0000 & $\one \tensor M \tensor \one$\\
0011 & $\one \tensor M \tensor \one$\\
1100 & $(\one \tensor M \tensor \one) \cdot (\one \tensor \one \tensor \K^\np \otimes \K^\np)$ \\
1111 & $(\one \tensor M \tensor \one) \cdot (\one \tensor \one \tensor \K^\np \otimes \K^\np)$
\end{tabular}
\ee
{}From this it is straightforward to read off the expression for $\Hp$ in \eqref{eq:qHopf-coend-dualstructuremaps}.
\end{proof}

\subsection{Non-degeneracy of the monodromy matrix}
 
We compute the monodromy matrix (or double braiding) $M=R_{21} R\in\Q\tensor\Q$. 
For this, we need the  identities (for any  $n,m\in\oZ_2$)
\begin{equation}
\ff^{\pm}\tensor\ff^{\mp}\omega_-\cdot \rho_{n,m}\cdot \idem_n\tensor\idem_m = (-1)^{m+1}  \rho_{n,m}\cdot \ff^{\pm}\omega_-\tensor\ff^{\mp}\cdot  \idem_n\tensor\idem_m
\end{equation}
(note that $\omega_-=(\idem_0-\rmi\idem_1)\K$ appears in different tensor factors)
and
\begin{equation}
(\rho_{m,n})_{21} \cdot \rho_{n,m} \cdot  \idem_n\tensor\idem_m = (-1)^{nm}  \K^m\tensor\K^n \cdot  \idem_n\tensor\idem_m \gc
\end{equation}
where $(\rho_{m,n})_{21}$ stands for the flip of $\rho_{m,n}$ from \eqref{eq:car-fac}.
Then, using the expression~\eqref{R+Riv} for the $R$-matrix the computation is straightforward:
\begin{multline}\label{M}
M=R_{21} R = \sum_{n,m=0}^{1} \bigl(-\beta^2\bigr)^{nm}\,
\K^m\idem_n\tensor\K^n\idem_m \\
\times \prod_{j=1}^{\np}\Big\{ 
\big(\one\tensor\one + 2 (-1)^m \ff^+_j\omega_-\tensor\ff^-_j\big)
\big(\one\tensor\one - 2\ff^-_j\omega_-\tensor\ff^+_j\big) \Big\}\gp
\end{multline}

\begin{lemma}\label{lem:M-non-deg}
The $M$-matrix is non-degenerate, i.e., it  can be written as $M=\sum_{I \in X} g_I\tensor f_I$, where $\{g_I\}_{I \in X}$ and $\{f_I\}_{I \in X}$ are two bases of $\Q$, for some indexing set $X$.
\end{lemma}

\begin{proof}
Using the expression~\eqref{M}, we can rewrite $M$ as
\begin{align}\label{M-non-deg}
\begin{split}
M=  \sum_{n,m=0}^{1} \sum_{s_1,t_1=0}^1 \!\!\!\dots\!\!\! \sum_{s_{\np},t_{\np}=0}^1 & \bigl(-\beta^2\bigr)^{nm}  2^{\sum_{i}(s_i+t_i)} (-1)^{\sum_i(m t_i + s_i)}
\\ 
&\times \K^m \idem_n \prod_{j=1}^\np (\tilde{\ff}^+_j)^{t_j}(\tilde{\ff}^-_j)^{s_j} \tensor  \K^n \idem_m\prod_{j=1}^\np (\ff^-_j)^{t_j}(\ff^+_j)^{s_j}
\end{split}
\end{align}
where $\sum_i$ stands for $\sum_{i=1}^{\np}$
and $\tilde{\ff}^\pm_j=\ff^\pm_j\omega_-$. This expression suggests to introduce  two bases in $\Q$,
\begin{align}
\begin{split}
f_I &=  \K^n \idem_m\prod_{j=1}^\np (\ff^-_j)^{t_j}(\ff^+_j)^{s_j} \gc\\
g_I &= \bigl(-\beta^2\bigr)^{nm}  2^{\sum_{i}(s_i+t_i)} (-1)^{\sum_i(m t_i + s_i)}
 \K^m \idem_n\prod_{j=1}^\np (\tilde{\ff}^+_j)^{t_j}(\tilde{\ff}^-_j)^{s_j}
 \ ,
\end{split}\label{W:basis}
\end{align}
where the indices $I$ run over the set
\be
X = \big\{ (n,m,s_1,t_1,\dots,s_\np,t_\np) \,\big|\, n,m,s_j,t_j\in\oZ_2,\, j=1,\dots,\np\big\} \ .
\ee
We then have $M=\sum_{I \in X} g_I\tensor f_I$ and thus the statement of the lemma is proven.
\end{proof}

As a consequence of Lemma~\ref{lem:M-non-deg}, we have that $\Hp$ is also non-degenerate (the antipode $S$ and the conjugation with $\K$ map a basis to another basis) and this is a direct proof of Corollary~\ref{cor:Q-fact}.

\subsection{Integrals and cointegrals for the coend}

Since $\rep\Q$ is factorisable (Corollary~\ref{cor:Q-fact}), $\coend = \Q^*$ has a one-dimensional space of two-sided integrals $\IntL : \oC \to \coend$ (see Proposition~I:\ref*{I-prop:integrals-exist}, due to \cite{Lyubashenko:1995}). This space of integrals contains an element (unique up to a sign) normalised such that $\Hpair \circ (\IntL \otimes \IntL) = \id_\oC$. From Lemma~I:\ref*{I-lem:coint-from-int} (due to \cite{Kerler:1996}) we furthermore know that there is a two-sided cointegral 
$\coint_\coend\colon \coend \to \one$ such that
$\coint_\coend=\omega_\coend  \circ (\IntL \ot \id_\coend)$. 

In this section we give $\IntL$ and $\coint_\coend$ for $\coend = \Q^*$ in the above normalisation.
We refer to Section~I:\ref*{I-subsec:conventions-Hopf-bmc} for our conventions for integrals and cointegrals in braided categories.

By Proposition~I:\ref*{I-prop:coint-coend-via-int-qHopf} the integrals on $\Q$ from \eqref{int-Q} define cointegrals on
$\coend = \Q^*$
by 
\be\label{eq:coint-coend-via-intQ}
\coint_\coend=\langle - \,,\, \intQ \rangle\in\Q^{**} \gp 
\ee  
To describe the integrals of $\coend$,
let $\tilde B_m=\{\K^m,\fp_{1}\K^m,\fm_1\K^m,\fp_{1}\fm_1\K^m, \, \ldots \, , \fp_{1}\fm_1\cdots\fp_\np\fm_\np\K^m\}$, so that
$B=\bigcup_{m=0}^3 \tilde B_m $ is a basis in $\Q$. We use the basis dual to $B$ to define the element $\hat\Lambda_{\coend}\in\Q^*$ as
\be \label{eq:intL-hat}
  \hat\Lambda_{\coend} = (-1)^\np \nu \beta^2 2^{1-\np} \Big(\prod_{i=1}^\np \ff^+_i \ff^-_i \Big)^\ast \ ,
\ee
where $\nu \in \oC$ is the same constant as in the definition of the integral for $\Q$ in \eqref{int-Q}.
In what follows, we will use several times the following simple lemma
	(whose proof we omit).

\begin{lemma}\label{lem-M-central}
Let $\{ w_I \}_{I \in X}$, for some indexing set $X$,
be a linearly independent subset of $\Q$ such that each $w_I$ is a monomial in the generators $\fpm_i$ and $\K$.
Assume that a central element $z\in\Zc(\Q)$ can be written as
 \be
 z= \sum_{I\in X} \alpha_I w_I\ , \qquad \alpha_I\in\mathbb{C}\ , \ \alpha_I \neq 0 \ .
 \ee
 Then each $w_I$ commutes with $\K$.
 \end{lemma}

\begin{proposition} \label{prop:int}
The linear map $\IntL\colon\oC\to \coend$ given by $\IntL(1)=\hat\Lambda_{\coend}$ is a two-sided integral for $\coend$. For $\nu \in \{ \pm 1\}$ this integral satisfies
\be
	\Hpair \circ (\IntL \otimes \IntL) = \id_\oC
	\quad , \quad
	\coint_\coend \circ \IntL = \id_\oC \ .
\ee
\end{proposition}

\begin{proof} 
We know that $\coint_\coend$ in \eqref{eq:coint-coend-via-intQ} is a cointegral for $\coend$. To show that $\IntL$ is a two-sided integral, by Lemma~I:\ref*{I-lem:coint-from-int} it is enough to verify the identity
$\coint_\coend=\omega_\coend  \circ (\IntL \ot \id_\coend)$. 
In the present setting, this is equivalent to 
$\langle - , \intQ \rangle = \omega_\coend (\hat\Lambda_{\coend} \ot -)$ or 
\be\label{eq:c-M}
 \intQ=\sum_{(\omega_\coend)} (\hat\omega_\coend)_1\, \hat\Lambda_{\coend}\big((\hat\omega_\coend)_2\big) 
 =  \sum_{(M)}  \hat\Lambda_{\coend}(M_1)S(M_2)
 \ ,
\ee
where for the last equality we used Lemma~\ref{lem-M-central} (each non-zero term in the sum commutes with $\K$).
\begin{align*}
 \text{RHS of}\; \eqref{eq:c-M} &=\\ \nonumber
  \sum_{m,n=0}^{1}&(-\beta^2\bigr)^{mn}  2^{2\np} (-1)^{\np(m+1)}  
	 \overbrace{\hat\Lambda_{\coend}(\K^m\idem_n \fp_1\omega_-\fm_1\omega_-\ldots\fp_\np\omega_-\fm_\np\omega_-)}^{= 
	 \, \delta_{m,0} \, \nu \beta^2 2^{-\np} } 
 S(\K^n \idem_m \fm_1\fp_1\ldots\fm_\np\fp_\np) \\ \nonumber
&=  (-2)^{\np} \nu \beta^2 \big(	S(\idem_0 \fm_1\fp_1\ldots\fm_\np\fp_\np) + 	S(\idem_0  \fm_1\fp_1\ldots\fm_\np\fp_\np \K) \big) \\ \nonumber
&=  2^{\np} \nu \beta^2 \idem_0 \fp_1\fm_1\ldots\fp_\np\fm_\np (1+\K) = \intQ \gp
\end{align*}
It remains to show that the normalisation condition holds. Indeed, we have
\begin{align}
	\Hpair(\hat\Lambda_{\coend}\tensor \hat\Lambda_{\coend}) &= \langle \hat\Lambda_{\coend}\tensor\hat\Lambda_{\coend} \,,\, \Hp \rangle 
\overset{\eqref{eq:c-M}}= \hat\Lambda_{\coend}(\intQ)  \\ \nonumber
&= 2^{\np} \nu  \beta^2 \hat\Lambda_{\coend}(\idem_0 \fp_1\fm_1\ldots\fp_\np\fm_\np) 
= \nu^2 = 1 \gp 
\end{align}
\end{proof}
  
\subsection{Internal characters and \texorpdfstring{$\phi_M$}{phi-M}}\label{sec:internal-ch}
We first recall that the {\em internal character}	of $V \in \rep \Q$ is the intertwiner from the trivial representation to the coend $\coend$
 given by (see \cite{Fuchs:2013lda,Shimizu:2015}) 
\be\label{eq:chiV-def}
	\chi_V ~=~ 
	\big[\, \one \xrightarrow{\widetilde\coev_V} V^* \ot V   \xrightarrow{\iota_{V}} \coend \,\big] \ ,
\ee
where we follow conventions Section~I:\ref*{I-sec:intchar-natendo}.
Via the isomorphism 
$\mathrm{Hom}_\Q(\one,\coend) \to \mathrm{End}(Id_{\repQ})$ 
given in I:(\ref*{I-eq:tildechi-def}), the internal characters $\chi_V$ correspond to natural endomorphisms $\phi_V$ of the identity functor. 
For $\cat=\repQ$, under the categorical modular $S$-transformation $\modS_{\cat}\colon \mathrm{End}(Id_{\cat}) \to  \mathrm{End}(Id_{\cat})$ in I:(\ref*{I-eq:SCTC}) the
 images $\phi_V$ of the internal characters
are mapped to natural endomorphisms $\modS_\cat(\phi_V)$ which are the 
	``Hopf link operators''  
in I:(\ref*{I-eq:hopf-link-op}). 
We also recall from 
Corollary~I:\ref*{I-cor:phiM-Gr-indepL}
and Theorem~I:\ref*{I-thm:S_C(phi_M)-algebramap}
that the assignments $[V] \mapsto \phi_V$ and  $[V] \mapsto \modS_\cat(\phi_V)$ are injective linear maps $\Gr_\mathbb{C}(\cat) \to \End(Id_\cat)$ (the second map is actually an algebra map).
Both, $\phi_V$ and $\modS_\cat(\phi_V)$ will be useful in the computation of 
the 
$S$-transformation on the centre $Z(\Q)$
in Section~\ref{sec:sl2Z-Q-SF} and when comparing $\SLiiZ$-actions in Section~\ref{sec:SL2Z-compare}.
	 
Let us identify $\chi_V$ with their images $\chi_V(1)\in\Q^{*}$. 
In I:(\ref*{I-eq:chi-Tr}), we expressed these linear forms via the traces
\be\label{eq:chi-tr}
\chi_V(-) = \tr_V(\bal \cdot - )\ ,\qquad  \bal = \sqs^{-1} \ribbon S(\Sbeta)\ ,
\ee
where explicitly
\be\label{eq:kappa}
\bal = (\idem_0  - \rmi\beta^2\idem_1)\K \ .
\ee 

The $\phi_V \in \mathrm{End}(Id_{\repQ})$ correspond to central elements  $\bphi_V \in \Zc(\Q)$, see  Section~I:\ref*{I-sec:intchar-qHopf}. We now compute the $\bphi_V$ for all simple modules $V$.

\begin{lemma}\label{eq:phiV-explicit}
We have 
\be \label{phiV-Q}
\bphi_V = \sum_{(\intQ)} \intQ' \ot \chi_V(S(\intQ'')) 
\ee
and, in particular,
	\be\label{eq:bphi-X}
	\bphi_{\XX^{\pm}_{0}}  = \nu \, 2^{\np}  \beta^2  (\K\pm \one)\idem_0 \prod_{i=1}^\np \fp_i\fm_i \ ,\qquad
	\bphi_{\XX^{\pm}_{1}}  = \pm \, \nu\, 2^{\np+1}  \idem_1^{\pm} \ . 
	\ee
\end{lemma}
\begin{proof}
Recall from Section~I:\ref*{I-sec:intchar-qHopf} that
\be \label{bphi}
\bphi_V ~=~ \sum_{(F)}  F_1 \, \chi_V(F_2) \ ,
\ee
where
\be
F= \eps(\Sbeta) 	\sum_{(\Psi),(\Phi),(\intQ)}
 \Psi_1 \intQ' \Phi_1 \tensor S(\Psi_2  \intQ''\Phi_2)\Salpha\Psi_3\Phi_3
\quad \text{and}\quad \Psi=\Phi^{-1}\ .
\ee
It is straightforward to see that 
for $\Q$ we have
\be
	F= \sum_{(\intQ)} \intQ'\ot S(\intQ'') \gc
\ee
which together with \eqref{bphi} proves \eqref{phiV-Q}.

Next we compute	from \eqref{int-Q} that
\begin{align}
\Delta(\intQ) &= a \big(\one\ot\one + \K^2\ot \K^2\big)\big(\one\ot\one + \Delta(\K)\big) 
\\ \nonumber
& \hspace{5em}
\times \prod_{i=1}^\np \big(\fp_i\fm_i \ot \one +  \fp_i \omega_- \ot \fm_i -  \fm_i \omega_+\ot \fp_i + \K^2\ot \fp_i\fm_i\big)\ ,
\end{align}
	where $a =  \nu \, 2^{\np-1}  \beta^2$.

We note that for any $x\in\Q$ we get 
\be
\tr_{\XX^{\pm}_{i}}(\idem_j x)= \delta_{i,j}\tr_{\XX^{\pm}_{i}}(x)\ , \quad i,j\in\{0,1\}\ .
\ee
Recall in Section~\ref{sec:modules} that  $\XX^{\pm}_{0}$ have the trivial action of $\fpm_i$. 
Therefore, in the traces over $\XX^{\pm}_{0}$ only the basis elements 
$\K^n$, $0\leq n\leq 3$, 
have non-zero contribution.
Then for $V=\XX^{\pm}_{0}$ and using~\eqref{eq:chi-tr} together with~\eqref{eq:kappa} we get
\be
\bphi_{\XX^{\pm}_{0}} = \sum_{(\intQ)} \tr_{\XX^{\pm}_{0}}(\K \intQ'') \intQ'   = 2a (\K\pm \one)\idem_0 \prod_{i=1}^\np \fp_i\fm_i  \ ,
\ee
where we used that only the components with $\intQ''=\K^n$ contribute to the expression, and that $S(\K^n\idem_0) = \K^n\idem_0$. This gives the first expression in~\eqref{eq:bphi-X}.

For the computation of $\bphi_{\XX^{\pm}_{1}}$ we first recall that $\XX^{\pm}_{1}$ are simple projective, as discussed in Section~\ref{sec:modules}. 
	It then follows from~\cite[Eq.\,(7.5)]{Gainutdinov:2017kae} that $\bphi_{\XX^{\pm}_{1}} = b_{\pm} \idem_1^{\pm}$ for some non-zero $b_{\pm}\in \mathbb{C}$. 
To determine $b_{\pm}$ it is enough to act on the highest-weight vector $v_0^{\pm}\in \XX^{\pm}_{1}$ defined by~\eqref{eq:v0}.
We then note that in the expression 
\be
\bphi_{\XX^{\pm}_{1}} . v_0^{\pm} = -\rmi \beta^2 \sum_{(\intQ)} \tr_{\XX^{\pm}_{1}}\big(\K S(\intQ'')\big) \intQ' . v_0^{\pm}
\ee
only the term $a (\one\ot\one + \K^2\ot \K^2)(\one\ot\one + \Delta(\K)) \big(\K^{2\np}\ot \prod_{i=1}^\np \fp_i\fm_i\big)$ in  $\Delta(\intQ)$ gives a non-zero contribution: 
the other terms  either have $\fm_i$ in $\intQ'$, and therefore are zero on $v_0^{\pm}$, or have $\fp_k$ without $\fm_k$ in  $\intQ''$ and therefore zero in the trace. 
Within the trace $\tr_{\XX^{\pm}_{1}}(-)$,
	we replace 
	$\idem_1 S( \prod_{i=1}^\np \fp_i\fm_i)= \idem_1 \prod_{i=1}^\np \fm_i\fp_i$. 
Therefore, we need to calculate the trace of the operator $\K^n \prod_{i=1}^\np \fm_i\fp_i$. For any $n$, this operator in the basis $v^{\pm}_i$
	from \eqref{eq:X1+-_basis} 
is given by a diagonal matrix with one-dimensional eigenspace of non-zero eigenvalue -- spanned by $v_0^{\pm}$. We thus have $\tr_{\XX^{\pm}_{1}}\big(\K^n \prod_{i=1}^\np \fm_i\fp_i\big) = (\pm \rmi)^n$.
Using this, a simple calculation finally gives
\be
\bphi_{\XX^{\pm}_{1}}    = \pm   \nu \, 2^{\np+1}   \idem_1^{\pm}  \ ,
\ee
which agrees with the second expression in~\eqref{eq:bphi-X}.
\end{proof}

The Hopf link operators $\modS_\cat(\phi_V) \in \mathrm{End}(Id_{\repQ})$ correspond to central elements  $\bchi_V \in \Zc(\Q)$ given by
\be\label{eq:bchi-general}
\bchi_V
~=
\sum_{(\Phi),(\Psi),(M)} \tr_V\Big( \bal\, S\big(\Psi_2 M_2  \Phi_2  \big)  \Salpha \, \Psi_3  \Phi_3  \Big) \, \Psi_1 \, M_1\,  \Phi_1 \ ,
\ee
	see I:(\ref*{I-eq:chiV-central-def}).
We now compute the $\bchi_V$ for all simple $\Q$-modules $V$.

\begin{lemma}\label{lem:bchi-Q}
We have 
\be\label{eq:bchi-V}
\bchi_V
~=
\sum_{(M)} \tr_V\big(\bal \, S(M_2)\big) M_1\ 
\ee
and on simple $V$:
\be\label{eq:bchi-X}
\bchi_{\XX^{\pm}_{0}}  = \idem_1 \pm \idem_0 \ ,\qquad
\bchi_{\XX^{\pm}_{1}}  = \pm  \beta^2 4^N \K \idem_0 \prod_{j=1}^\np \fp_j\fm_j  + 2^N (\idem_1^+ - \idem^-_1) \ . 
\ee
\end{lemma}
\begin{proof}
It is straightforward to see that
	for
$\Q$ the expression in~\eqref{eq:bchi-general} reduces to
\be
\bchi_V
~=
\sum_{(M)} \tr_V\big(\bal \, S(M_2)) M_1\ .
\ee
We first compute
\be
\bchi_{\XX^{\pm}_0}
~=
\sum_{(M)} \tr_{\XX^{\pm}_0}\big(\K \, S(M_2)) M_1 = \idem_1 \pm \idem_0\ ,
\ee
where we used that only the Cartan part of $M$ 
	in \eqref{M}
contributes into the trace over $\XX^{\pm}_0$.

Next, to compute $\bchi_{\XX^{\pm}_1}$ we study the action of these central elements on all the four projective covers, actually on generating vectors of the covers. Recall that they are $v^{\pm}_1$ for $\XX^{\pm}_1$ and $\idem^{\pm}_0$ for $\PP^{\pm}_0$,
recall~\eqref{eq:idem0-pm}.
 The point is that only very few terms in an expansion of $\bchi_{\XX^{\pm}_1}$ will contribute 
to the action. 
To start we act on $v^{\alpha}_1$, $\alpha=\pm$,
\be\label{eq:bchi1-1}
\bchi_{\XX^{\pm}_1}.v^{\alpha}_1 = -\rmi \beta^2 \sum_{(M)} \tr_{\XX^{\pm}_1}\big(\K \, S(M_2)) M_1 .v^{\alpha}_1 = \alpha 2^N v^{\alpha}_1 \ ,
\ee
where again only the Cartan part of $M$ (actually just the $n=m=1$ term) contributed into the non-zero action. Next, we have
\be\label{eq:bchi1-2}
\bchi_{\XX^{\pm}_1}.\idem^{\alpha}_0 = -\rmi \beta^2 \sum_{(M)} \tr_{\XX^{\pm}_1}\big(\K \, S(M_2)) M_1 .\idem^{\alpha}_0 =\pm \alpha \beta^2 4^N \prod_{j=1}^\np \fp_j\fm_j \idem^{\alpha}_0 \ ,
\ee
where we also noticed that only one term in the expansion of $M$:
\be
M = (-1)^N 4^N (\K\idem_0 \ot \idem_1) \prod_{j=1}^\np \fp_j\fm_j \ot \prod_{j=1}^\np \fm_j\fp_j + \ldots
\ee
contributes, otherwise the trace over $\XX^{\pm}_1$ is zero. We used here
\be
\tr_{\XX^{\pm}_{1}}\big(\K \prod_{k\ \text{terms}} \fp_i\fm_i\big) = \pm \rmi (-1)^k \delta_{\np,k}\ .
\ee
Combining~\eqref{eq:bchi1-1} and~\eqref{eq:bchi1-2}, we finally get
\be
\bchi_{\XX^{\pm}_1} = \pm  \beta^2 4^N \K \idem_0 \prod_{j=1}^\np \fp_j\fm_j  + 2^N (\idem_1^+ - \idem^-_1)\ .
\ee
\end{proof}

\section{\texorpdfstring{$SL(2,\oZ)$}{SL2Z}-action on the centre of \texorpdfstring{$\Q$}{Q}}\label{sec:sl2Z-Q-SF}

In this section, we specialise the projective $SL(2,\oZ)$-action on the centre $\Zc$ of a general factorisable ribbon quasi-Hopf algebra obtained in Section~I:\ref*{I-sec:SL2Z-quasiHopf} to the symplectic fermion quasi-Hopf algebra $\Q$,
	and we study the decomposition of this projective representation.
 The result is summarised in Theorem~\ref{thm:ST-from-Q}.

By a \textit{projective representation} of a group $G$ we mean a group homomorphism $\bar\rho : G \to PGL(V)$ for $V$ a vector space. When specifying projective representations, we will give the map $\rho : G \to GL(V)$ instead, with the projective representation then being the composition with the projection $GL(V) \to PGL(V)$. 
Two projective representations $\bar\rho : G \to PGL(V)$ and $\bar\sigma : G \to PGL(W)$ are \textit{equivalent} if there is a linear isomorphism $f : V \to W$ such that $\bar\sigma = \overline{f \circ \rho \circ f^{-1}}$, where $\rho : G \to GL(V)$ is any lift of $\bar\rho$.
In particular, changing $\rho$ by $g$-dependent constants, $\rho(g) \leadsto \lambda_g \rho(g)$, $\lambda_g \in \mathbb{C}^\times$,
$g\in G$,
 leads to an equivalent (in fact, the same) projective representation.
Thus, every one-dimensional projective representation is equivalent to the trivial representation. 
We call a projective representation of $SL(2,\oZ)$ fundamental, if it is projectively equivalent to the fundamental representation of $SL(2,\oZ)$ on~$\mathbb{C}^2$.

By the $S$- and $T$-generators of $SL(2,\oZ)$ we mean the $2{\times}2$ matrices $\mathbf{S} = \big( \begin{smallmatrix} 0 & -1 \\ 1 & 0 \end{smallmatrix} \big)$ 
and
$\mathbf{T} = \big( \begin{smallmatrix} 1 & 1 \\ 0 & 1 \end{smallmatrix} \big)$,
respectively. One can describe $SL(2,\oZ)$ as the group freely generated by $\mathbf{S}$ and $\mathbf{T}$ subject to the relations
\be\label{eq:SL2Z-gen-and-rel}
	(\mathbf{S}\mathbf{T})^3 = \mathbf{S}^2
	\quad , \quad
	\mathbf{S}^4 = \id \ .
\ee

By Theorem~I:\ref*{I-thm:SL2Z-on-centre},
the (projective) action of $\mathbf{S}$ and $\mathbf{T}$ on the center $\Zc$ of a factorisable ribbon quasi-Hopf algebra is given by the following invertible linear endomorphisms of $\Zc$, for $z \in \Zc$,
\begin{align} \label{eq:STact-ZA}
	\modS_\Zc(z)&~= \sum_{\substack{(\Psi),(\Hp)}} 
	\Psi_1 \, \Sbeta \, S(\Psi_2)
	\, ({\Hp})_1 \, \Psi_3 \, 
		\hat\Lambda_\coend
	\Big(\co\big((\Hp)_2\ot \Salpha z\big)\Big) \ , \\
	\modT_\Zc(z)&~= \ribbon^{-1} z \gc 
\label{eq:Tact-ZA}
\end{align}
where $\Psi=\Phi^{-1}$, the dual structure maps $\co$, $\Hp$ are defined  by~\eqref{eq:coend:structuremaps}.
	We evaluate~\eqref{eq:STact-ZA} for the quasi-Hopf algebra $\Q$ using  the map $\co$ and  the explicit form of $\Hp$ computed in Proposition~\ref{prop:coend-Q}:
\be\label{eq:modS-Q}
	\modS_\Zc(z)~= \sum_{\substack{(\Hp)}} 
	({\Hp})_1 \,
		\hat\Lambda_\coend
	\big(z (\Hp)_2\big) = \sum_{\substack{(M)}} \, \hat\Lambda_\coend(M_1 z) S(M_2) \ ,
\ee
where we also used the fact that in the sum all  non-zero summands commute with $\K$,  recall Lemma~\ref{lem-M-central}  and that $\modS_\Zc(z)$ is central.
We also recall that $\hat\Lambda_\coend$ is given in \eqref{eq:intL-hat}.

Our aim is to give a decomposition of the 
	$SL(2,\oZ)$-action~\eqref{eq:Tact-ZA}, \eqref{eq:modS-Q} 
on $\Zc=\Zc(\Q)$.
In Proposition~\ref{prop:cen-Q}, we described a decomposition of the centre $\Zc= Z_0\oplus Z_1$ and a basis in it.
In what follows, we will need a slightly different decomposition:
\be\label{ZQ-2}
\Zc(\Q) = \Zb \oplus \Za  
\ee
with
\begin{gather}\label{eq:Zproj-Zf-span}
\Zb\coloneqq\rm{span}_\oC\big\{\,\bphi_{\PP^{+}_{0}}\,,\, \bphi_{\XX^{+}_{1}}\,,\,\bphi_{\XX^{-}_{1}}\big\} \gc 
\\ \nonumber
 \Za\coloneqq {\rm span}_\oC\Big\{
  \idem_0 \prod_{j=1}^{2k} \ff_{i_j}^{\eps_j}  \mid
   k\in \{0,1, \dots, \np\}\,,\,
    i_j \in \{1,\dots,\np\}
     \,,\, \eps_j=\pm \Big\} \ , 
\end{gather}
where  $\Zb$ is spanned by the internal characters $\bphi_V$ for $V$ projective,
while $\Za$ is the centre of $\algGr_{2\np}$ from~\eqref{Q0Q1}.
	The central elements $\bphi_{\XX^{\pm}_{0}}$ are given in~\eqref{eq:bphi-X}.
For $\bphi_{\PP^{+}_{0}}$, recall that the map  $V \mapsto \bphi_V$ factors through the Grothendieck ring  (as reviewed in Section~\ref{sec:internal-ch} and see also Section~I:\ref*{I-sec:intchar-natendo}) and thus using the Cartan matrix in~\eqref{eq:CM} we can write 
\be\label{eq:bphiP}
\bphi_{\PP^{+}_{0}} = 2^{2\np-1} \bphi_{\XX^{+}_{0}} +  2^{2\np-1} \bphi_{\XX^{-}_{0}} = \nu \, 2^{3\np}  \beta^2 \K\idem_0 \prod_{i=1}^\np \fp_i\fm_i \ .
\ee
	This shows that \eqref{ZQ-2} is indeed a decomposition of the centre $Z(\Q)$ as computed in Proposition~\ref{prop:cen-Q}.

{}From \cite[Prop.\,7.1\,\&\,Cor.\,8.5]{Gainutdinov:2017kae} we know that $\Zb$ is an invariant subspace of $Z(\Q)$ for the $SL(2,\oZ)$-action. The next lemma gives the action of $\modS_Z$ 
and $\modT_Z$
on $\Zb$.

\begin{lemma}\label{lem:ST-on-Zproj}
The restriction of the linear maps $\modS_\Zc$, $\modT_\Zc$ from~\eqref{eq:STact-ZA} and \eqref{eq:Tact-ZA} to $\Zb$ in the basis \eqref{eq:Zproj-Zf-span}
is given by the matrices
\be\label{eq:Sb}
\Sb =
\nu 
	\begin{pmatrix}
		 0          & 2^{-\np}  &  -2^{-\np} \\ 
	  2^{\np-1}   & 1/2       & 1/2 \\
	  - 2^{\np-1}   & 1/2       & 1/2
\end{pmatrix} 
\quad , \quad
\modT_{\Zb} =
	\begin{pmatrix}
		 1          & 0  &  0 \\ 
	  0   & \beta^{-1}   & 0 \\
	  0  & 0      & -\beta^{-1}
\end{pmatrix} \gp
\ee
They provide an irreducible (projective) representation of $SL(2,\mathbb{Z})$ for all $\np$ and $\beta$ from~\eqref{eq:beta-param} if and only if $\beta \ne \pm1$.
Otherwise, the representation is the direct sum of a 1- and 
a 2-dimensional irreducible subrepresentation, and its isomorphism class is the same for all even $\np\geq2$ and $\beta=\pm1$.
\end{lemma}

We remark that the 2-dimensional projective representation in the above statement is \textsl{not} projectively isomorphic to the fundamental one, because, as we will see in the proof below, the action of $\mathbf{T}$-generator is diagonal here, while in the fundamental representation it cannot be diagonalised.

\begin{proof}[Proof of Lemma~\ref{lem:ST-on-Zproj}]
To compute the $S$-transformation on $\Zb$
 we use the relation stated in 
Corollary I:\ref*{I-cor:S(PHI)=CHI},
\be\label{eq:Sbphi}
 \modS_Z(\bphi_V) = \bchi_V \ ,
 \ee
where the central elements  $\bchi_V$ are given by~\eqref{eq:bchi-V} and were computed for simple $V$ in Lemma~\ref{lem:bchi-Q}.
Applying this formula for $V=\XX^{\pm}_1$ and using~\eqref{eq:bchi-X} together with~\eqref{eq:bphiP}, we can write
\be
\modS_Z(\bphi_{\XX^{\pm}_1} ) = \bchi_{\XX^{\pm}_1} =  \nu \big(\pm 2^{-\np} \bphi_{\PP^{+}_0} + \ffrac12 (\bphi_{\XX^{+}_1} + \bphi_{\XX^{-}_1})\big) \ .
\ee
This gives the second and third column
in~\eqref{eq:Sb}.

As discussed in Section~\ref{sec:internal-ch} (see also Section~I:\ref*{I-sec:intchar-natendo}), the map $V\mapsto \bchi_V$ factors through $\Gr_\mathbb{C}(\repQ)$ and therefore we can write $\bchi_{\PP^{+}_{0}} = 2^{2\np-1} \bchi_{\XX^{+}_{0}} +  2^{2\np-1} \bchi_{\XX^{-}_{0}}$.
 Using~\eqref{eq:bchi-X} together with the expression for $\bphi_{\XX^{\pm}_1}$ in~\eqref{eq:bphi-X}, recall that $\idem_1= \idem_1^+ + \idem_1^-$,  the relation~\eqref{eq:Sbphi}  for $V=\PP^{+}_0$ is
\be
\modS_Z(\bphi_{\PP^{+}_0} ) = \bchi_{\PP^{+}_0} =  \nu \, 2^{\np-1} (\bphi_{\XX^{+}_1} - \bphi_{\XX^{-}_1}) \ .
\ee
This finally gives the first column in~\eqref{eq:Sb} and it finishes our calculation of $\Sb$.

For the $T$-transformation~\eqref{eq:Tact-ZA} we use the expression~\eqref{eq:ribinv} for $\ribbon^{-1}$. This immediately gives the diagonal matrix for $\modT_Z$  with the entries $\{1, \beta^{-1}, - \beta^{-1}\}$ in the basis $\big\{\,\bphi_{\PP^{+}_{0}}\,,\, \bphi_{\XX^{+}_{1}}\,,\,\bphi_{\XX^{-}_{1}}\big\}$.

We now  prove  irreducibility of the $SL(2,\mathbb{Z})$ representation when $\beta\ne \pm1$. First, we notice that the decomposition under the subgroup $\langle\mathbf{T}\rangle$ generated by $\mathbf{T}$ is the direct sum of three one-dimensional representations characterised by eigenvalues of the $\mathbf{T}$-action $(1, \beta^{-1}, -\beta^{-1})$, and they are pair-wise inequivalent.
Assume now that there exists a proper  invariant subspace with respect to the action of $SL(2,\mathbb{Z})$. Clearly,  this subspace is a direct sum of irreducible representations of $\langle\mathbf{T}\rangle$, or equivalently, $\mathbf{T}$-eigenspaces. 
Because $\mathbf{S}$ generator does not leave invariant any of the $\mathbf{T}$-eigenspaces, we  have no 1-dimensional invariant subspace. It is then two-dimensional and it is a direct sum of $\mathbf{T}$-eigenspaces of non-equal eigenvalues. For any pair of such eigenvalues, we see from~\eqref{eq:Sb} that the $\mathbf{S}$-action does not leave the corresponding subspace invariant. We thus get a contradiction and the 3-dimensional representation of $SL(2,\mathbb{Z})$ is irreducible. 

Consider now the case $\beta = \pm1$. We then have a two-dimensional $\mathbf{T}$-eigenspace (of eigenvalue $\pm1$) spanned by $\bphi_{\PP^{+}_0}$ and $\bphi_{\XX^{\pm}_1}$. It is straightforward to check that both vectors 
\be\label{eq:inv-comb-char}
v_\pm = \bphi_{\PP^{+}_0} \pm 2^N \bphi_{\XX^{\pm}_1}
\ee
 are stabilised by the  $\mathbf{S}$-action. Therefore, if $\beta=1$ the $SL(2,\mathbb{Z})$-invariant subspace is spanned by $v_+$, and for $\beta=-1$ it is spanned by $v_-$. 
There is a direct sum complement to the above 1-dimensional subrepresentation: for $\beta=\pm1$ the subspace spanned by 
\be
w_1=\bphi_{\XX^{\mp}_1}\quad \text{and} \quad w_2=\frac{1}{2}\bphi_{\XX^{\pm}_1} \mp 2^{-\np}\bphi_{\PP^{+}_0}
\ee
 is invariant under the $SL(2,\mathbb{Z})$-action. 
Indeed, $w_1$ and $w_2$ are eigenvectors of $\mathbf{T}$-generator with eigenvalues $-1$ and $+1$,
while the $\mathbf{S}$-generator acts as $\mathbf{S} .w_1 = w_2 + \frac{1}{2} w_1$ and $\mathbf{S} .w_2 = \frac{3}{4}w_1 - \frac{1}{2} w_2$.
We finally note from the explicit action that the isomorphism class of these 3-dimensional 	projective 
 $SL(2,\mathbb{Z})$ representations is the same for all even $\np\geq2$ and $\beta\in\{\pm1\}$.
\end{proof}

Let $U_i \subset \Q_0$ be the subalgebra
 $U_i := {\rm span}_\oC \{\idem_0,\fm_i\idem_0,\fp_i\idem_0,\fp_i\fm_i\idem_0\}$,
  with $1\leq i\leq \np$. Consider the injective linear map $\vartheta : U_1\tensor \cdots \tensor U_\np \to \Q_0$
 given by
\be\label{eq:U1...Un-to-Q0-map}
	\vartheta(a_1 \otimes \cdots \otimes a_\np) := a_1 \cdots a_\np \ .
\ee
Note that this is not an algebra map.
Write $U_+ \subset U_1\tensor \cdots \tensor U_\np$ for the subspace spanned by all homogeneous vectors with an even overall number of $\fpm_i$'s. In other words, $U_+$ is the eigenspace of $(\K(-)\K^{-1})^{\otimes \np}$ of eigenvalue $+1$.
The direct summand $\Za$ of the centre in \eqref{eq:Zproj-Zf-span} is the image of $U_+$,
\be
	\Za = \vartheta(U_+) \ .
\ee
	In particular, $\vartheta|_{U_+} : U_+ \to \Za$ is a bijection and below we use $\vartheta^{-1}$ on elements from $\Za$.
With the help of $\vartheta$ we
will now describe the $S$-
and $T$-transformations
on $\Za$ as a tensor product of linear maps 
$\modS_{\Za}^i, \modT_{\Za}^i\colon U_i\to U_i$, for $i=1,\dots,\np$.

\begin{lemma}\label{lem:Sa}
$\modS_\Zc$ and $\modT_\Zc$ from~\eqref{eq:STact-ZA}-\eqref{eq:Tact-ZA} map $\Za$ to itself. The restrictions of $\modS_\Zc, \modT_\Zc$ to $\Za$ are given by 
\begin{align}
\label{eq:S-action-on-Zf-from-Q}
	\modS_\Zc\big|_{\Za}
	~&=~
	\vartheta \circ \Big(
	\nu\beta^{2}\cdot\big(\Sa^1\tensor \cdots \tensor\Sa^\np\big)\big|_{U_+}
	\Big)
	\circ \vartheta^{-1}\ ,
	\\ 
\label{eq:T-action-on-Zf-from-Q}
	\modT_\Zc\big|_{\Za}
	~&=~
	\vartheta \circ \big(\Ta^1\tensor \cdots \tensor\Ta^\np\big)\big|_{U_+}
	\circ \vartheta^{-1} \ ,
\end{align}
where the individual maps 
$\modS_{\Za}^i$, $\modT_{\Za}^i$ are given in the basis
 $\{ \fp_i\fm_i\idem_0,  \idem_0, \fm_i\idem_0,\fp_i\idem_0 \}$
   of $U_i$ by the matrices
\be \label{eq:STact-mat} 
\Sa^i =
\begin{pmatrix}
0  &2 & 0 & 0 \\
 -2^{-1} & 0 & 0 & 0\\
0 & 0 & 1 & 0  \\
0 & 0 & 0 & 1  
\end{pmatrix}
\qquad , \qquad
\Ta^i =
\begin{pmatrix}
1 & 2 & 0 & 0\\
0 & 1 & 0 & 0 \\
0 & 0 & 1 & 0 \\
0 & 0 & 0 & 1 
\end{pmatrix}
\quad .
\ee
In particular, the restriction of  $\modS_{\Zc}^2$ to $\Za$ is the identity,
and $\modT_\Zc$ has Jordan blocks of maximum rank $\np+1$. 
\end{lemma}

\begin{proof}
We first recall that the $S$-transformation $\modS_\Zc$ for $\Q$ was given in~\eqref{eq:modS-Q}.
Using the explicit expression of $M$ from~\eqref{M-non-deg},  $\modS_\Zc$  on a central element $z$ becomes
\al{\label{Sz-Q}
\modS_\Zc(z) &=
	\sum\limits_{\substack{n,m,t_1,s_1,\\ \ldots,t_\np,s_\np=0}}^{1}  (-\beta^2)^{nm} \, 2^{\sum_{i}(s_i+t_i)} (-1)^{\sum_i(m t_i + s_i)}
\hat\Lambda_\coend\big(z\idem_n\K^m  \prod_{j=1}^\np(\tilde{\ff}^+_j)^{t_j}(\tilde{\ff}^-_j)^{s_j}\big) \\ \nonumber
&\qquad \times \idem_m S\big(\K^n\prod_{j=1}^\np(\ff^-_j)^{t_j}(\ff^+_j)^{s_j}\big) \gp
}

Our aim is to calculate $\modS_\Zc (z) $ on $\Za$. Recall the basis elements of $\Za$ in~\eqref{eq:Zproj-Zf-span}, they are all in $\Zc \idem_0$. Therefore the terms with $n=1$ are zero in the sum~\eqref{Sz-Q}. Then, we recall $\hat\Lambda_\coend$ from~\eqref{eq:intL-hat} and as a basis element of $\Za$ has even number of $\fpm_i$ the sum $\sum_j(t_j + s_j)$ should be an even number too, otherwise the value of $\hat\Lambda_\coend$ on the corresponding element is zero. 
	As $\sum_i(t_i + s_i)$  is even, 
	the factors of $\K$ in $\idem_0 \tilde{\ff}^\pm_j= \idem_0 \ff^\pm_j \K$ cancel (resulting in signs), and so the term $\hat\Lambda_\coend(\cdots)$ in \eqref{Sz-Q} is zero for $m=1$.
Combining all these observations, for $z\in\Za$ we have
\be\label{eq:Sf-Q}
	\modS_\Zc (z) = \sum\limits_{\substack{t_1,s_1,\\ \ldots,t_\np,s_\np=0}}^{1} \!\!\!\!\! 2^{\sum_{i}(s_i+t_i)} 
	(-1)^{\sum_i( s_i+(t_i + s_i)/2)}
\hat\Lambda_\coend\Big(z\idem_0\prod_{j=1}^\np(\ff^+_j)^{t_j}(\ff^-_j)^{s_j}\Big)\idem_0 S\Big(\prod_{j=1}^\np(\ff^-_j)^{t_j}(\ff^+_j)^{s_j}\Big)\ ,
\ee
	where the sign $(-1)^{\sum_i (t_i + s_i)/2}$ arises when cancelling the factors of $\K$ contained in $\tilde{\ff}^\pm_j$ (it helps to rewrite the product as in \eqref{eq:z-Zf} below to see this).
It is clear that we can write any basis element of $\Za$
from~\eqref{eq:Zproj-Zf-span} in the form
 \be\label{eq:z-Zf}
 z = \Big(\prod_{j=1}^{2k} \ff_{i_j}^{\eps_j}\Big)\Big(\prod_{j=1}^M \fp_{k_j}\fm_{k_j}\Big)\idem_0 \ ,
 \ee
	for some	$0\leq k,M \leq \np$ 
	and $i_j,k_j \in \{1,\dots,\np\}$
 such that $ i_1<\cdots<i_{2k}$, and $\eps_j \in \{ \pm \}$.
The essential part of the calculation is to analyse the coefficients $\hat\Lambda_\coend\big(z\idem_0\prod_{j}(\ff^+_j)^{t_j}(\ff^-_j)^{s_j}\big)$ in the sum~\eqref{eq:Sf-Q}.
For a given $z$ as in~\eqref{eq:z-Zf}, 
 the $2\np$-tuple $\{t_1,s_1,\ldots, t_\np,s_\np\}$ here is unique because $z\prod_{j=1}^\np(\ff^+_j)^{t_j}(\ff^-_j)^{s_j}$ should be proportional to the support of $\hat\Lambda_\coend$, recall~\eqref{eq:intL-hat}. We thus get 
 	that $\hat\Lambda_\coend(\cdots)$ is non-zero iff, for the same index choice as in \eqref{eq:z-Zf},
\be\label{eq:z-tjsj}
\prod_{j=1}^\np(\ff^+_j)^{t_j}(\ff^-_j)^{s_j} =
\Big(
	\prod_{j=1}^{2k}
\ff_{i_j}^{-\eps_j}\Big)\cdot \Big(\prod_{n=1}^{\np-M-2k} \fp_{l_n}\fm_{l_n}\Big) \ ,
\ee
where the $l_n$ take the $N-M-2k$ values in $\{1,\dots,\np\}$ which are different from all $i_j,k_j$ in \eqref{eq:z-Zf}.
Note that this set of $l_n$'s is unique.
After reordering the generators $\fpm_j$, the element within $\hat\Lambda_\coend(\dots)$ is $(-1)^{\half\sum_{j=1}^{2k}\eps_j}\idem_0\prod_{j=1}^{\np}\fp_j\fm_j$
\footnote{To get the sign, first we write $\prod_{j=1}^{2k} \ff_{i_j}^{\eps_j}\cdot \prod_{j=1}^{2k} \ff_{i_j}^{-\eps_j} = (-1)^k\prod_{j=1}^{2k} \ff_{i_j}^{\eps_j} \ff_{i_j}^{-\eps_j}$, 
where $\ff_{i_1}^{-\eps_1}$ is commuted past an odd number of fermions, so we get $-1$, then $\ff_{i_2}^{-\eps_2}$ is commuted past an even number of fermions, so we get $+1$, etc. -- this explains $(-1)^k$. Then to reorder $\ff_{i_j}^{\eps_j} \ff_{i_j}^{-\eps_j}$ we do nothing if $\eps_j=+1$ and get $-1$ otherwise. Thus we have to multiply by $(-1)^{\sum_{j=1}^{2k}\frac{1-\eps_j}{2}}$. This together with $(-1)^k$ gives the sign $(-1)^{\half\sum_{j=1}^{2k}\eps_j}$.}.
We also note that the $2\np$ tuple corresponding to~\eqref{eq:z-tjsj} gives
\be
 \sum_i(t_i+s_i) =  2(\np-M-k)\ , \qquad \sum_i s_i = \np -M -2k + \half\sum_{j=1}^{2k}(\eps_j+1)\ .
\ee
We thus get for~\eqref{eq:Sf-Q} the following expression, for $z$ as in~\eqref{eq:z-Zf},
\begin{align}\label{eq:Sf-fin}
\modS_\Zc (z) &= 2^{2(N-M-k)}\hat\Lambda_\coend\big(\idem_0\prod_{j=1}^{\np}\fp_j\fm_j\big) \idem_0 
S\Bigg(\Big(\prod\limits_{j=1}
^{2k}\ff_{i_j}^{\eps_j}\Big)\cdot
\Big(\prod_{n=1}^{\np-M-2k} \fm_{l_n}\fp_{l_n}\Big)\Bigg)\\
 & =   (-1)^{M} \nu  \beta^2 2^{\np-2(M+k)}\idem_0 \Big(\prod\limits_{j=1}
^{2k}\ff_{i_j}^{\eps_j}\Big)\cdot
\Big(\prod_{n=1}^{\np-M-2k} \fp_{l_n}\fm_{l_n}\Big)\ , \nonumber
\end{align}
	where the $l_n$ are as in \eqref{eq:z-tjsj}, and
where we used our assumption that $i_1<\cdots<i_{2k}$ and thus had to reorder $\ff_{i_j}^{\eps_j}$'s after the application of the antipode $S$. It is now clear that the image of $\Za$ under $\modS_\Zc$ is $\Za$ itself. 
We also note by a direct calculation that $\modS_\Zc^2 (z) = z$.

Our aim is now to check the expression in~\eqref{eq:S-action-on-Zf-from-Q} together with~\eqref{eq:STact-mat}. Recall that the bijection $\vartheta$ was defined in~\eqref{eq:U1...Un-to-Q0-map}. Then having $z$ as in~\eqref{eq:z-Zf}, its image under $\vartheta^{-1}$ has the form (of course here for a particular choice of  $i_j,k_j$)
\be
\vartheta^{-1}(z) = \idem_0\tensor\ldots \tensor \idem_0\tensor\idem_0 \ff^{\eps_1}_{i_1}\tensor \ldots \tensor \idem_0 \fp_{k_1}\fm_{k_1}\tensor \ldots  \tensor \idem_0 \gp
\ee
The action of each tensor component $\modS_{\Za}^i$, as defined in~\eqref{eq:STact-mat}, replaces $\idem_0$ in the $i$'th factor by $2\fp_i\fm_i\idem_0$ and $\fp_i\fm_i\idem_0$ by $-\half \idem_0$, while does nothing to the factors $\idem_0 \ff^{\eps_1}_{i_1}$. With this, it is now straightforward to  see the equality in~\eqref{eq:S-action-on-Zf-from-Q} using the final expression in~\eqref{eq:Sf-fin}.

For the $T$-transformation on $\Za$, we use the expression~\eqref{eq:ribinv} that can be written as
\be
\ribbon^{-1} \idem_0= \prod_{k=1}^\np(\one+2\fp_k\fm_k) \idem_0 \gp
\ee
It is therefore clear that $\modT_\Zc$ acts on $\Za$.
Recall  that the  injective linear map~\eqref{eq:U1...Un-to-Q0-map} restricts to 
   an isomorphism $\vartheta\bigl|_{U_+}: U_+ \to \Za$. 
Note that
 \be\label{eq:TZf-in-U+}
 \vartheta^{-1}(\ribbon^{-1}\idem_0) = (\idem_0 +2 \fp_1\fm_1\idem_0)\tensor  (\idem_0 + 2\fp_2\fm_2\idem_0)\tensor \ldots \tensor (\idem_0 +2 \fp_\np\fm_\np\idem_0)\gp
 \ee 
It is easy to check that, even though $\vartheta
\bigl|_{U_+}: U_+ \to \Za$ is itself not an algebra map, when restricted to the subalgebra of $\Q_0$ generated by $\fp_i\fm_i$, $1\leq i\leq \np$, it does become an algebra map.
In particular for $\ribbon^{-1}\idem_0$ we have $\vartheta^{-1}(\ribbon^{-1}z) = \vartheta^{-1}(\ribbon^{-1}\idem_0) \cdot \vartheta^{-1}(z)$ for all $z \in \Za$.
 Thus $\modT_{\Zc}|_{\Za}$ acts on $U_+$ by multiplication with \eqref{eq:TZf-in-U+}. From this, we immediately get~\eqref{eq:T-action-on-Zf-from-Q} with~\eqref{eq:STact-mat}. 

Finally, from the $\modT_\Zc(\idem_0)$ we see that $\modT_\Zc$ has Jordan blocks of maximum rank $\np+1$.
Indeed, recall that Jordan blocks of a given dimension behave under tensoring like irreducible $sl(2)$-modules of that same dimension. 
The matrix 
$\Ta^i$ in 
\eqref{eq:STact-mat}
 has one rank-2 Jordan block and two rank-1 blocks (use the basis $\{2\fp_i\fm_i\idem_0, \idem_0,\fm_i\idem_0,\fp_i\idem_0 \}$). Thus the maximal rank Jordan cell in the tensor product arises from from tensoring the $\np$ rank-2 
blocks.
 This
 corresponds to the tensor product of $\np$ fundamental $sl(2)$-modules, which decomposes into a direct sum with the largest irreducible summand being of dimension $\np+1$.  
\end{proof}

As a corollary of the two last lemmas we have the following theorem.

\begin{theorem} \label{thm:ST-from-Q}
The 	projective
$SL(2,\oZ)$-action  on the centre $\Zc$ of $\Q$ given by the linear maps $\modS_\Zc$ and $\modT_\Zc$
from~\eqref{eq:STact-ZA}-\eqref{eq:Tact-ZA} is
\be
\modS_\Zc = \Sb \oplus \Sa \gc 
\qquad
\modT_\Zc = \Tb \oplus \Ta \gc
\ee
	with the constituent maps as defined in Lemmas~\ref{lem:ST-on-Zproj} and~\ref{lem:Sa}.
We have $\modS_\Zc^2 = \id_{\Zc}$ and $\modT_\Zc$ has Jordan blocks of 
	ranks up to and including $\np+1$.
\end{theorem}

\begin{remark}\label{rem:SL2Z-decomp}
\mbox{}
\begin{enumerate}
	\item
The two $SL(2,\oZ)$ representations $\Za$ and $\Zb$ from Theorem~\ref{thm:ST-from-Q} are decomposed as follows.
We notice first that the 4-dimensional representation of the algebra generated by the matrices $\Sa^i$ and $\Ta^i$ from~\eqref{eq:STact-mat}  is decomposed onto irreducible representations as  $\oC^4 = \oC^2 \oplus \oC \oplus \oC$ where the first summand is even in terms of fermions (the adjoint action of $\K$ is $\mathrm{Id}$) while the last two summands are odd (the adjoint action of $\K$ is $-\mathrm{Id}$). With this and using~\eqref{eq:S-action-on-Zf-from-Q} and~\eqref{eq:T-action-on-Zf-from-Q} it is easy to determine
	a direct sum decomposition of $\Za$ in terms of the trivial and powers of the fundamental representations. 
For example in the case $\np=1$, $\Za= \oC^2$ is just 
the fundamental (projective) representation of $SL(2,\oZ)$. For $\np=2$, we have $\Za= \oC^2\otimes \oC^2 \oplus \oC^{\oplus 4}$. In general, we have
$$
\Za = \bigoplus_{k=0}^{\left\lfloor{\frac{\np}{2}}\right\rfloor} 4^k \binom{\np}{2k}  \bigl(\oC^2\bigr)^{\otimes (\np-2k)}\ ,
$$
where $\oC^2$ is the fundamental representation
	and where we set $(\mathbb{C}^2)^{\otimes 0} := \mathbb{C}$ to be the trivial representation.
	In particular, if $\np$ is even we have (at least) $2^\np$ invariants of the $SL(2,\oZ)$ action.
Finally, the decomposition of $\Zb$ was already discussed in Lemma~\ref{lem:ST-on-Zproj}.
We only note that every two pairs of matrices in~\eqref{eq:Sb}, for different choices of $\np$ and~$\beta$, say $(\np_1,\beta_1)$ and  $(\np_2,\beta_2)$, can be related 
	(projectively)
by a transition matrix if and only if $\np_1$ and $\np_2$ are of the same parity and $\beta_1^2=\beta_2^2$.
The transition matrix is diagonal with non-zero entries $(2^{\np_2-\np_1},1,1)$ if the two values of $\beta$ are equal, and it is the matrix
$$\footnotesize{\begin{pmatrix}
-2^{\np_2-\np_1} & 0 & 0 \\
0 & 0 & 1 \\
0 & 1 & 0
\end{pmatrix}}$$ if the two values of $\beta$ have different signs.
Therefore, $\Zb$ provides in total 4 isomorphism classes of 3-dimensional representations of $SL(2,\oZ)$.

\item
From the point of conformal field theory, modular invariant vectors as above appear as torus partition functions of holomorphic CFTs. 
Since partition functions are always traces, rather than pseudo-traces (as reviewed in Section~\ref{sec:SL2Z-compare} below), they correspond to invariant combinations of the elements $\bphi_{\XX^{\pm}_0}$, $\bphi_{\XX^{\pm}_1}$ of the centre $Z$. 
The above analysis implies that there is a projectively invariant combination of these iff $\beta = \pm 1$ (and so $\np$ is even), and it is spanned by the combination in \eqref{eq:inv-comb-char}.
Indeed, while $\Za$ does contain invariants, none of them is a combination of $\bphi_{\XX^{\pm}_0}$ (see \eqref{eq:Sbphi} and Lemmas~\ref{eq:phiV-explicit} and~\ref{lem:bchi-Q}), and so we only need to consider $\Zb$.
Combining with~\eqref{eq:bphiP} and rescaling, we get the spanning vector 
$2^{\np-1} \bphi_{\XX^{+}_{0}} +  2^{\np-1} \bphi_{\XX^{-}_{0}} + \beta  \bphi_{\XX^{\pm}_1}$
for the space of invariants in the span of $\bphi_{\XX^{\pm}_0}$, $\bphi_{\XX^{\pm}_1}$.
For a CFT partition function, all coefficients have to be non-negative integers, and so only $\beta=+1$ remains. The resulting invariant has been investigated in detail in \cite{Davydov:2016euo}, and it has been shown that there is a Lagrangian algebra in 
$\catSF(\np,\beta) \cong \rep \Q(\np,\beta)$ with this underlying character (cf.\ Remark~\ref{rem:HN-comments}\,(3)). Consequently, there exist holomorphic symplectic fermion CFTs at $\np = 0\; \mathrm{mod}\; 8$ (using the ribbon equivalence in Conjecture~\ref{conj:SF-equiv-RepVev-ribbon} below).
\end{enumerate}
\end{remark}

\section{Equivalence of the two projective \texorpdfstring{$SL(2,\oZ)$}{SL2Z}-actions}\label{sec:SL2Z-compare}

In this final section, we review the modular-group action on the symplectic fermion pseudo-trace functions and compare it with the $SL(2,\mathbb{Z})$ action computed in the previous section on the centre of $\Q$.
The main result of this paper is that these two actions are projectively equivalent (Theorem~\ref{thm:compare-Q-SF}).

\subsection{Modular properties of symplectic fermion pseudo-trace functions}\label{sec:modprop-SF}

Here,
we review the computation of the symplectic fermion pseudo-trace functions and of their modular properties carried out in \cite{Gainutdinov:2016qhz}.

We define two affine Lie super-algebras, $\widehat\hLie$ and $\widehat\hLie_\mathrm{tw}$, in terms of a symplectic $\oC$-vector space~$\hLie$ of dimension $2\np$ with symplectic form $(-,-)$. The underlying super-vector spaces are
\be
	\widehat\hLie = \hLie \ot \oC[t^{\pm1}]  \, \oplus \, \oC \hat k
	\quad , \quad
	\widehat\hLie_\mathrm{tw} = \hLie \ot t^{\frac12}\oC[t^{\pm1}] \, \oplus \, \oC \hat k \ ,
\ee
where $t^{\pm1}$ and $\hat k$ are parity-even, and $\hLie$ is parity-odd. For $u \in \hLie$ and $m \in \oZ$ (resp.\, $m \in \oZ + \frac12$), abbreviate $u_m := u \ot t^m$. The Lie super-bracket is given by taking $\hat k$ central and setting, for $u,v \in \hLie$ and $m,n \in \oZ$ (resp.\ $m,n \in \oZ + \frac12$),
\be
	[u_m,v_n] \,=\, (u,v) \,m\, \delta_{m+n,0} \,\hat k \ .
\ee
By convention for the bracket of a Lie super-algebra, 
this is an anti-commutator as $u_m$, $v_n$ are parity odd.
We refer to \cite{Runkel:2012cf} for more on $\widehat\hLie_{(\mathrm{tw})}$ and its representations.

To make the connection to Section~\ref{sec:SFdef}, choose a basis $\{ a_i, b_i \,|\, i = 1,\dots,\np \}$ of $\hLie$ which satisfies
\be
	(a_j,b_k) ~=~ \rmi \pi \delta_{j,k} \ . 
\ee
With this basis as generators, we have $\algGr = \Lambda(\hLie)$, that is, the $2^{2 \np}$-dimensional Gra\ss mann algebra defined in Section~\ref{sec:SF-abelian} is the exterior algebra of $\hLie$.
The reason to put the factor of $\rmi \pi$ in the above normalisation is that we will also work with a rescaled basis of $\hLie$, namely
\be\label{eq:alpha-basis-of-h}
	\alpha^1 = a_1 
	~,~~
	\alpha^2 = \tfrac{1}{\rmi \pi} b_1
	~,~ \dots 
	~,~
	\alpha^{2\np-1} = a_\np
	~,~~
	\alpha^{2\np} = \tfrac{1}{\rmi \pi} b_\np \ .
\ee
The $\alpha^i$ are a symplectic basis in the sense that $(\alpha^1,\alpha^2)=1$, etc. The basis $\{ \alpha^i \}$ is a natural choice when working with the affine Lie algebra $\widehat\hLie_{(\mathrm{tw})}$, and it is used e.g.\ in \cite{Runkel:2012cf} and \cite[Sec.\,6]{Gainutdinov:2016qhz}. 

For later use we recall the action of the Virasoro zero-mode $L_0$ on highest weight modules of $\widehat\hLie_{(\mathrm{tw})}$ from \cite[Rem.\,2.5\,\&\,2.7]{Runkel:2012cf}:
\begin{align}\label{eq:L0-action-on-h(tw)}
\text{$\widehat\hLie$-module :}\quad
&L_0 = \sum_{k=1}^{\np} (
\alpha^{2k}_0 \alpha_{0}^{2k-1}
- \alpha^{2k-1}_0 \alpha_{0}^{2k}) 
+ H \ ,
\\\nonumber
\text{$\widehat\hLie_{\mathrm{tw}}$-module :}\quad
&L_0 =  - \tfrac{\np}{8} + H_{\mathrm{tw}} \ . 
\end{align}
where
\begin{align}\label{eq:H(tw)-def}
H = \sum_{m \in \mathbb{Z}_{> 0}} \sum_{k = 1}^\np 
\big( \alpha_{-m}^{2k} \alpha^{2k-1}_{m} - \alpha_{-m}^{2k-1} \alpha^{2k}_{m}  \big) \ ,
\\ \nonumber
H_{\mathrm{tw}} =  \hspace{-.8em} \sum_{m \in \mathbb{Z}_{\ge 0} + \frac12}\sum_{k =1}^\np 
\big( \alpha_{-m}^{2k} \alpha^{2k-1}_{m} - \alpha_{-m}^{2k-1} \alpha^{2k}_{m} \big) \ .
\end{align}
One verifies that $H_{(\mathrm{tw})}$ acts as grading operator on a highest-weight module of $\widehat\hLie_{(\mathrm{tw})}$, assigning grade zero to the space of ground states.

The braiding on $\catSF$ was originally computed in \cite{Runkel:2012cf} with respect to
the basis in \eqref{eq:alpha-basis-of-h}
and, for example, for $X,Y \in \catSF_0$ the result was
\be
	c_{X,Y} = \sflip_{X,Y} \circ \exp\Big(- \rmi \pi 
	\sum_{k=1}^\np (\alpha^{2k} \otimes \alpha^{2k-1} - \alpha^{2k-1} \otimes \alpha^{2k}) \Big) \ ,
\ee
see \cite[Eqn.\,(6.1)]{Runkel:2012cf}. This formula, and many others, have factors of $\rmi \pi$, which can be absorbed into a suitably rescaled copairing. Indeed, with respect to the basis $\{a_k,b_k\}$, the above braiding produces the one presented in \eqref{eq:SF-braiding-sectors}.

Given a module $M \in \catSF_0 = \repsv\algGr$, we can use induction to construct an $\widehat\hLie$-module $\widehat M$ as follows.
We take $u_m$ for $u \in \hLie$, $m>0$, to act as zero on $M$, $\hat k$ to act by $1$, and the zero modes $(a_i)_0$ and $(b_i)_0$ to act as $a_i$ and $b_i$, respectively. Then
\be\label{eq:induced-M-def}
\widehat M := U(\widehat\hLie) \otimes_{U(\widehat\hLie_{\ge 0} \oplus \oC \hat k)} M \ ,
\ee 
where $U(-)$ denotes the universal enveloping algebra (in super-vector spaces) of a Lie super-algebra, and $\widehat\hLie_{\ge 0}$ is the subalgebra spanned by non-negative modes.
Similarly, for $V \in \catSF_1 = \svect$, we consider the induced $\widehat\hLie_{\mathrm{tw}}$-module 
\be
\widehat V := U(\widehat\hLie_{\mathrm{tw}}) \otimes_{U((\widehat\hLie_{\mathrm{tw}})_{> 0} \oplus \oC \hat k)} V \ ,
\ee
which is defined as above, except that in $\widehat\hLie_{\mathrm{tw}}$ there are no zero modes to take care of.

The symplectic fermion vertex operator super-algebra $\mathcal{V}$ is defined in \cite{Abe:2005}, see also \cite[Sec.\,3.1]{Davydov:2016euo} which uses the present notation (except that $\mathcal{V}$ is denoted by $\mathbb{V}(\hLie)$ there). The underlying super-vector space is $\mathcal{V} = \widehat{\one}$, for $\one \in \catSF$ the tensor unit as given in \eqref{eq:SF-simple}. Write $\mathcal{V}_\mathrm{ev}$ for the parity-even subspace of $\mathcal{V}$ -- this is a vertex operator algebra. 

It is shown in \cite{Abe:2005} (and stated this way in 
	\cite[Cor.\,6.4]{Gainutdinov:2016qhz}) 
that the functor of first inducing a $\widehat\hLie_{(\mathrm{tw})}$-module and then restricting to its even subspace defines a $\mathcal{V}_\mathrm{ev}$-module:

\begin{proposition}\label{prop:induction-evenpart-functor}
$\big( \widehat- \big)_{\mathrm{ev}}$ is a faithful $\oC$-linear functor $\catSF \to \rep \mathcal{V}_\mathrm{ev}$.
\end{proposition}

\begin{remark}
Write $\chi_1,\chi_2,\chi_3,\chi_4$ for the characters
 $\tr\,e^{2 \pi \rmi\tau (L_0 - c/24)}$
  of the irreducible representations of $\widehat{\one}$, $\widehat{\Pi\one}$, $\widehat{T}$, $\widehat{\Pi T}$, respectively. Their lowest conformal weights are, in the same order,
\be
h_1 = 0 \quad , \qquad
h_2 = 1 \quad , \qquad
h_3 = -\tfrac{\np}{8} \quad , \qquad
h_4 = -\tfrac{\np}{8} + \tfrac12 \ .
\ee
The explicit expressions for the characters $\chi_i$ are most easily given in terms of the $\np=1$ symplectic fermions characters (as usual, $q=e^{2 \pi i \tau}$)
\be\label{eq:symp-ferm-super-char}
\chi^{\np=1}_{ns,\pm}(\tau) = \Big( q^{\frac1{24}} \prod_{n=1}^\infty (1\pm q^n) \Big)^2 \ ,
\qquad
\chi^{\np=1}_{r,\pm}(\tau) =  \Big( q^{-\frac1{48}} \prod_{n=1}^\infty (1\pm q^{n-\frac12}) \Big)^2 \ .
\ee
Writing $\chi_{ns,\pm}(\tau) = (\chi^{\np=1}_{ns,\pm}(\tau))^{\np}$
and $\chi_{r,\pm}(\tau) = (\chi^{\np=1}_{r,\pm}(\tau))^{\np}$, 
we have \cite{Kausch:1995py,Gaberdiel:1996np,Abe:2005}
\begin{align}
\chi_1(\tau) &= \tfrac12\big( \chi_{ns,+}(\tau) + \chi_{ns,-}(\tau) \big) \ ,
\qquad &
\chi_2(\tau) &= \tfrac12\big( \chi_{ns,+}(\tau) - \chi_{ns,-}(\tau) \big)\ ,
\\ \nonumber 
\chi_3(\tau) &= \tfrac12\big( \chi_{r,+}(\tau) + \chi_{r,-}(\tau) \big)\ ,
\qquad &
\chi_4(\tau) &= \tfrac12\big( \chi_{r,+}(\tau) - \chi_{r,-}(\tau) \big) \ .
\end{align}
\end{remark}

Next we compute the image under $\big( \widehat- \big)_{\mathrm{ev}}$
 of the endomorphisms of the minimal projective generator $G_\catSF$ of $\catSF$ given in \eqref{eq:G-def}. 
This requires a bit of preparation. Define
\be\label{eq:SF-calc-specific-projgen}
	\mathcal{G} := \widehat{\algGr} \oplus \widehat{T} \ .
\ee
Note that we do not take the even part. Instead we consider the super-vector space on the RHS simply as a vector space. In terms of the functor $(\widehat-)_\mathrm{ev}$ this means
\be
	\mathcal{G} \cong 
	\big(\widehat 
G_{\catSF} \big)_\mathrm{ev} \ .
\ee
We will need this isomorphism explicitly. 
We first describe the isomorphism 
\be
\jmath
\colon\; \widehat\algGr \xrightarrow{\, \sim \,} \widehat{(\algGr\otimes \oC^{1|1})}_{\mathrm{ev}}
\ee
	of $\mathcal{V}_\mathrm{ev}$-modules.
Recall by~\eqref{eq:induced-M-def} that $\widehat \algGr = U(\widehat\hLie) \otimes_{U(\widehat\hLie_{\ge 0} \oplus \oC \hat k)} 
 \algGr$.
 We then define for homogeneous  $u\in  U(\widehat\hLie)$ and $f\in\algGr$:
 \be
 %\iota
 \jmath
 \colon\; u\tensor f \mapsto u\tensor f \tensor 
 \begin{cases}
 (1,0)\gc \quad \text{if}\; u\tensor f \; \text{even}\gc\\
 (0,1)\gc \quad \text{if}\; u\tensor f \; \text{odd}\gc
 \end{cases}
 \ee
 with the inverse map given by (note that $v$ has to be homogeneous too)
 \be\label{eq:iota-inv}
 %\iota
 \jmath^{-1}\colon\; u\tensor f \tensor v \mapsto 
 v_{|u| + |f|} \cdot
 u\tensor f\gc \qquad  v\in \oC^{1|1} \gp
 \ee
The isomorphism $\jmath_\mathrm{tw} \colon\; \widehat T \xrightarrow{\, \sim \,} \widehat{( T \otimes \oC^{1|1})}_{\mathrm{ev}}$ is defined analogously. 

\begin{lemma}\label{lem:image-of-E-in-EndVev}
The image 
	$\mathcal E$
under $\big( \widehat- \big)_{\mathrm{ev}}$ of 
	$\End_{\catSF}(G_\catSF)$
in $\End_{\mathcal{V}_\mathrm{ev}}(\mathcal{G})$ is generated by
\begin{itemize}
\item
the action of the zero modes $\alpha^i_0$, 
$i=1,\dots,2\np$, 
\item
$\widehat{\id_T}$, the idempotent corresponding to $\widehat T$, i.e.\  the identity map on $\widehat T$ and zero on $\widehat{\algGr}$, 
\item
the parity involution $\omega_{\mathcal{G}}$ on $\mathcal{G}$.
\end{itemize}
\end{lemma}

\begin{proof}
We first refer to
	Corollary~\ref{lem:End-of-projgen-SF} 
 where
	$\End_{\catSF}(G_\catSF)$
was described.  The functor $\big( \widehat- \big)_{\mathrm{ev}}$ maps  a morphism $g$ to $\id_{U(\widehat\hLie)}\tensor g$. 
We first calculate the image (under the functor) of action of $a\in\hLie$ from~\eqref{eq:SF-a-acts}:
\be
u\tensor f \tensor v \mapsto u\tensor f\cdot a \tensor \Pi(v) = (-1)^{|u|+|f|} a_0 \cdot u\tensor f \tensor \Pi(v)\gc
\ee
were the latter equality is by the definition of the induced representation~\eqref{eq:induced-M-def}.
Then composing with the isomorphism $\jmath^{-1}$ from~\eqref{eq:iota-inv} we get the corresponding endomorphism on~$\widehat\algGr$:
\be
 v_{|u| + |f|}
 u\tensor f \mapsto  v_{|a_0| + |u| + |f| +1} (-1)^{|u|+|f|} a_0 \cdot u\tensor f  = v_{|u| + |f|}  a_0 \cdot \omega_{\mathcal{G}}(u\tensor f)\gp
\ee
We finally note that the image of the $\KK$ generator is given by $\omega_{\mathcal{G}}$.
	Therefore the image of $\End_{\catSF}(G_\catSF)$ is generated by $\omega_{\mathcal{G}}$, by $a_0 \circ \omega_{\mathcal{G}}$ for $a\in\hLie$ (or, equivalently, by $a_0$), and by the idempotent for $\widehat {T}$. 
\end{proof}

We
note that the algebra from  Lemma~\ref{lem:image-of-E-in-EndVev} was also introduced  in~\cite{Arike:2011ab}
in order to describe pseudo-trace functions for $\mathcal{V}_\mathrm{ev}$. Let us briefly review this construction.
For a $k$-algebra $A$, the space of  {\em central forms on $A$} is defined as
\be
	C(A) ~=~ \big\{ \,\varphi : A \to k \,\big|\, \varphi(ab) = \varphi(ba) \text{ for all }a,b \in A \,\big\}  \ .
\ee
By definition, an element $\varphi \in C(A)$ induces a symmetric pairing $(a,b) \to \varphi(ab)$ on $A$. 

Recall the algebra $\mathcal E \subset\End_{\mathcal{V}_\mathrm{ev}}(\mathcal{G})$
 from Lemma~\ref{lem:image-of-E-in-EndVev}
and note that $\mathcal{G}$ is  an $\mathcal E$-module.
To each $\varphi \in C(\mathcal{E})$ one can assign a {\em pseudo-trace function} \cite{Arike:2011ab}
\be
\xi^\varphi_\mathcal{G} ~:~ \mathcal{V}_\mathrm{ev} \times \mathbb{H} \longrightarrow \oC \ .
\ee
Explicitly, $\xi^\varphi_\mathcal{G}$ is written as, 
for $v \in \mathcal{V}_\mathrm{ev}$, $\tau \in \mathbb{H}$,
\be\label{eq:ptf-def}
	\xi^\varphi_\mathcal{G}(v,\tau)	\,=\, t^\varphi_\mathcal{G}\big(\,o(v) \,e^{2 \pi \rmi \tau (L_0 - c/24)}\, \big) \ ,
\ee
where $t^\varphi_\mathcal{G} : \mathcal{G} \to \oC$ is a Hattori-Stallings trace, $o\colon \mathcal{V}_\mathrm{ev} \to \End_\oC(\mathcal{G})$ is 
the \textit{zero mode}
 linear map
 and $c = -2 \np$ is the central charge of $\mathcal{V}_\mathrm{ev}$. We refer to \cite{Arike:2011ab} and \cite[Sec.\,4]{Gainutdinov:2016qhz} for details (see also
  \cite[Prop.\,5.2]{Gainutdinov:2016qhz} on how the above definition relates to \cite{Arike:2011ab}).
Pseudo-trace functions provide a linear map 
$\xi_{\mathcal{G}}
: C(\mathcal{E}) \to C_1(\mathcal{V}_\mathrm{ev})$. 
For the choice of $\mathcal{G}$ given in \eqref{eq:SF-calc-specific-projgen},
this map is injective \cite[Thm.\,6.3.2]{Arike:2011ab}. 
 
The pseudo-trace functions are generalisations of the characters of VOA modules. 
Indeed, for each simple $\mathcal{V}_\mathrm{ev}$-module $\mathcal{M}$ there is a unique $\varphi_{\mathcal{M}} \in C(\mathcal{E})$ such that
\be
	\xi^{\varphi_\mathcal M}_\mathcal{G}(v,\tau)
	\,=\, \tr_\mathcal{M}\big(\,o(v) \,e^{2 \pi \rmi\tau (L_0 - c/24)}\, \big) \ ,
\ee
i.e.\ the pseudo-trace function for $\varphi_{\mathcal{M}}$ equals 
the usual trace over $\mathcal{M}$. 
	Existence of $\varphi_{\mathcal M}$ is proven in \cite[Prop.\,5.5]{Gainutdinov:2016qhz} which also provides 
	the simple expression  $\varphi_{\mathcal M} = \tr_{\tilde{\mathcal{M}}}(-)$ where $\tilde{\mathcal{M}} = \Hom_{\mathcal{V}_\mathrm{ev}}(\mathcal{G}, \mathcal{M})$ is the corresponding $\mathcal{E}$-module. 
	Uniqueness follows from injectivity of~$\xi_{\mathcal{G}}$. In particular, there is a unique $\varphi_{\mathcal{V}_\mathrm{ev}}$ corresponding to the vacuum character.

\newcommand{\VOA}{\mathcal{V}}
\newcommand{\Gc}{\mathcal{G}}

\begin{remark}\mbox{}
The general background for the above definition
 is the theory -- due to \cite{Zhu1996} -- of torus one-point functions 
	$C_1(\mathcal{V})$ for a vertex operator algebra $\mathcal{V}$. Torus one-point functions close under modular transformations, and 
 for a certain class of vertex operator algebras with semisimple representation theory, the space of torus one-point functions is spanned by characters \cite{Zhu1996}. These results were extended in \cite{Miyamoto:2002ar} to 
 certain vertex operator algebras with non-semisimple representation theory. It is proved in \cite{Miyamoto:2002ar}  that
in this case the torus one-point functions are spanned by a variant of pseudo-trace functions. 
Pseudo-trace functions in the sense of \cite{Miyamoto:2002ar} are hard to work with, and we use here an easier notion introduced in \cite{Arike:2011ab}. However, it is not known if those simpler functions span the torus one-point functions (even in the case of~$\mathcal{V}_\mathrm{ev}$).
\end{remark}

We need the following conjecture about $\xi_{\mathcal{G}}$:

\begin{conjecture}[{\cite[Conj.\,6.3.5]{Arike:2011ab}, \cite[Conj.\,5.8]{Gainutdinov:2016qhz}}]\label{conj:SF-CE=C1}
$\xi_{\mathcal{G}} : C(\mathcal{E}) \to C_1(\mathcal{V}_\mathrm{ev})$ is an isomorphism.\footnote{
	Note added in revision: Since the arXiv submission of this paper in 2017 and after the submission of the manuscript to the journal, the work \cite{ACe} has appeared, which implies Conjecture~\ref{conj:SF-CE=C1}. Namely, \cite[Prop.\,5.2]{ACe} states that $\dim C_1(\mathcal{V}_\mathrm{ev}) = 2^{2 \np-1}+3$, which agrees with the dimension of $Z(\mathcal{E})$ in \eqref{eq:ZE-sum-dec}. The third part of Proposition~\ref{prop:ST-Vev} below (which does not rely on the conjecture) gives $\dim C(\mathcal{E}) = \dim Z(\mathcal{E})$. Since $\xi_{\mathcal{G}}$ is injective, it is a bijection.}
\end{conjecture}

Under the above conjecture, the following statement
is verified for $\mathcal{V}_\mathrm{ev}$ in \cite[Sec.\,6]{Gainutdinov:2016qhz} and Lemma~\ref{lem:T-calc-SF} below
 by explicit calculation.
\begin{proposition}\label{prop:ST-Vev}
Assuming Conjecture~\ref{conj:SF-CE=C1},
we have
\begin{itemize}
\item
the action of the modular $S$- and $T$-transformation defines 
linear  isomorphisms
\be
	S_{\mathcal{V}_\mathrm{ev}} , T_{\mathcal{V}_\mathrm{ev}} \colon \; C(\mathcal{E}) \to C(\mathcal{E}) \ ;
\ee
\item
the element $\delta := S_{\mathcal{V}_\mathrm{ev}}(\varphi_{\mathcal{V}_\mathrm{ev}}) \in C(\mathcal{E})$ 
is given by
\be\label{eq:delta-form}
\delta(\alpha^1_0 \cdots \alpha^{2\np}_0 \omega_{\Gc}) = (2\pi)^{-\np} \ , \qquad \delta(\widehat{\id_T} \, \omega_{\Gc}) = 2^{-\np}\ ,
\ee
and it is zero on all other monomials in the generators of the algebra $\mathcal E$ (cf.\ Lemma~\ref{lem:image-of-E-in-EndVev}).

\item 
$\delta$
defines a non-degenerate pairing on $\mathcal{E}$ via $(f,g) \mapsto \delta(f \circ g)$;
the assignment $\hat\delta : Z(\mathcal{E}) \to C(\mathcal{E})$, $z \mapsto \delta(z \cdot (-))$ therefore is a linear isomorphism.
\end{itemize}
\end{proposition}

As a consequence of Proposition~\ref{prop:ST-Vev}, we obtain unique linear isomorphisms
\be\label{eq:ST-transport-CE-ZE}
 S_Z = \hat\delta^{-1} \circ S_{\mathcal{V}_\mathrm{ev}}\circ \hat\delta\gc \qquad
  T_Z= \hat\delta^{-1} \circ T_{\mathcal{V}_\mathrm{ev}}\circ \hat\delta \gp
  \ee

Next we give the explicit results for $\hat\delta(\varphi_{\mathcal{M}})$ and 
$S_Z,T_Z$.
Setting $\varphi_X := \varphi_{(\widehat X)_{\mathrm{ev}}}$, for $X \in \catSF$,
the calculation in \cite[Sec.\,6]{Gainutdinov:2016qhz} can be expressed as
\begin{align}\label{eq:hatdelta-of-varphi-simple}
	\hat\delta^{-1}( \varphi_{\one} ) &= (2 \pi)^\np \alpha^1_0 \cdots \alpha^{2\np}_0 (\omega_\mathcal{G} + 1) \ ,
	&
	\hat\delta^{-1}( \varphi_{T} ) &= 2^\np \,\widehat{\id_T} \,(\omega_\mathcal{G} + 1) \ ,
	&
	\\ \nonumber
	\hat\delta^{-1}( \varphi_{\Pi\one} ) &= (2 \pi)^\np \alpha^1_0 \cdots \alpha^{2\np}_0 (\omega_\mathcal{G} - 1) \ ,
	&
	\hat\delta^{-1}( \varphi_{\Pi T} ) &= 2^\np \, \widehat{\id_T} \, (\omega_\mathcal{G} - 1) \ .
\end{align}

The centre of $\mathcal{E}$ is computed in \cite{Arike:2011ab} to be
\be\label{eq:ZE-sum-dec}
	 Z(\mathcal{E}) = \Zb(\mathcal{E}) \oplus \Za(\mathcal{E}) \ ,
\ee
where $\Za(\mathcal{E})$ is spanned by even monomials in the zero modes $\alpha^i_0$, and $\Zb(\mathcal{E})$ has the basis
\be\label{eq:ZprojE-basis}
	\Zb(\mathcal{E}) = \mathrm{span}_\oC\big\{~
	\tfrac12 \hat\delta^{-1}(  \varphi_{\one} + \varphi_{\Pi\one} ) \,,\,
	\hat\delta^{-1}( \varphi_{T} )  \,,\, 
	\hat\delta^{-1}( \varphi_{\Pi T} ) ~\big\}
	\ .
\ee
The form of the centre also follows from Proposition~\ref{prop:cen-Q} via the equivalence in Theorem~\ref{SF-repQ-rib-eq}.

\begin{remark}\label{rem:explicit-modular}\mbox{}
		In our example of the symplectic fermions VOA $\VOA_\mathrm{ev}$ the pseudo-trace functions $\xi^\varphi_\mathcal{G}(v,\tau)$ were explicitly constructed in~\cite[App.\,A]{Gainutdinov:2016qhz}. 
		The first important observation~\cite[Lem.\,A.1]{Gainutdinov:2016qhz} is that  to separate all $\xi^\varphi_\Gc$ it is sufficient to evaluate them on elements $e^{-\tfrac1{12}L_2} \, w$ for $w \in \VOA_\mathrm{ev}$ of the form
		\be\label{eq:w-insertion-pt}
		w := w_{r,s} = \tilde\gamma_{l_1} \cdots \tilde\gamma_{l_r} \alpha^{j_1}_{-1} \cdots \alpha^{j_s}_{-1} \one		\ ,
		\ee
		where 
		$\tilde\gamma_{l} =  \alpha^{2l}_{-1} \alpha^{2l-1}_{-1} \in U(\widehat \hLie)$, \
		for  $l = 1,\dots,\np$ 
		and the set
		$\{j_1,\dots,j_s\}$ contains no pair $\{2l-1,2l\}$.
		That is, if $\xi^\varphi_\Gc(e^{-\tfrac1{12}L_2} \,w,\tau)=0$ for all $w$ of the form \eqref{eq:w-insertion-pt} and all $\tau \in \mathbb{H}$, then $\varphi=0$.
		Using
the characters in \eqref{eq:symp-ferm-super-char}, we can can express  
the pseudo-trace functions associated to the (only) three characters of indecomposable projective $\mathcal E$-modules:  
\begin{align*}
		\xi^{\varphi_{P_{\one}}}_\Gc(e^{-\tfrac1{12}L_2} \,w_{r,s},\tau)
		~&=~ 
2^{2\np-1}
 \, \delta_{s,0} \,(2 \pi i)^{-r}
		\tfrac{\partial}{\partial\tau_{l_1}} \cdots \tfrac{\partial}{\partial\tau_{l_r}}
		\prod_{j=1}^\np \chi^{\np=1}_{ns,+}(\tau_j)
		\Big|_{\tau_j = \tau} \ ,
		\nonumber\\
		\xi^{\varphi_T}_\Gc(e^{-\tfrac1{12}L_2} \,w_{r,s},\tau)
		~&=~ 
		\tfrac12 \, \delta_{s,0} \,(2 \pi i)^{-r}
		\tfrac{\partial}{\partial\tau_{l_1}} \cdots \tfrac{\partial}{\partial\tau_{l_r}}
		\Big\{ 
		\prod_{j=1}^\np \chi^{\np=1}_{r,+}(\tau_j)
		~+~ 
		\prod_{j=1}^\np \chi^{\np=1}_{r,-}(\tau_j)
		\Big\}\Big|_{\tau_j = \tau} \ ,
		\nonumber\\
		\xi^{\varphi_{\Pi T}}_\Gc(e^{-\tfrac1{12}L_2} \,w_{r,s},\tau)
		~&=~ 
		\tfrac12 \, \delta_{s,0} \,(2 \pi i)^{-r}
		\tfrac{\partial}{\partial\tau_{l_1}} \cdots \tfrac{\partial}{\partial\tau_{l_r}}
		\Big\{ 
		\prod_{j=1}^\np \chi^{\np=1}_{r,+}(\tau_j)
		~-~ 
		\prod_{j=1}^\np \chi^{\np=1}_{r,-}(\tau_j)
		\Big\}\Big|_{\tau_j = \tau} \ ,
		% \label{eq:pseudo-for-z123}
		\end{align*}
		where $\varphi_{P_{\one}} =  2^{2\np-1}(\varphi_{\one} + \varphi_{\Pi\one})$.
		The rest of  pseudo-trace functions are parametrised  by even monomials $z$ in the zero modes $\alpha^i_0$ and  given by	
		\begin{equation*}
		\xi^{\varphi_z}_\Gc(e^{-\tfrac1{12}L_2} \,w_{r,s},\tau) ~=~ 
		\tfrac12 \,(2 \pi i)^{-r}
		\tfrac{\partial}{\partial\tau_{l_1}} \cdots \tfrac{\partial}{\partial\tau_{l_r}}
		\Big\{\delta
\big(z \alpha^{j_1}_0 \cdots \alpha^{j_s}_0\omega_{\Gc} \, {\textstyle \prod_{j=1}^\np} (1 + 2 \pi i \gamma_j \tau_j)\big)
		{\textstyle \prod_{j=1}^\np} \, \chi^{\np=1}_{ns,-}(\tau_j)
		\Big\}\Big|_{\tau_j = \tau}
		\end{equation*}
for  the central form $\varphi_z(-) := \delta(z \cdot -)$ introduced above in~\eqref{eq:delta-form}, and we used the notation $\gamma_j = \alpha^{2j}_0 \alpha^{2j-1}_0$.
\end{remark}

The linear map $S_Z$ from \eqref{eq:ST-transport-CE-ZE} is computed explicitly in \cite[Cor.\,6.9]{Gainutdinov:2016qhz}. For $T_Z$ we give the explicit expression in the next lemma.

\begin{lemma}\label{lem:T-calc-SF}
For $z \in \Za(\mathcal E)$ we have
\be
	T_Z(z) 
	= 
	e^{2 \pi \rmi \np/12} \cdot 
	e^{2\pi \rmi \sum_{k=1}^\np 
	\alpha^{2k}_0  \alpha^{2k-1}_0  }z \ ,
\ee
while on $\Zb(\mathcal E)$ in the basis \eqref{eq:ZprojE-basis} we get the matrix representation
\be
	T_Z\big|_{\Zb(\mathcal E)}
	= e^{2 \pi \rmi \np/12} \cdot
	\begin{pmatrix}
	1 & 0 & 0 \\
	0 & e^{- \pi \rmi \np/4} & 0 \\
	0 & 0 & -e^{- \pi \rmi \np/4}
	\end{pmatrix}  \quad .
\ee
\end{lemma}

\begin{proof}
We start by computing $T_{\mathcal{V}_\mathrm{ev}}$. By definition, for all $\varphi \in C(\mathcal E)$,
\be
	\xi^\varphi_\mathcal{G}(v,\tau+1)
	\,=\, \xi^{T_{\mathcal{V}_\mathrm{ev}}(\varphi)}_\mathcal{G}(v,\tau)
	\ .
\ee
Recall the direct sum decomposition \eqref{eq:SF-calc-specific-projgen} of  $\mathcal{G}$. There are no $\mathcal{V}_\mathrm{ev}$-intertwiners between $\widehat\algGr$ and $\widehat T$, nor between the different parity subspaces of $\widehat T$, and so (as we have already seen in the explicit description above)
\be
	\mathcal E = \mathcal E_0 \oplus \mathcal E_1^+ \oplus \mathcal E_1^-
	\quad , \quad
	\mathcal E_0 = \End_{\mathcal{V}_\mathrm{ev}}(\widehat\algGr)
	~~,~~
	\mathcal E_1^+ = \oC \, \id_{(\widehat T)_\mathrm{ev}} 
	~~,~~
	\mathcal E_1^- = \oC \, \id_{(\widehat T)_\mathrm{odd}} \ .
\ee
{}From the definition of the pseudo-trace functions in \eqref{eq:ptf-def} and from the explicit form in \cite[Eqn.\,(6.16)\,\&\,App.\,A]{Gainutdinov:2016qhz},
see also their expressions in Remark~\ref{rem:explicit-modular},
 it is a straightforward exercise to show that
\be
	\xi^\varphi_\mathcal{G}(v,\tau+1)
	\,=\, t^\varphi_\mathcal{G}\big(\,e^{2 \pi \rmi (L_0 - c/24)}\,o(v) \,e^{2 \pi \rmi \tau (L_0 - c/24)}\, \big)
	\,=\, t^{\varphi'}_\mathcal{G}\big(\, o(v) \,e^{2 \pi \rmi \tau (L_0 - c/24)}\, \big) \ ,
\ee
where 
we used that the zero mode $o(v)$ commutes with $L_0$ and set
\be
	\varphi'(f) = \begin{cases}
	 e^{2 \pi \rmi (L_0 - c/24)} f & ; ~ f \in \mathcal E_0 
	 \\
	  e^{2 \pi \rmi (-\np/8 - c/24)} f & ; ~ f \in \mathcal E_1^+
	 \\
	 - e^{2 \pi \rmi (-\np/8 - c/24)} f & ; ~ f \in
	  \mathcal E_1^-
	\end{cases}
\ee
and, for $f \in \mathcal E_0$,
\be
	e^{2 \pi \rmi (L_0 - c/24)} f ~=~ \exp\Big\{ 2 \pi \rmi \Big( \ffrac12 \sum_{k=1}^\np \big(\alpha^{2k}_0  \alpha^{2k-1}_0 - \alpha^{2k-1}_0  \alpha^{2k}_0 \big) + \ffrac{\np}{12} \Big) \Big\} f \ .
\ee
The above computation makes use of the explicit form of $L_0$ as given in \eqref{eq:L0-action-on-h(tw)}, and of the fact that $H$ has  eigenvalues in $\mathbb{Z}_{\ge 0}$, and so vanish in $\exp(2 \pi \rmi( \cdots))$. On the other hand, $H_{\mathrm{tw}}$ has eigenvalues in $\mathbb{Z}_{\ge 0}$ on the even part of an $\widehat\hLie_{\mathrm{tw}}$-module, and in $\mathbb{Z}_{\ge 0} + \frac12$ on the odd part.
\end{proof}

Similar to the construction  in \eqref{eq:U1...Un-to-Q0-map} we let 
	$W_i \subset \mathcal{E}_0$ 
be the subalgebra with basis
\be\label{eq:Wi-basis}
	W_i = {\rm span}_\oC\big\{
	 \alpha^{2i}_0\alpha^{2i-1}_0 , 
	 \id_{\mathcal{E}_0}, 
	 \alpha^{2i}_0, 
	 \alpha^{2i-1}_0 
	 \big\} \ ,
\ee
and consider the linear map $\varpi : W_1\tensor \cdots \tensor W_\np \to 
	\mathcal{E}_0$ 
given by
\be
	\varpi(f_1 \otimes \cdots \otimes f_\np) := f_1 \cdots f_\np \ .
\ee
Let $W_+ \subset W_1\tensor \cdots \tensor W_\np$ be the subspace spanned by products with an even total number of $\alpha^k_0$'s. 
The map $\varpi$ restricts to an isomorphism $\varpi : W_+ \to \Za(\mathcal{E})$. Let $\sigma_k, \tau_k : W_k\to W_k$ be the linear maps whose matrix representations in the basis \eqref{eq:Wi-basis} are
\be\label{eq:sigma-tau-def}
\sigma_k = 
	\begin{pmatrix}
	0 & -2 \pi & 0 & 0 \\
	(-2 \pi)^{-1} & 0 & 0 & 0 \\
	0 & 0 & -\rmi & 0 \\
	0 & 0 & 0 & -\rmi \\
\end{pmatrix}
	\quad,\quad
\tau_k = 
	\begin{pmatrix}
	1 & 2 \pi \rmi & 0 & 0 \\
	0 & 1 & 0 & 0 \\
	0 & 0 & 1 & 0 \\
	0 & 0 & 0 & 1 \\
\end{pmatrix}
	\quad .
\ee

\begin{theorem}\label{thm:SL2Z-from-pseudotrace}
Assume Conjecture~\ref{conj:SF-CE=C1} holds.
The linear maps $S_Z,T_Z : Z(\mathcal E) \to Z(\mathcal E)$ preserve the direct sum decomposition~\eqref{eq:ZE-sum-dec}:
\be
	S_Z = S_{\Zb} \oplus S_{\Za}
	\quad , \quad
	T_Z = T_{\Zb} \oplus T_{\Za} \ .
\ee
The linear endomorphisms 
$S_{\Zb}, T_{\Zb}$ of $\Zb(\mathcal E)$ are given by the matrices, with respect to the basis \eqref{eq:ZprojE-basis},
\be
	S_{\Zb} = 
	\begin{pmatrix}
	0 & 2^\np & -2^\np \\
	2^{-\np-1} & 2^{-1} & 2^{-1} \\
	-2^{-\np-1} & 2^{-1} & 2^{-1}
	\end{pmatrix} 
	\quad , \quad
	T_{\Zb} = e^{2 \pi \rmi \np/12} \cdot
	\begin{pmatrix}
	1 & 0 & 0 \\
	0 & e^{- \pi \rmi \np/4} & 0 \\
	0 & 0 & -e^{- \pi \rmi \np/4}
	\end{pmatrix}  \quad .
\ee
The linear endomorphisms $S_{\Za},T_{\Za}$ of $\Za(\mathcal E)$ are given by
\begin{align}\label{eq:Zf(E)-ST-via-sig-tau}
	S_{\Za}
	~&=~
	\varpi \circ \Big(  \big(
	\sigma_1 \tensor \cdots \tensor\sigma_\np \big)\big|_{W_+}
	\Big)
	\circ \varpi^{-1} \ ,
	\\ \nonumber
	T_{\Za}
	~&=~ e^{2 \pi \rmi \np/12} \cdot 
	\varpi \circ \Big(  \big(
	\tau_1 \tensor \cdots \tensor\tau_\np \big)\big|_{W_+}
	\Big)
	\circ \varpi^{-1} \ .
\end{align}
\end{theorem}

\begin{proof}
The expression for $S_Z$ is computed in \cite[Cor.\,6.9]{Gainutdinov:2016qhz}. 
The expression for $T_Z$ is immediate from Lemma~\ref{lem:T-calc-SF}.
\end{proof}

We conclude this section by a conjectural refinement 
of the $\Cb$-linear	inclusion
in Proposition~\ref{prop:induction-evenpart-functor} to a ribbon equivalence which we need in the next section.
For details on how to define the necessary 
 coherence isomorphisms we refer to  \cite[Conj.\,7.4]{Davydov:2016euo}, and for the specific value of $\beta$ to \cite[Thm.\,6.4]{Runkel:2012cf}.

\begin{conjecture}[{\cite[Conj.\,7.4]{Davydov:2016euo}}]\label{conj:SF-equiv-RepVev-ribbon}
For the choice
\be\label{eq:beta-SF-value}
	\beta = e^{- \rmi \pi \np / 4 } \ ,
\ee
$\big( \widehat- \big)_{\mathrm{ev}}$ is an equivalence of $\oC$-linear ribbon categories.
\end{conjecture}

\subsection{Comparison of \texorpdfstring{$SL(2,\oZ)$}{SL2Z}-actions}\label{sec:compare-to-SF}

To compare the results in Theorems \ref{thm:ST-from-Q} and \ref{thm:SL2Z-from-pseudotrace} we first need to correct a mismatch in conventions for the ribbon twist and then relate $Z(\Q)$ from \eqref{eq:cQ} to $Z(\mathcal E)$ from \eqref{eq:ZE-sum-dec}.

We start by remarking that in \cite{Lyubashenko:1995,Lyubashenko:1994tm} -- which we use to compute the 
$SL(2,\oZ)$-action on $Z(\Q)$ -- the convention is that the $T$ generator acts by composing with the ribbon twist~$\theta$. 
However, the modular $T$-transformation acts by composition with $\theta^{-1}$.\footnote{The reason for this mismatch is that by the convention chosen in \cite{Runkel:2012cf}, following e.g.\ \cite{tft1}, the ribbon twist acts by $e^{-2 \pi \rmi L_0}$ while 
for $\rep \mathcal{V}_\mathrm{ev}$
the modular $T$-action is  (up to a constant) given by $e^{2 \pi \rmi L_0}$. One of course could redefine the Lyubashenko's $SL(2,\oZ)$-action, but we find it more convenient to adapt the conventions from~\cite{Runkel:2012cf} to the present context.}
Below we will show that the agreement with the categorical $T$-transformation can be achieved by changing the value of $\beta$ in Conjecture~\ref{conj:SF-equiv-RepVev-ribbon} to its inverse.

According to Conjecture~\ref{conj:SF-equiv-RepVev-ribbon}, $\big(\,\rep \mathcal{V}_\mathrm{ev} \,, \theta = e^{-2 \pi \rmi L_0}\,\big) \cong 
\catSF\big(\np\,,\,\beta = e^{- \rmi \pi \np / 4 }\,\big)$ as ribbon categories. This is equivalent to
\be\label{eq:RepV-theta-SFrev}
	\big(\,\rep \mathcal{V}_\mathrm{ev} \,, \theta = e^{2 \pi \rmi L_0}\,\big) ~\simeq~ \catSF\big(\np\,,\,\beta = e^{- \rmi \pi \np / 4 }\,\big)^{\mathrm{rev}} \ .
\ee
Here, given a ribbon category $\mathcal{C}$,  $\mathcal{C}^\mathrm{rev}$ denotes the 
{\em reversed ribbon category}, where braiding and twist are replaced by their inverses. 
On the 
LHS of \eqref{eq:RepV-theta-SFrev} we now have the convention for~$\theta$ that matches our quasi-Hopf computation
of the $T$-transformation.
 To reformulate the RHS, 
we need the following lemma.

\begin{lemma}\label{lem:SFrev-SF}
For all $\np,\beta$ we have $\catSF(\np,\beta)^\mathrm{rev} \simeq \catSF(\np,\beta^{-1})$ as ribbon categories.
\end{lemma}

\begin{proof}
For the sake of the proof we will write $\catSF(\np,\beta;\Lambda^\mathrm{co}_\algGr(\beta),C)$, where $\Lambda^\mathrm{co}_\algGr(\beta)$ is the cointegral in \eqref{eq:SF-coint-def} and $C$ refers to the copairing in \eqref{def:copair}.\footnote{
More generally, the ribbon category $\catSF$ is constructed from a 
symplectic vector space $\mathfrak{h}$, a cointegral and a choice of $\beta$, see \cite[Sec.\,5.2]{Davydov:2012xg} for details. $C$ is the copairing for the symplectic form. Including the cointegral and the copairing in the notation makes these choices explicit.}
We stress the $\beta$-dependence in the notation $\Lambda^\mathrm{co}_\algGr(\beta)$ as our normalisation convention for the cointegral is $\beta$-dependent.
By \cite[Prop.\,4.12]{Davydov:2012xg} there is a ribbon equivalence
\be\label{eq:SFrev-SF_aux1}
	\catSF(\np,\beta;\Lambda^\mathrm{co}_\algGr(\beta),C)^\mathrm{rev}
	~\simeq~ 
	\catSF(\np,\beta^{-1};\Lambda^\mathrm{co}_\algGr(\beta),-C) 
 \ ,
\ee
whose underlying functor is the identity (with non-trivial coherence isomorphisms). In \cite[Prop.\,4.12]{Davydov:2012xg} the cointegral stays the same, hence $\Lambda^\mathrm{co}_\algGr(\beta)$ appears also on the RHS instead of $\Lambda^\mathrm{co}_\algGr(\beta^{-1})$.

It remains to show that on the RHS of \eqref{eq:SFrev-SF_aux1} one can replace
$\Lambda^\mathrm{co}_\algGr(\beta)$ by $\Lambda^\mathrm{co}_\algGr(\beta^{-1})$ 
and $-C$ by $C$. 
Let $\varphi : \algGr \to \algGr$ be the Hopf algebra isomorphism given on the generators in \eqref{eq:aibi-relation} as $\varphi(a_i)=a_i$, $\varphi(b_i)=-b_i$, $i=1,\dots,\np$. The isomorphism $\varphi$ satisfies
\be
	\Lambda^\mathrm{co}_\algGr(\beta) \circ \varphi = (-1)^\np \Lambda^\mathrm{co}_\algGr(\beta) = 
	\Lambda^\mathrm{co}_\algGr(\beta^{-1})
	\quad , \qquad
	(\varphi \otimes \varphi)(C) = -C \ .
\ee
We can now apply \cite[Prop.\,4.11]{Davydov:2012xg} to conclude that
\be\label{eq:SFrev-SF_aux2}
	\catSF(\np,\beta^{-1};\Lambda^\mathrm{co}_\algGr(\beta),-C) 
	~\simeq~ 
	\catSF(\np,\beta^{-1};\Lambda^\mathrm{co}_\algGr(\beta^{-1}),C) 
\ee
as ribbon categories. The underlying functor of this equivalence is the identity functor in $\catSF_1$, and it is given by $\varphi^*$, the pullback of representations along $\varphi$, on $\catSF_0$.
Combining \eqref{eq:SFrev-SF_aux1} and \eqref{eq:SFrev-SF_aux2} proves the statement of the lemma.
\end{proof}

\begin{remark}
Thanks to Lemma \ref{lem:SFrev-SF} and Theorem~\ref{SF-repQ-rib-eq}, the equivalence in Conjecture~\ref{conj:SF-equiv-RepVev-ribbon} can be restated as
\be\label{eq:relation-SF-RepV-with-exp2piiL0}
	\big(\,\rep \mathcal{V}_\mathrm{ev} \,, \theta = e^{2 \pi \rmi L_0}\,\big) ~\simeq~ \catSF\big(
\np
\,,\,\beta = e^{\rmi \pi \np / 4 }\,\big)
~\simeq~\rep  \Q\big(\np,\beta = e^{\rmi \pi \np / 4}\big)\ . 
\ee
\end{remark}

Denote by $\mathcal{J} : \rep \Q(\np,\beta) \to \catSF(\np,\beta)$
 the ribbon equivalence from Proposition~\ref{SF-repQ-rib-eq} and constructed
in Appendices~\ref{app:SF-S} and~\ref{app:Q-S},
see also the inverse equivalence in Section~\ref{app:rib-equiv-QSF}. 
We need to compute the following chain of algebra isomorphisms for $\beta = e^{\rmi \pi \np/4}$:
\be
\Upsilon ~:~ 
	Z(\Q)
	\overset{(1)}{\longrightarrow}
	\End(Id_{\rep \Q})
	\overset{(2)}{\longrightarrow}
	\End(Id_{\catSF(\beta)})
	\overset{(3)}{\longrightarrow}
	\End(Id_{\catSF(\beta^{-1})^{\mathrm{rev}}})
	\overset{(4)}{\longrightarrow}
	Z(\mathcal E) \ ,
\ee
and use this to compare the results of Theorems~\ref{thm:ST-from-Q} and~\ref{thm:SL2Z-from-pseudotrace}.
Let $z \in Z(\Q)$.
Arrow (1) is simply given by acting with $z$ on a module. 
Arrow (2) maps $\eta$ to the unique $\eta'$ such that
\be
	\eta'_{\mathcal J(X)} = \mathcal J(\eta_X)
	\qquad \text{for all}\quad X \in \rep \Q \ .
\ee
Arrow (3) is described in Lemma \ref{lem:SFrev-SF} and arrow (4) maps $\psi$ to 
$(\widehat{\psi_{G}})_\mathrm{ev}$.

We compute the composition of these maps separately for $Z(\Q_0)$ and $Z(\Q_1)$. Let us start with $z \in Z(\Q_0)$.
The functor $\funQSF:  \catSF(\np,\beta) \to \rep \Q(\np,\beta)$
which is inverse to $\mathcal J$ is given explicitly in Appendix~\ref{app:rib-equiv-QSF}. 
Conversely, the functor $\mathcal J$, as $\oC$-linear functor, is $\mathcal J = \mathcal E \circ \funQS$, where $\mathcal E$ 
is given in the proof of Proposition~\ref{prop:D-is-C-lin-equiv} and $\funQS$ in the proof of Proposition~\ref{prop:G-equiv}.
Then for $X \in \catSF_0$, $\eta'_X$ is given by the action of $z$,
 where now $\fp_k$ acts as $a_k$, $\fm_k$ acts as $b_k$
  and $\K$ acts as parity involution $\omega_X$.
This gives arrow (2). Arrow (3) is pullback of representations (see the proof of Lemma \ref{lem:SFrev-SF}) with the result that $\fp_k$ still acts as $a_k$ but now $\fm_k$ acts as $-b_k$.
Finally, for arrow (4) the relation between $a_k,b_k$ and the zero modes $(\alpha^j)_0$ is given in \eqref{eq:alpha-basis-of-h}.
Altogether, for $\nu_j,\eps_j \in \{ 0,1\}$, $\sum_{j=1}^\np (\nu_j + \eps_j)$ even,
\begin{align}
\Upsilon\big(  \idem_0 (\fp_1)^{\nu_1}(\fm_1)^{\eps_1} \cdots (\fp_\np)^{\nu_\np}(\fm_\np)^{\eps_\np} \big)
&=
(a_1)_0^{\nu_1} (-b_1)_0^{\eps_1} \cdots (a_\np)_0^{\nu_\np} (-b_\np)_0^{\eps_\np} \ ,
\\ \nonumber
&= (-\pi \rmi)^{\eps_1+\cdots+\eps_\np} \cdot 
(\alpha^1_0)^{\nu_1} (\alpha^2_0)^{\eps_1} \cdots (\alpha^{2\np-1}_0)^{\nu_\np} (\alpha^{2\np}_0)^{\eps_\np} \ ,
\end{align}
where the factors on the RHS are zero modes of $\widehat\hLie$ acting on $\mathcal{G}$. Furthermore,
\begin{align}
\Upsilon\big( \K \idem_0 \fp_1 \fm_1 \cdots \fp_\np\fm_\np \big)
&=
(a_1)_0 (-b_1)_0 \cdots (a_\np)_0 (-b_\np)_0 \, \omega_\mathcal{G} \ ,
\\ \nonumber
&= (-\pi \rmi)^{\np} \cdot 
\alpha^1_0 \alpha^2_0 \cdots \alpha^{2\np}_0 \, \omega_\mathcal{G} \ .
\end{align}
For $Z(\Q_1)$ one quickly finds that
\begin{align}
\Upsilon(\idem_1^+) &= \id_{(\widehat T)_\mathrm{ev}} = \tfrac12 \cdot \widehat{\id_T} \circ (\id+\omega_\mathcal{G}) \ ,
\\ \nonumber
\Upsilon(\idem_1^-) &= \id_{(\widehat T)_\mathrm{odd}} = \tfrac12 \cdot \widehat{\id_T} \circ (\id-\omega_\mathcal{G}) \ .
\end{align}

After these preparations, we can finally compare the 
elements $\bphi_V$ from \eqref{eq:bphi-X} to \eqref{eq:hatdelta-of-varphi-simple} and the
action of the generators of $SL(2,\oZ)$ as given in Theorems~\ref{thm:ST-from-Q} and~\ref{thm:SL2Z-from-pseudotrace}.

\begin{theorem}\label{thm:compare-Q-SF}
If one chooses the normalisation constant of the integral for the coend $\coend$ in \eqref{eq:intL-hat} to be $\nu=1$, then
\begin{align} \label{eq:bphi-varphi-compare}
	\Upsilon(\bphi_{\XX^{+}_{0}}) &= \hat\delta^{-1}( \varphi_{\one} ) \ ,
	&
	\Upsilon(\bphi_{\XX^{+}_{1}}) &= \hat\delta^{-1}( \varphi_{T} )\ ,
	\\ \nonumber
	\Upsilon(\bphi_{\XX^{-}_{0}}) &= \hat\delta^{-1}( \varphi_{\Pi\one} )\ ,
	&
	\Upsilon(\bphi_{\XX^{-}_{1}}) &= \hat\delta^{-1}( \varphi_{\Pi T} )\ .
\end{align}
Furthermore,
$S_Z \circ \Upsilon = \Upsilon \circ \modS_\Zc$ and
$T_Z \circ \Upsilon = e^{2 \pi \rmi \np/12} \cdot  \Upsilon \circ \modT_\Zc$.
\end{theorem}

\begin{proof}
That the identities in \eqref{eq:bphi-varphi-compare} hold is immediate from comparing \eqref{eq:bphi-X} and \eqref{eq:hatdelta-of-varphi-simple} via the explicit form of $\Upsilon$ given above.
For example,
\be
	\Upsilon(\bphi_{\XX^{+}_{0}})
	=
	\nu \, 2^{\np}  \beta^2  
	(-\pi \rmi)^{\np} 
	(\omega_\mathcal{G}+\id) \alpha^1_0 \cdots \alpha^{2\np}_0 \ ,
\ee
which is equal to $\hat\delta( \varphi_{\one} )$ in \eqref{eq:hatdelta-of-varphi-simple} since $\nu=1$ and $\beta$ is fixed as in \eqref{eq:beta-SF-value}.

The basis \eqref{eq:Zproj-Zf-span} gets mapped to
\be
\big\{~  \bphi_{\PP^{+}_0} \, , \,
	 \bphi_{\XX^{+}_{1}}  \,,\, 
	\bphi_{\XX^{-}_{1}} ~\big\} 
\overset{\Upsilon}\longmapsto
\big\{~ 2^{2\np} \cdot 
	\tfrac12 \hat\delta^{-1}(  \varphi_{\one} + \varphi_{\Pi\one} ) \,,\,
	\hat\delta^{-1}( \varphi_{T} )  \,,\, 
	\hat\delta^{-1}( \varphi_{\Pi T} ) ~\big\}
	\ ,
\ee
which differs from \eqref{eq:ZprojE-basis} by a coefficient
 in the first basis vector. Taking this into account one arrives at 
$S_{\Zb} \circ \Upsilon = \Upsilon \circ \modS_{\Zb}$ and
$T_{\Zb} \circ \Upsilon = e^{2 \pi \rmi \np/12} \cdot  \Upsilon \circ \modT_{\Zb}$.

The basis for $U_k$ used in Lemma~\ref{lem:Sa} is transported to a basis of $W_k$ (via $\varpi^{-1} \circ \Upsilon \circ \vartheta$) as
\be
	\big\{\, \idem_0 \,,\, 
	\fm_k\idem_0 \,,\, 
	\fp_k\idem_0 \,,\, 
	\fp_k\fm_k\idem_0 \,\big\}
	\longmapsto
	\big\{\, \id \,,\, 
	-\pi\rmi \, \alpha_0^{2k} \,,\, 
	\alpha_0^{2k-1} \,,\, 
	-\pi\rmi \, \alpha_0^{2k-1} \alpha_0^{2k} \,\big\} \ .
\ee
This differs from the basis used in \eqref{eq:Wi-basis}
in the order of basis factors and by coefficients. 
Note that for the present choice of $\beta$, the factor $\beta^2$ in \eqref{eq:S-action-on-Zf-from-Q} is equal to 
	$\rmi^\np$, 
which can be distributed over the individual factors of $\modS_{\Za}^k$. 
One then verifies that the above change of basis transports 
	$\rmi \, \modS_{\Za}^k$ 
to $\sigma^k$ and $\modT_{\Za}^k$ to $\tau^k$ as given in \eqref{eq:sigma-tau-def}. 
Comparing \eqref{eq:S-action-on-Zf-from-Q} and \eqref{eq:Zf(E)-ST-via-sig-tau}, we see that this
proves
$S_{\Za} \circ \Upsilon = \Upsilon \circ \modS_{\Za}$ and
$T_{\Za} \circ \Upsilon = e^{2 \pi \rmi \np/12} \cdot  \Upsilon \circ \modT_{\Za}$.
\end{proof}

This also proves Theorem~\ref{thm:intro-main} from the introduction,
which states that the linear $SL(2,\mathbb{Z})$-action 
$\pi_{\mathcal{E}}$
on $Z(\mathcal{E})$ and the projective $SL(2,\mathbb{Z})$-action $\pi_{\Q}$ on $Z(\Q)$ agree projectively. 
Namely, there is a function $\gamma: SL(2,\mathbb{Z}) \to \mathbb{C}^\times$ 
such that the linear isomorphism $\Upsilon : Z(\Q) \to Z(\mathcal E)$ satisfies
\be
	\pi_{\mathcal{E}}(g) \circ \Upsilon = \gamma(g) \cdot \Upsilon \circ \pi_{\Q}(g) 
	\qquad  \text{ for all } g \in SL(2,\mathbb{Z}) \ .
\ee

%%%%%%%%%%%%%%%%
\appendix
%%%%%%%%%%%%%%%%

\section{Equivalence between \texorpdfstring{$\catSF$}{SF} and \texorpdfstring{$\repS$}{Rep(S)}} \label{app:SF-S}

We give here
 the first part of the proof of Lemma~\ref{lem-trans-QSF}.
Fix $\np \in \mathbb{N}$ and $\beta \in \oC$ with $\beta^4 = (-1)^\np$, see \eqref{eq:beta-param}. In this section we introduce a quasi-bialgebra $\Salg = \Salg(\np,\beta)$ in $\svect$, define a braiding on its category $\repS$ of finite-dimensional representations in $\svect$, and show that $\repS$ is equivalent to $\catSF(\np,\beta)$ as a braided monoidal category.

%%%%%%%%%%%%%%%%%%%%%%%%

\subsection{A quasi-bialgebra in \texorpdfstring{$\svect$}{Svect}} \label{dfn:algS}

The unital associative algebra $\Salg=\Salg(\np,\beta)$ over $\oC$ has generators $\xpm_i$, $i=1,\ldots,\np$ and $\LL$, subject to the relations
\begin{align} \label{def:SN}
	\xpm_i \LL=\LL \xpm_i \gcg \{\xp_i,\xm_j\}=\delta_{i,j}\tfrac12 (1-\LL)\gcg \{\xpm_i,\xpm_j\}=0\gcg \LL^2=\one\gp
\end{align}
We have $\dim\Salg = 2^{2\np+1}$. Next we turn $\Salg$ into an algebra in $\svect$ by giving it a $\oZ_2$-grading
such that $\xpm_i$ have odd degree and $\LL$ has even degree. 

Define the central idempotents
\begin{align} \label{def:e}
	\idem_0\coloneqq \tfrac12 (1+\LL)\gc \qquad \idem_1\coloneqq \tfrac12 (1-\LL)\gp
\end{align}
Using these, $\Salg$ decomposes as
\begin{align} \label{eq:Salg-dec}
	\Salg =\Salg_0\oplus \Salg_1 
	\qquad , \qquad \Salg_0 \coloneqq \idem_0 \Salg
	\quad , \quad \Salg_1 \coloneqq \idem_1 \Salg \ .
\end{align}
From the defining relations it is immediate that $\Salg_0$ is a $2^{2\np}$-dimensional Gra\ss{}mann algebra and $\Salg_1$ is a $2^{2\np}$-dimensional Clifford algebra.

In the following we will often use the tensor product $A \ot B$ of algebras $A,B$ in $\svect$. As a super-vector space, $A \ot B$ is the tensor product $A \ot_{\svect} B$ of the underlying super-vector spaces. The unit is $1 \ot 1$ and the product $\mu_{A \ot B}$ includes parity signs:
\be\label{eq:super-alg-tensor}
	\mu_{A \ot B}
	:=
	(\mu_A \ot \mu_B) \circ (\id_A \ot \sflip_{B,A} \ot \id_B)
	~~:~~ A \ot B \ot A \ot B \longrightarrow A \ot B \ ,
\ee
where $\sflip$ is the symmetric braiding in $\svect$, see \eqref{sflip}.
In terms of homogeneous elements, this reads $(a \ot b) \cdot (a' \ot b') = (-1)^{|b||a'|} (aa') \ot (bb')$.

The notion of a quasi-bialgebra in $\vect$ is given e.g.\ in Definition~I:\ref*{I-def:quasi-Hopf} (that definition is for quasi-Hopf algebras -- just omit the antipode and the corresponding conditions). In $\svect$, the definition is basically the same (and works in general symmetric monoidal categories):

\begin{definition}\label{def:quasi-bialg-symmoncat}
A quasi-bialgebra in $\svect$ is a tuple $(A,\eps,\Delta,\Phi)$, where
\begin{itemize}
\item $A \in \svect$ is a unital associative algebra, such that the unit $\one \in A$ is even and the product respects the $\mathbb{Z}_2$-grading.
\item $\eps : A \to \mathbb{C}^{1|0}$ and $\Delta : A \to A \ot A$ are even linear maps and algebra homomorphisms. (The algebra structure on $A \ot A$ involves the symmetric braiding in $\svect$ as described in \eqref{eq:super-alg-tensor}.)
\item $\Phi \in A \ot A \ot A$ is an even element.
\end{itemize}
These data is subject to the conditions I:(\ref*{I-eq:eps-Delta})--I:(\ref*{I-eq:3-cocycle}) in Definition~I:\ref*{I-def:quasi-Hopf}, with products in $A \ot A \ot A$ involving parity signs as in \eqref{eq:super-alg-tensor}.
\end{definition}

In presenting the quasi-bialgebra structure for $\Salg$, we first list the data and will then prove in Proposition~\ref{prop:S-quasiHopf} below that these indeed define a quasi-bialgebra in $\svect$.

The (non-coassociative) coproduct $\copS\colon \Salg\to \Salg\otimes \Salg$ and counit  $\epsilon\colon \Salg \to \oC$ are given by
\begin{align} \label{S:delta}
\Delta^{\Salg}(\xpm_i) &= \xpm_i \otimes \one + (\idem_0-\rmi\idem_1)\otimes \xpm_i \pm \rmi\idem_1\otimes \idem_1(\xp_i-\xm_i) \gc 
& \eps^\Salg(\xpm_i) &= 0 \gc \\ \nonumber
	\Delta^{\Salg}(\LL)&=\LL\otimes \LL \gc & \eps^\Salg(\LL) &= 1 \ .
\end{align}
It is straightforward to check that $\Delta^\Salg$ is well-defined on $\Salg$, i.e.\ that the above definition in terms of generators is compatible with the relations in \eqref{def:SN}.

The coassociator $\Sas \in \Salg\tensor\Salg\tensor\Salg$ is given by
\begin{align}\label{eq:Sas}
\Sas = & ~ \idem_0\tensor\idem_0\tensor\idem_0 \\
&
+ \idem_1\tensor\idem_0\tensor\idem_0
+ \Sas^{010}\idem_0\tensor\idem_1\tensor\idem_0
+ \idem_0\tensor\idem_0\tensor\idem_1
\nonumber\\
& + \Sas^{110}\idem_1\tensor\idem_1\tensor\idem_0
  + \Sas^{101}\idem_1\tensor\idem_0\tensor\idem_1
  +\idem_0\tensor\idem_1\tensor\idem_1
\nonumber\\
& + \Sas^{111}\idem_1\tensor\idem_1\tensor\idem_1 \gc
\nonumber
\end{align}
where 
\be\label{eq:Lambda-abc}
\Sas^{abc}= \beta^{2abc} \prod_
{k=1}
^\np \Sas^{abc}_{(k)}
\ee
(i.e.\ there is a factor of $\beta^2$ in sector 
\textbf{111})
with
\begin{align}\label{eq:lambda-sector-def}
\Sas^{010}_{(k)} =~& \one\tensor\one\tensor\one
    + (1 + \rmi) \xm_k\tensor\one\tensor\xp_k - (1 - \rmi) \xp_k\tensor\one\tensor\xm_k -
  2 \xp_k\xm_k\tensor\one\tensor\xp_k\xm_k\gc
\\ \nonumber
\Sas^{110}_{(k)} =~& \one\otimes \one\otimes \one - \rmi \one\otimes (\xp_k + \xm_k)\otimes (\xp_k + \xm_k)\gc\\ \nonumber
\Sas^{101}_{(k)} =~& \one\otimes \one\otimes \one + \rmi \one\otimes (\xp_k + \rmi\xm_k)\otimes (\xp_k + \rmi\xm_k) + (1 - \rmi)(\xp_k- \xm_k)\otimes \xp_k  \xm_k\otimes (\rmi\xp_k-\xm_k)\\ \nonumber
    &+ (1 + \rmi) \xm_k\otimes \xp_k\otimes \one
    - (1 - \rmi) \xp_k\otimes \xm_k\otimes \one
 + (1 + \rmi) \one\otimes \xp_k  \xm_k\otimes \one
     - 2 \xp_k  \xm_k\otimes \xp_k  \xm_k\otimes \one \gc\\ \nonumber
\Sas^{111}_{(k)} =~& (\rmi-1) \one\otimes \one\otimes \one
        + (1 - \rmi) \one\otimes \one\otimes \xp_k\xm_k + (1 + 2 \rmi) \one\otimes \xm_k\otimes \xm_k 
        + \one\otimes \xm_k\otimes \xp_k 
\\ \nonumber
 &
        + \one\otimes \xp_k\otimes \xm_k
+   \one\otimes \xp_k\otimes \xp_k + (1 - \rmi) \one\otimes \xp_k\xm_k\otimes \one - (2 - 2 \rmi) \one\otimes \xp_k\xm_k\otimes \xp_k\xm_k 
\\ \nonumber
 &
- (1 - \rmi) \xm_k\otimes \one\otimes \xm_k 
+ (1 + \rmi) \xm_k\otimes\one\otimes \xp_k - (1 - \rmi) \xm_k\otimes \xm_k\otimes \xp_k  \xm_k 
\\ \nonumber
 &
+ (1 + \rmi) \xm_k\otimes \xp_k\otimes \xp_k  \xm_k 
+ (1 - \rmi) \xm_k\otimes \xp_k  \xm_k\otimes \xm_k
- (1 + \rmi) \xm_k\otimes \xp_k  \xm_k\otimes \xp_k 
\\ \nonumber
 &
+ (1 - \rmi) \xp_k\otimes \xm_k\otimes \one - (1 - \rmi) \xp_k\otimes \xm_k\otimes \xp_k  \xm_k 
- (1 + \rmi) \xp_k\otimes \xp_k\otimes \one 
\\ \nonumber
 &
+ (1 + \rmi) \xp_k\otimes \xp_k\otimes \xp_k  \xm_k - (1 - \rmi) \xp_k\otimes \xp_k\xm_k\otimes \xm_k
+ (1 + \rmi) \xp_k\otimes \xp_k  \xm_k\otimes \xp_k 
\\ \nonumber
 &
+ 2 \xp_k  \xm_k\otimes \one\otimes \one - 2 \xp_k  \xm_k\otimes \one\otimes \xp_k  \xm_k - 2 \rmi \xp_k  \xm_k\otimes \xm_k\otimes \xm_k 
- 2 \rmi \xp_k\xm_k\otimes \xp_k\otimes \xp_k 
\\ \nonumber
 &
- 2 \xp_k\xm_k\otimes \xp_k  \xm_k\otimes \one + 4 \xp_k  \xm_k\otimes \xp_k  \xm_k\otimes \xp_k  \xm_k
 \gp
\end{align}
Note the product of the $\Sas^{abc}_{(k)}$ is taken in the tensor product of super-algebras $\Salg \ot \Salg \ot \Salg$, which includes parity signs, as defined in \eqref{eq:super-alg-tensor}. The
$\Sas^{abc}_{(k)}$ is parity-even, and one quickly checks that for $k \neq l$ the elements $\Sas^{abc}_{(k)}$ and $\Sas^{abc}_{(l)}$ commute in $S \ot S \ot S$, so that the order in which one takes the product does not matter.

Finally, we define an element $\RS\in\Salg\tensor\Salg$ which is given by 
\be \label{eq:R-element-for-S}
	\RS = \RS^{00}\cdot\idem_0\tensor\idem_0+\RS^{01}\cdot\idem_0\tensor\idem_1+\RS^{10}\cdot\idem_1\tensor\idem_0+\RS^{11}\cdot\idem_1\tensor\idem_1 
\ee
where 
$\RS^{nm} = \beta^{nm} \prod_{k=1}^\np \RS^{nm}_{(k)}$ 
and
\begin{align} \label{eq:R-element-for-S-sectors}
	\RS^{00}_{(k)} &= \one\tensor\one-2\xm_k\tensor\xp_k \gc \\ \nonumber
	\RS^{01}_{(k)} &= \one\tensor\one-(1+\rmi)\xm_k\tensor\xp_k-(1+\rmi)\xp_k\tensor\xm_k-(1+\rmi)\xp_k\xm_k\tensor\one+2\rmi\xp_k\xm_k\tensor\xp_k\xm_k \gc\\ \nonumber
	\RS^{10}_{(k)} &= \one\tensor\one-(1-\rmi)\xm_k\tensor\xp_k-(1-\rmi)\xp_k\tensor\xm_k-(1-\rmi)\one\tensor\xp_k\xm_k+2\rmi\xp_k\xm_k\tensor\xp_k\xm_k \gc\\ \nonumber
	\RS^{11}_{(k)} &= 
	-\rmi\xm_k\tensor\xm_k 
	-\xm_k\tensor\xp_k
	-\xp_k\tensor\xm_k
	+\rmi\xp_k\tensor\xp_k
	+ 2\xp_k\xm_k\tensor\xp_k\xm_k\gp 
\end{align}
Again, for $k \neq l$ the elements $\RS^{mn}_{(k)}$ and $\RS^{mn}_{(l)}$ commute in $\Salg \ot \Salg$.
We will use $\RS$ later to define a braiding in $\repS$ (but $\RS$ is not an R-matrix for $\Salg$ as the braiding will involve an extra parity map).

For $\np=1$, the quasi-bialgebra algebra $\Salg$ was defined in \cite[Sec.\,4]{Gainutdinov:2015lja}. The proof that $\Salg(\np,\beta)$ is a quasi-bialgebra for general $\np$ relies on reducing the problem to the $\np=1$ case.
To do so, choose $\beta_1,\dots,\beta_\np$ such that $\beta_i^4 = -1$ and $\beta = \beta_1 \cdots \beta_\np$. Define $\mathsf{A}$ to be the $\np$-fold tensor product
\be\label{eq:A-via-SN=1-def}
	\mathsf{A} := \Salg(1,\beta_1) \ot \cdots \ot \Salg(1,\beta_\np)
\ee
of quasi-bialgebras in $\svect$. 
Thus $\mathsf{A}$ is itself a quasi-bialgebra in $\svect$. The product and coproduct of $\mathsf{A}$ are defined with parity signs (see \eqref{eq:super-alg-tensor} for the product) and the coassociator 
$\Lambda^\mathsf{A}$
of $\mathsf{A}$ is the product in $\mathsf{A} \ot \mathsf{A} \ot \mathsf{A}$ of those in the individual factors $\Salg(1,\beta_i)$.
Consider the two-sided ideal $I$ in $\mathsf{A}$ generated by
\be
	I := \big\langle \LL_i-\LL_{j} \,\big|\, 1\leq i,j\leq\np \big\rangle \ .
\ee
where $\LL_i$ stands for the element with $\LL$ on the $i$-th tensor factor and $\one$ else. One verifies that this ideal satisfies $\Delta^{\mathsf{A}}(I) \subset I\tensor\mathsf{A}+\mathsf{A}\tensor I$ and $\eps^\mathsf{A}(I)=0$, i.e.\ it is a (quasi-)bialgebra ideal.

\begin{proposition}\label{prop:S-quasiHopf}
For $\np,\beta$ as in \eqref{eq:beta-param}, the
data $(\Salg,\cdot,\one,\copS,\eps^\Salg,\Sas)$ defines a quasi-bialgebra in $\svect$ and $\Salg \cong \mathsf{A}/I$ as quasi-bialgebras.
\end{proposition}

\begin{proof}
	We write $\xp$, $\xm$ and $\LL$ for the generators of $\Salg(1,\beta_j)$.
Consider the surjective even linear map $\varphi : \mathsf{A} \to \Salg(\np,\beta)$ given by
\begin{align}
&\varphi\Big(\,(\xp)^{\eps_1}(\xm)^{\delta_1} 
\LL^{m_1}
\ot\cdots \ot(\xp)^{\eps_\np}(\xm)^{\delta_\np} \LL^{m_\np} \, \Big)
\\ \nonumber
&\hspace{4em}
	=~
(\xp_1)^{\eps_1}(\xm_1)^{\delta_1}\cdots (\xp_\np)^{\eps_\np}(\xm_\np)^{\delta_\np} \LL^{m_1+\cdots+m_\np} 
\ .
\end{align}
The map $\varphi$ satisfies
\be
	\varphi(\one^{\otimes\np})=\one
	~~,\quad
	\eps^\Salg \circ \varphi = \eps^{\mathsf{A}}
	~~,\quad
	\varphi(ab) = \varphi(a)\varphi(b)
	~~,\quad
	(\varphi \ot \varphi) \circ \Delta^{\mathsf{A}}
	= \Delta^{\Salg} \circ \varphi
\ee
Since $\varphi$ is surjective, this proves that $\Delta^\Salg$ is an algebra homomorphism. Furthermore, $\varphi^{\ot 3}$ maps the coassociator of $\mathsf{A}$ to that of $\Salg$:
\be
	\varphi^{\ot 3}(\Lambda^\mathsf{A}) = \Lambda \ .
\ee
This follows from the product form of $\Lambda \in \Salg^{\ot 3}$ and the fact that $\beta = \beta_1 \cdots \beta_\np$. Using once more that $\Salg$ is surjective, it follows that $\Lambda$ satisfies the defining relations of a coassociator 
(see Definition~\ref{def:quasi-bialg-symmoncat}).

Thus $\Salg$ is a quasi-bialgebra. Finally, it is clear from the definition of $\varphi$ that $\mathrm{ker}(\varphi) = I$.
\end{proof}

%%%%%%%%%%%%%%%%%%%
\subsection{An equivalence from \texorpdfstring{$\catSF$}{SF} to \texorpdfstring{$\repS$}{Rep(S)}}\label{sec:eqSFRepS}

Due to \eqref{eq:Salg-dec}, $\repS$ can be decomposed into two parts:
\be\label{eq:repS-dec}
	\repS = \repS_0 \oplus \repS_1 \gp
\ee
In this section we construct an equivalence of $\oC$-linear categories $\mathcal D\colon \catSF \to \repS$.
The functor has two components due to the decompositions in \eqref{eq:catSF-dec} and \eqref{eq:repS-dec}. 
On $\catSF_0$ we have
\be\label{eq:D-functor-sec0-def}
\mathcal D_0\colon \catSF_0\to \repS_0 \gcg \mathcal D_0(U)=U \quad\text{where}\quad \xp_i u=a_i u \gcg \xm_i u=b_i u \quad (u\in U) \gc 
\ee
and $D_0(f)=f$ on morphisms. 

For the second component we need the non-central idempotent 
\begin{align} \label{eq:Bel}
\Bel\coloneqq \prod_{i=1}^\np \xm_i \xp_i \idem_1 \in \Salg_1 \gp
\end{align}
To see that indeed $\Bel \Bel = \Bel$ note that for each $i$ separately we have $\xm_i \xp_i \xm_i \xp_i \idem_1 = \xm_i \xp_i \idem_1$. The idempotent $\Bel$ generates a
$\Salg_1$-submodule $\B\subset \Salg_1$:
\begin{align} 
	\B\coloneqq \Salg \Bel= \Salg_1\Bel = \operatorname{span} \left\{(\xp_\np)^{i_\np}\cdots (\xp_1)^{i_1} \Bel \mid i_1,\ldots ,i_\np\in\{0,1\}\right\}.
\end{align}
Since $\Salg_1$ is a Clifford algebra, $\B$ is a simple (and projective) $\Salg_1$-module in $\svect$. There are up to isomorphism two distinct simple $\Salg_1$-modules, namely $\B$ and $\B \otimes_{\svect} \oC^{0|1}$, the parity-shifted copy of $\B$.

The second component of $\mathcal D$ is defined as
\begin{align} \label{eq:D-functor-sec1-def}
\mathcal D_1\colon \catSF_1\to \repS_1 \gcg \mathcal D_1(U)=\B\otimes_{\svect}U \gc 
\end{align}
where the $\Salg_1$-action on $\mathcal D_1(U)$ is the left multiplication on $\B$ and where $U\in \catSF_1\simeq\svect$. 
On morphisms we set $\mathcal D_1(f)=\id_\B\otimes f$. 

\begin{proposition}\label{prop:D-is-C-lin-equiv}
The functor $\mathcal D$ is an equivalence of $\oC$-linear categories. 
\end{proposition}
\begin{proof}
We will give an inverse functor $\mathcal E \colon \repS \to \catSF$ to $\mathcal D$. The functor $\mathcal E$ will again be defined separately in the two sectors. 

Given $V\in\repS_0$, the $\algGr$-module $\mathcal E(V)$ has $V$ as the underlying super-vector space with the action of $a_i$ and $b_i$ given
by the action of $\xp_i$ and $\xm_i$, respectively. On morphisms $f$ in $\repS_0$ we set $\mathcal E(f)=f$. Clearly, $\mathcal D$ and $\mathcal E$ are inverse to each other as functors between $\catSF_0$ and $\rep \Salg_0$.

For a given object $M \in \repS_1$, we set $\mathcal E (M)=\Hom^{\vect}_{\Salg_1}(\B,M)$. This means we consider all $\Salg_1$-linear maps from $\B$ to $M$, not just the parity-even maps. The decomposition into parity-even and parity-odd maps turns $\Hom^{\vect}_{\Salg_1}(\B,M)$ into a super-vector space. Since $\B$ is simple, we get a canonical isomorphism $\Hom^{\vect}_{\Salg_1}(\B,\B \ot_\svect V) \cong V$ of super-vector spaces which is natural in $V$, that is, $\mathcal{E}(\mathcal{D}(V)) \cong V$.
Conversely, since every $\Salg_1$-module is isomorphic to a direct sum of copies of $\B$ and $\B \otimes_{\svect} \oC^{0|1}$, that is $M\cong \B\otimes V$ for some $V\in\svect$, we get 
$M \cong \B \ot \Hom^{\vect}_{\Salg_1}(\B,M)$.
\end{proof}

%%%%%%%%%%%%%%%
\subsection{\texorpdfstring{$\funD$}{D} as a multiplicative functor}\label{sec:funD}
Below we will make an ansatz for isomorphisms
\begin{equation}\label{eq:isoD-source-target}
\isoD_{U,V}\colon \funD(U\Ctensor V) \to 
\funD(U)\Stensor \funD(V) \gp
\end{equation}
which will be given sector by sector. Both, the expression for $\isoD_{U,V}$ and the proof that it is an isomorphism mimics the corresponding construction for $\np=1$ in \cite[Sec.\,6.3]{Gainutdinov:2015lja}.

The multiplication map $\Salg\ot\Salg\to\Salg$ is denoted by $\mu^{\Salg}$, and for an $\Salg$-module $U$ the expression $\rho^U : \Salg \tensor U \to U$ stands for the $\Salg$-action  on $U$. 
We will also need the constants
\begin{align} \label{eq:delta0-delta1}
\delta_0 = \prod_{i=1}^\np \left(\one\otimes\one + \xm_i\otimes \xp_i\right) \in \Salg\otimes \Salg,  \qquad
\delta_1 = \prod_{i=1}^\np \left(\one + \xp_i \xm_i\right) \in \Salg .
\end{align}
Note that they are multiplicative invertible with inverses 
\begin{align} \label{eq:delta0-delta1-inv}
\delta_0^{-1} = \prod_{i=1}^\np \left(\one\otimes\one - \xm_i\otimes \xp_i\right) \in \Salg\otimes \Salg  , \quad
\delta_1^{-1} = \prod_{i=1}^\np \left(\one - \xp_i \xm_i \left(\idem_0+\tfrac12 \idem_1\right)\right) \in \Salg.
\end{align}
We denote the right multiplication with $a\in \Salg$ by $R_a$.

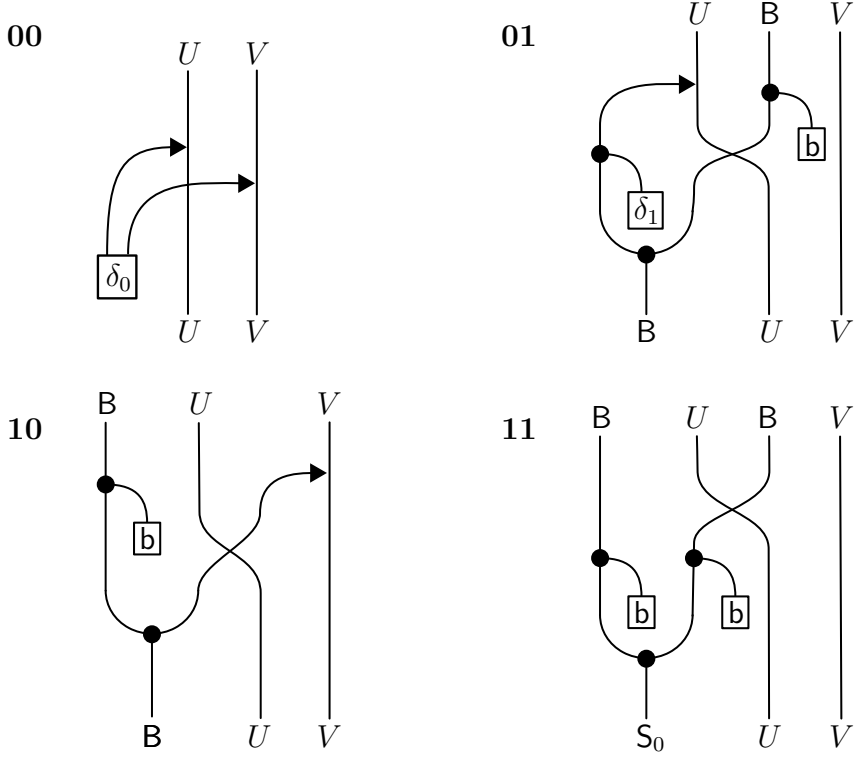
\begin{figure}[bt]
\eqpics{Delta-mult-struct} {300} {125} {\setlength\unitlength{.85pt}
\put(10,180){\scalebox{.2}{\includegraphics[scale=1]{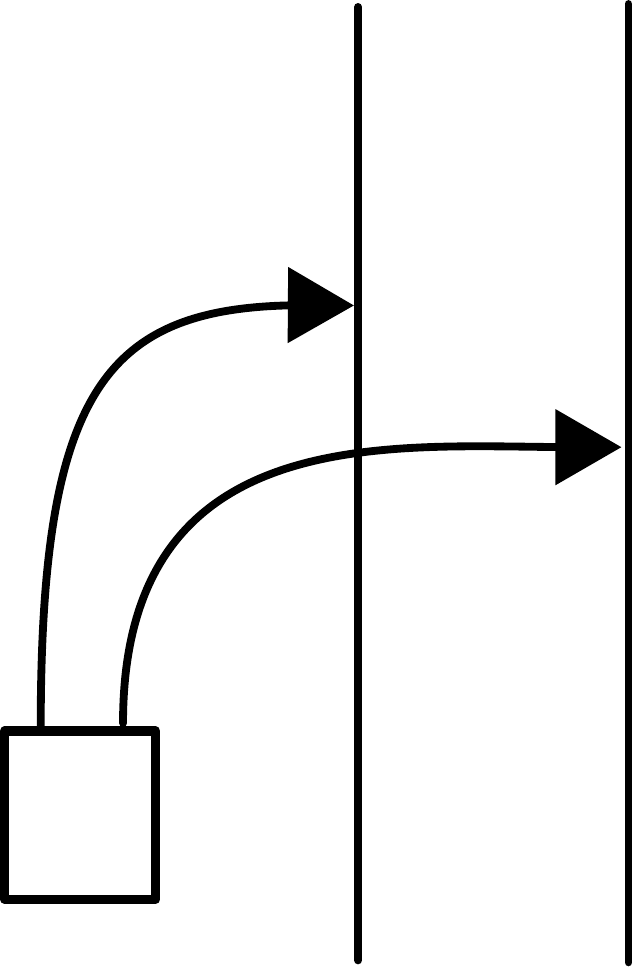}}}
\put(-30,300){\bf 00} \put(14,193){$\delta_0$} \put(46,168){$U$}  \put(76,168){$V$} \put(46,292){$U$}  \put(76,292){$V$}
\put(230,180){\scalebox{.2}{\includegraphics[scale=1]{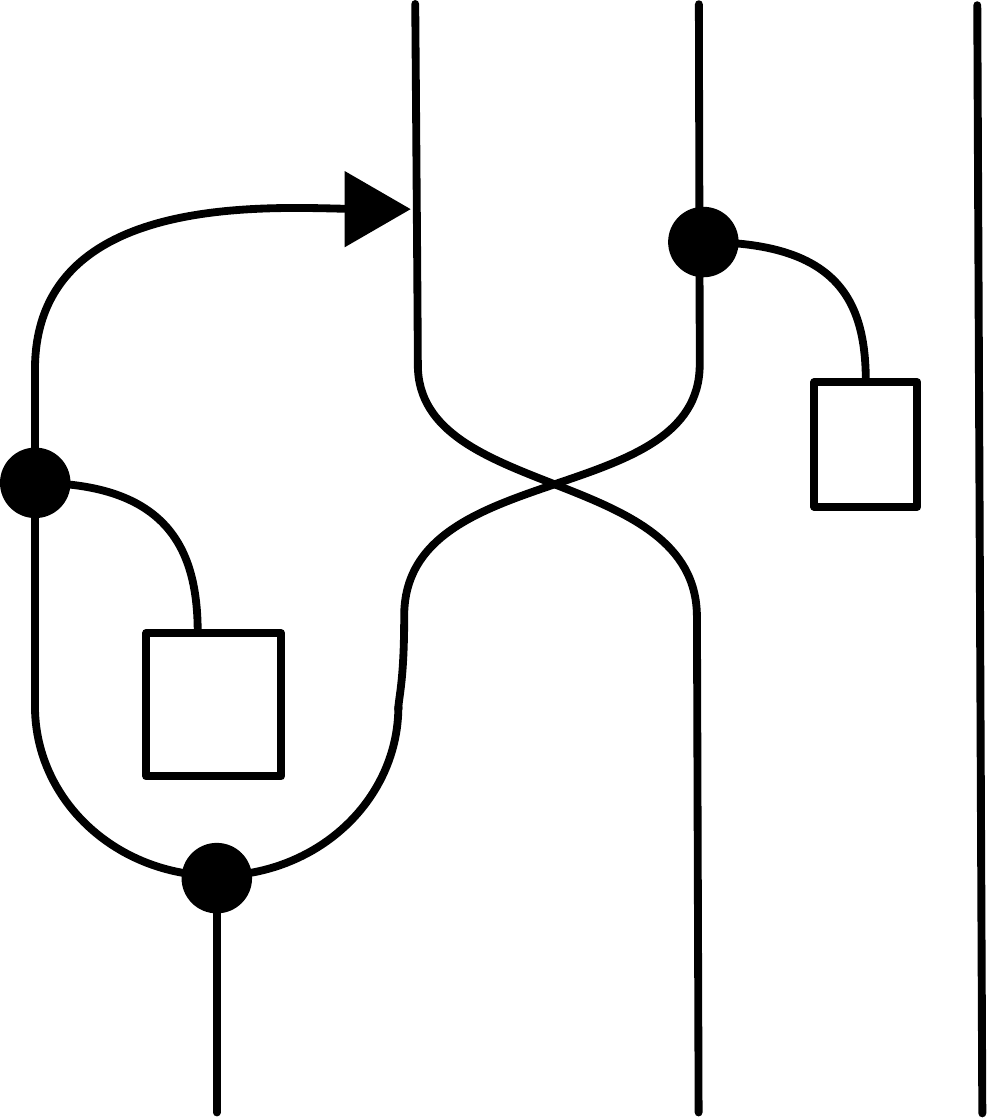}}}
\put(190,300){\bf 01} \put(250,168){$\B$} \put(305,168){$U$} \put(336,168){$V$} \put(273,310){$U$} \put(305,310){$\B$} \put(336,310){$V$} \put(249,223)
{$\delta_1$} 
\put(325,251){$\Bel$}
\put(10,0){\scalebox{.2}{\includegraphics[scale=1]{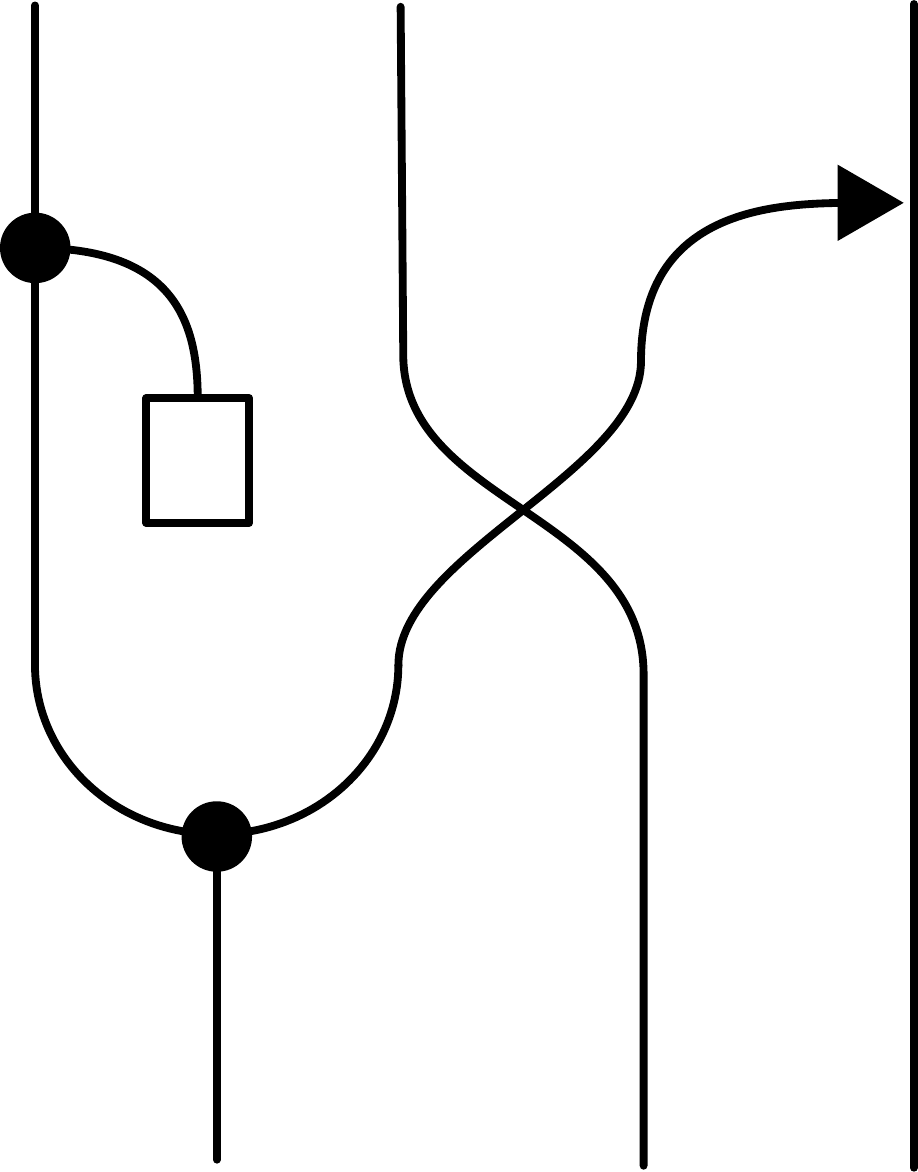}}}
\put(-30,125){\bf 10} \put(30,-12){$\B$} \put(10,136){$\B$} \put(77,-12){$U$} \put(108,-12){$V$}  \put(29,76){$\Bel$}
 \put(51,136){$U$} \put(108,136){$V$}
\put(230,0){\scalebox{.2}{\includegraphics[scale=1]{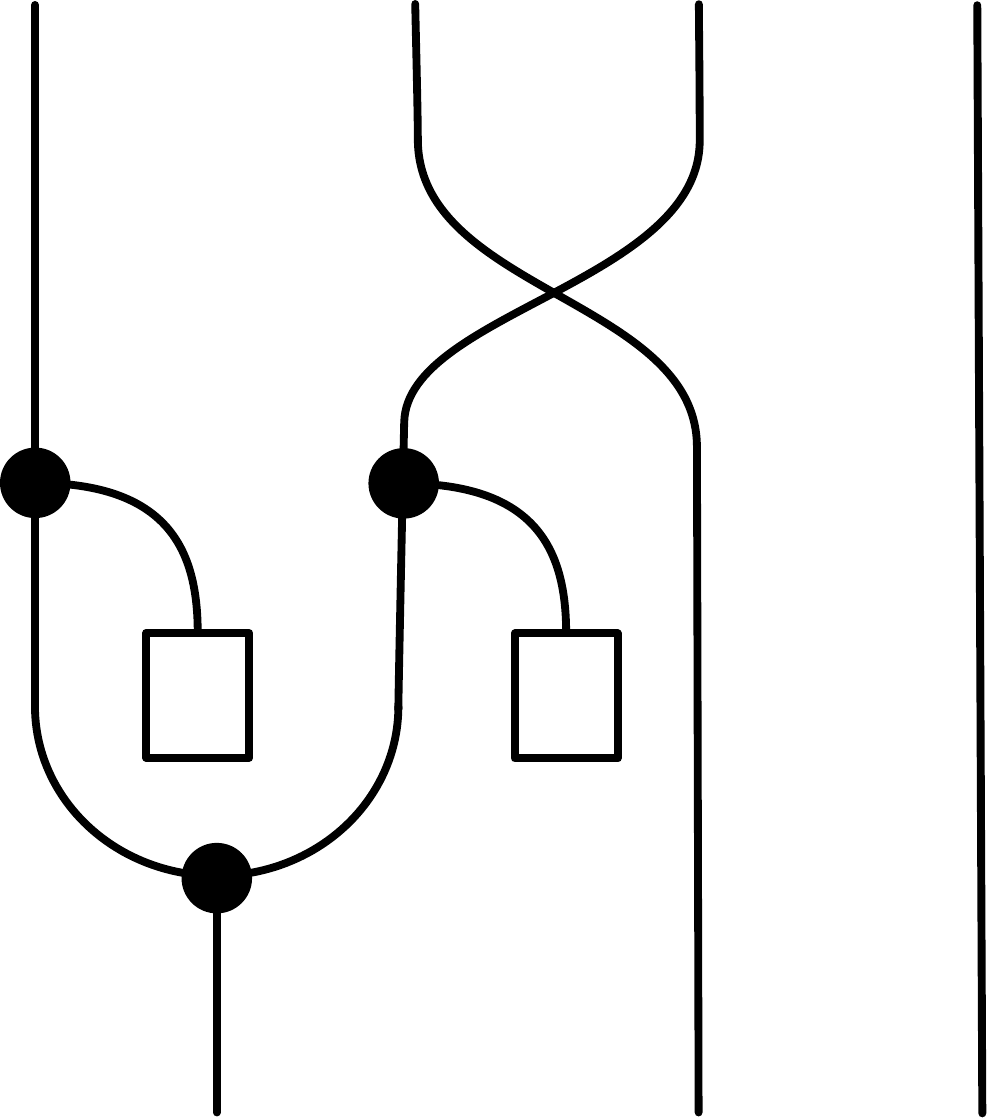}}}
\put(190,125){\bf 11} \put(30,-12){$\B$} 
\put(250,-12){$\Salg_0$} 
\put(272,130){$U$}  \put(249,43)
{$\Bel$}\put(291,43){$\Bel$}
 \put(230,130){$\B$} 
\put(304,-12){$U$} \put(336,-12){$V$} \put(304,130){$\B$} \put(336,130){$V$}
}
\caption{String diagram notation for $\isoD_{U,V}$ in the four sectors.}
\label{fig:Delta_UV}
\end{figure}

The following list defines $\isoD_{U,V}$ for each of the four possibilities to choose $U \in \catSF_i$ and $V \in \catSF_j$, $i,j \in \{0,1\}$, which we refer to as ``sector $ij$''. The underlines indicate on which factors $\Salg$ acts. In Figure~\ref{fig:Delta_UV} we give string diagram expressions for $\isoD_{U,V}$.
\begin{itemize}\setlength{\itemsep}{10pt}
\item \textbf{00} sector: $\isoD_{U,V}: \underline U\tensor \underline V\to \underline U \tensor \underline V$ is given by
\begin{equation}\label{eq:isoD00}
\isoD_{U,V} = \big(\rho^U \tensor \rho^V\big) \circ \big(\id_{\Salg} \tensor \sflip_{\Salg,U} \tensor \id_V\big) \circ \big(\delta_0 \tensor \id_U \tensor \id_V\big) \gc
\end{equation}
where $\sflip$ is defined in \eqref{sflip}.

\item \textbf{01} sector: $\isoD_{U,V}: \underline\B\tensor U\tensor V\to \underline U \tensor \underline\B\tensor V$ is given by
\begin{equation}\label{eq:isoD01}
\isoD_{U,V} =
\big(\rho^U \tensor R_\Bel \tensor \id_V \big) \circ
 \big(R_{\delta_1} \tensor \sflip_{\Salg,U} \tensor \id_V\big)
\circ \big( \copS \tensor \id_U \tensor \id_V \big) \gp
\end{equation}
The image of $R_\Bel$ on
$\Salg$
 is $\B$, that is, the target is indeed $U \tensor \B \tensor V$.

\item \textbf{10} sector: $\isoD_{U,V}: \underline\B\tensor U\tensor V\to \underline\B \tensor U\tensor \underline V$  is given by
\begin{equation}\label{eq:isoD10}
\isoD_{U,V} =
\big(R_\Bel \tensor \id_U \tensor \rho^V \big) \circ
 \big(\id_\Salg \tensor \sflip_{\Salg,U} \tensor \id_V\big)
\circ \big( \copS \tensor \id_U \tensor \id_V \big) \gp
\end{equation}

\item \textbf{11} sector: $\isoD_{U,V}: \underline{\Salg_0}\tensor U\tensor V\to \underline\B\tensor U \tensor \underline\B\tensor V$ is given by
\begin{equation}\label{eq:isoD11}
\isoD_{U,V} =
 \big(\id_\B \tensor \sflip_{\B,U} \tensor \id_V\big)
\circ \big( [(R_\Bel \tensor R_\Bel) \circ \copS] \tensor \id_U \tensor \id_V \big) \gp
\end{equation}
Here, in the source $\Salg$-module we have identified $\Salg_0$ and $\algGr$.
\end{itemize}

\begin{lemma}\label{lem:isoD-is-S-map}
The linear maps $\Delta_{U,V}$ are intertwiners of $\Salg$-modules. 
\end{lemma}
\begin{proof}
In all sectors beside \textbf{00} this clear since $\copS$
is an algebra map, and since the right-multiplications $R_{\delta_1}$ and $R_b$
are left-module intertwiners.
In sector \textbf{00} we have to check that 
$(\idem_0\otimes \idem_0)\cdot \copS (a)\cdot\delta_0=  (\idem_0\otimes \idem_0) \cdot\delta_0 \cdot \copS (a)$ holds for all $a\in \Salg$.
This is clear for 
$a=\LL$,
 and for $a=\xpm_i$ it is an easy check.
\end{proof}

\begin{lemma}\label{lem:isoD-is-iso}
The $\Delta_{U,V}$ are isomorphisms. 
\end{lemma}
\begin{proof}
In sectors \textbf{00, 01} and \textbf{10} the proof is similar as that in \cite[Lemma~6.5]{Gainutdinov:2015lja}. Indeed, for sector {\bf 00}, invertibility of $\isoD_{U,V}$ follows from that of $\delta_0$. For sector {\bf 10} and {\bf 01} the inverse of $\isoD_{U,V}$ is given by the same string diagram as that in \cite[Lemma~6.5]{Gainutdinov:2015lja}.
The proof that the expressions in \cite[Lemma~6.5]{Gainutdinov:2015lja} are indeed inverses rests on the two identities
\begin{align}\label{eq:Delta-iso-identity}
(\Bel\otimes \one)\cdot((\id\otimes \omega_{\Salg})\circ\copS (\Bel))\cdot (\Bel\otimes \one) &= \Bel\otimes (\delta_1 \idem_0) \ ,
\\ \nonumber
(\Bel\tensor \one)\cdot \big((\id\tensor \omega_ {\Salg}) \circ \sflip\circ \copS(\Bel)\big) \cdot (\Bel\tensor \one) &= \Bel\tensor (\delta_1^{-1}\idem_0) \ ,
\end{align}
where $\omega_\Salg$ denotes the parity involution on the super-vector space $\Salg$ as defined in \eqref{eq:svect-parityinv-def}\footnote{Note that the antipode in \cite{Gainutdinov:2015lja} coincides with $\omega_\Salg$ in sector 0.}.
We need to show that these identities remain true for general $\np$, and we will do so for the first identity.

Namely, for $i\neq j$ and by defining $\Bel^{(i)}=\xm_i \xp_i \idem_1 $ and $\delta_1^{(i)}=\one + \xp_i \xm_i$ we get
\begin{align}
&\Bel^{(i)}\Bel^{(j)} \otimes \delta_1^{(i)}\delta_1^{(j)} \idem_0 
= (\Bel^{(i)}\otimes (\delta_1^{(i)} \idem_0)) \cdot (\Bel^{(j)}\otimes (\delta_1^{(j)} \idem_0)) \\ \nonumber
&\overset{(*)}=\left((\Bel^{(i)}\otimes \one) ((\id\otimes \omega_ {\Salg})\circ\copS (\Bel^{(i)})) (\Bel^{(i)}\otimes \one)\right) \\  \nonumber
&\qquad\cdot \left((\Bel^{(j)}\otimes \one)  ((\id\otimes \omega_ {\Salg})\circ\copS (\Bel^{(j)}))  (\Bel^{(j)}\otimes \one)\right) \\ \nonumber
&\overset{(**)}=(\Bel^{(i)}\otimes \one)\cdot(\Bel^{(j)}\otimes \one) \cdot ((\id\otimes \omega_ {\Salg})\circ\copS(\Bel^{(i)})\cdot(\id\otimes \omega_ {\Salg})\circ\copS (\Bel^{(j)})) \\  \nonumber
&\qquad\cdot(\Bel^{(i)}\otimes \one)\cdot(\Bel^{(j)}\otimes \one) \\ \nonumber
&=(\Bel^{(i)}\Bel^{(j)}\otimes \one)\cdot ((\id\otimes \omega_ {\Salg}) \circ\copS (\Bel^{(i)}\Bel^{(j)})) \cdot (\Bel^{(i)}\Bel^{(j)}\otimes \one) \gc 
\end{align} 
where we used that $\Bel^{(i)}$ is an even element
and in step (*) that the corresponding identity holds for $\np=1$,
then  in step (**) that 
$\omega_\Salg$,
 $\copS$ are morphisms in $\svect$
and so $(\id\otimes \omega_ {\Salg})\circ\copS (\Bel^{(i)})$ is even and moreover does not contain $\xpm_j$ and therefore it commutes with $\Bel^{(j)}\otimes\one$.
  Using then the fact that $\Bel = \prod_{i=1}^{\np}\Bel^{(i)}$ and $\delta_1 = \prod_{i=1}^{\np}\delta_1^{(i)}$, this analysis shows the first identity in~\eqref{eq:Delta-iso-identity}. The second identity there is established in a similar way.

In sector \textbf{11}, we have to show that 
\begin{align}
\Theta \colon \Salg_0 &\to \B\otimes_{\svect} \B \\ \nonumber
a &\mapsto \copS(a) \cdot \Bel\otimes \Bel 
\end{align} 
is an isomorphism.
Note that
 $\copS (\xpm_i) \cdot \idem_1\otimes \idem_1 = (\xpm_i \otimes\one - \rmi\,\one\otimes \xmp_i) \cdot \idem_1\otimes \idem_1$
and $\xm_i \Bel=0$. We get
\begin{align}
\Theta (\idem_0) &=  \copS(\idem_0) \cdot \Bel\otimes \Bel = \Bel\otimes \Bel\\ \nonumber
\Theta (\xp_i \xp_j \idem_0) 
&= \copS(\xp_i)\cdot (\xp_j\Bel)\otimes \Bel = (\xp_i\xp_j\Bel)\otimes \Bel\\ \nonumber
\Theta (\xm_i \xm_j \idem_0) 
&= -\copS(\xm_i) \cdot \rmi\Bel\otimes (\xp_j \Bel) = (-\rmi)^2\Bel\otimes (\xp_i\xp_j\Bel) \\ \nonumber
\Theta (\xp_i \xm_j \idem_0) 
&= -\copS(\xp_i) \cdot \rmi\Bel\otimes (\xp_j \Bel) = -(\rmi\xp_i\Bel)\otimes (\xp_j \Bel) + \delta_{i,j} \rmi\Bel\otimes (\xm_i\xp_j \Bel) \\ \nonumber
&= -(\rmi\xp_i\Bel)\otimes (\xp_j \Bel) + \delta_{i,j} \Bel\otimes \Bel \gp
\end{align} 
Now it is straightforward to see that for $i_1<\ldots <i_k$ and $j_1<\ldots <j_m$ we have
\begin{align}
\Theta (\xp_{i_1}\cdots \xp_{i_k} \xm_{j_1}\cdots \xm_{j_m} \idem_0) 
&= \copS(\xp_{i_1}\cdots \xp_{i_k}) \cdot (-\rmi)^m \Bel\otimes (\xp_{j_1} \cdots \xp_{j_m} \Bel) \\ \nonumber
&= (-\rmi)^m (\xp_{i_1} \cdots \xp_{i_k} \Bel)\otimes (\xp_{j_1} \cdots \xp_{j_m} \Bel) + \tilde w ,
\end{align} 
where $\tilde w \in \bigoplus_{i=0}^{k+m-2}\tilde W_{i}$ and $\tilde W_i$ is the span of elements $(\xp_{i_1} \cdots \xp_{i_u} \Bel)\otimes (\xp_{j_1} \cdots \xp_{j_v} \Bel)$
with $i=u+v$.
So we can show by induction over $m+k$ that $\Theta$ is surjective and since $\Salg_0$ and $\B\otimes \B$ have dimension $2^{2\np}$, $\Theta$ is even bijective. 
\end{proof}

From the definition of $\isoD_{U,V}$ it is immediate that these maps are natural in $U$ and $V$. Together with Lemmas \ref{lem:isoD-is-S-map} and \ref{lem:isoD-is-iso} this proves the following proposition:

\begin{proposition}\label{prop:functor-D-mult}
With the isomorphisms $\Delta_{U,V}$, the functor $\mathcal D\colon \catSF \to \repS$ is multiplicative. 
\end{proposition}

%%%%%%%%%%%%%%%%%%%

\subsection{Compatibility with associator and unit isomorphisms} \label{sec:trans-assoc-SF-repS}

In this section we verify that the functor $\mathcal{D}$ from Proposition~\ref{prop:functor-D-mult} is monoidal, i.e.\ that it is compatible with associators and unit isomorphisms.

In fact, since the unit isomorphisms in $\catSF$ and $\repS$ are just those of the underlying super-vector spaces, after choosing the isomorphism $\one \to \mathcal{D}(\one)$ to be the identity on $\oC^{1|0}$, compatibility of $\mathcal{D}$ with the unit isomorphisms is immediate.

\begin{table}[bt]
\begin{align*}
&{\bf000}~:~&& \underline\Sas^{000} \cdot (\one \tensor \delta_0) \cdot (\id \tensor \copS)(\delta_0)\\
&&&= (\idem_0 \tensor \idem_0 \tensor \idem_0) \cdot	(\delta_0 \tensor \one) \cdot (\copS \tensor \id)(\delta_0) \\[.6em]
&{\bf001}~:~&&  \underline\Sas^{001} \cdot \bigl( (\id\tensor\copS)\circ(R_{\delta_1}\tensor R_{\Bel})\circ\copS(v)\bigr)\cdot(\one\tensor\delta_1\tensor\Bel)\\
&&&= (\idem_0\tensor\idem_0\tensor\one)\cdot(\delta_0\tensor\one)\cdot\bigl((\copS\tensor\id)\circ(R_{\delta_1}\tensor R_{\Bel})\circ\copS(v)\bigr)\\[.6em]
&{\bf010}~:~&&  \underline\Sas^{010} \cdot \bigl( (\id\tensor\copS)\circ(R_{\delta_1}\tensor R_{\Bel})\circ\copS(v)\bigr)\cdot(\one\tensor\Bel\tensor\one)\\
&&&= 
 \bigl((\copS\tensor
	\id
 )\circ(R_{\Bel}\tensor
 	\id
 )\circ\copS(v)\bigr)\cdot(\delta_1\tensor\Bel\tensor\one)\cdot\gamma^{(13)}\\[.6em]
&{\bf100}~:~&&  \underline\Sas^{100}\cdot (\one\tensor\delta_0)\cdot\bigl( (\id\tensor\copS)\circ(R_{\Bel}\tensor
	\id
)\circ\copS(v)\bigr)\\
&&&=\bigl((\copS\tensor
	\id
)\circ(R_{\Bel}\tensor
	\id
) \circ\copS(v)\bigr)\cdot(\Bel\tensor\one\tensor\one)\\[.6em]
&{\bf110}~:~&&  \underline\Sas^{110} \cdot \bigl( (\id\tensor\copS)\circ(R_{\Bel}\tensor R_{\Bel})\circ\copS(h)\bigr)\cdot (\one\tensor\Bel\tensor\one)  \\
&&&= \bigl\{(\copS\tensor\id)\bigl(\delta_0\cdot\copS(h)\bigr)\bigr\}\cdot 
(\Bel\tensor\Bel\tensor\one)\\[.6em]
&{\bf101}~:~&&  \underline\Sas^{101} \cdot \bigl\{ (\id\tensor\copS)\bigl( \copS(h)\cdot \Bel\tensor\Bel\bigr)\bigr\}\cdot (\one\tensor\delta_1\tensor \Bel)\\
&&&= (R_{\Bel}\tensor\mu^{\Salg}\tensor
	\id
)\circ(\copS\tensor
   \sflip_{\B,\Salg_0}
)\circ\bigl(\bigl\{ (R_{\Bel}\tensor R_{\Bel})\circ\copS\circ\mu^{\Salg}\bigr\}\tensor
	\id
\bigr) (h\tensor\gamma)\\[.6em]
&{\bf011}~:~&&  \underline\Sas^{011} \cdot \bigl( (\id\tensor\copS)(\delta_0)\bigr)\cdot\bigl(\one\tensor(\copS(h)\cdot\Bel\tensor\Bel)\bigr)\\
&&&= \bigl\{\bigl(\omega_{\Salg}\circ\mu^{\Salg}\circ(\id\tensor \omega_{\Salg})\bigr)\tensor\id\tensor \id\bigr\}\\
&&&\qquad\qquad\circ \bigl\{\id\tensor \bigl((R_{\delta_1}\tensor R_{\Bel})\circ\copS\bigr)\tensor
	\id
\bigr\}
 \circ \bigl\{\id\tensor \bigl((R_{\Bel}\tensor R_{\Bel})\circ \copS\bigr)\bigr\}\circ \copS(h)\\[.6em]
&{\bf111}: \quad&& \underline\Sas^{111} \cdot \bigl\{\bigl(\id\tensor(\copS\circ\mu^{\Salg})\bigr)\bigl(\copS(v)\tensor h\bigr)\bigr\} \cdot \Bel\tensor\Bel\tensor\Bel\\
&&&=\bigl\{\bigl((R_{\Bel}\tensor R_{\Bel})\circ\copS\circ\mu^{\Salg}\bigr)\tensor
	\id
\bigr\}
\circ(
	\id
\tensor\sflip_{\B,\Salg_0})\circ\bigl\{\bigl((R_{\delta_1}\tensor R_{\Bel})\circ\copS(v)\bigr)\tensor\phi(h)\bigr\}
\end{align*}
\caption{
Evaluation of the compatibility of associators as in \eqref{eq:transport-assoc-diag} for $\mathcal{D}$ and $\Delta_{U,V}$ in each of the eight sectors.
The constant $\gamma$ is defined as $\gamma := \exp(C) \cdot \idem_0 \ot \idem_0$ with $C$ as in \eqref{def:copair} and we use the identification $\xp_k =a_k$, $\xm_k = b_k$ as in \eqref{eq:D-functor-sec0-def}.
The map $\phi : \Salg_0 \to \Salg_0$ is defined as in \eqref{eq:SF-phi-def} under the same identification.
The above conditions have to hold for all $h\in\Salg_0$ and $v\in\B$.
}\label{tab:Lam-eqn-sectors}
\end{table}

The main effort lies in showing that the diagram in \eqref{eq:transport-assoc-diag} commutes for $\mathcal D\colon \catSF \to \repS$ and $\Delta_{U,V}$.
In a calculation similar to that in \cite[Sec.\,7.2]{Gainutdinov:2015lja} one can evaluate \eqref{eq:transport-assoc-diag} for each of the eight possibilities of choosing $U \in \catSF_a$, $V \in \catSF_b$, $W \in \catSF_c$, which we refer to as ``sector $abc$''.  
We define 
\be\label{eq:tilde-Lambda}
\underline\Sas^{abc}=\Sas^{abc} \cdot (\idem_a\tensor\idem_b\tensor\idem_c)
\ee
 with $\Sas^{abc}$ as in \eqref{eq:Sas} 
	(it is understood that the $\Sas^{abc}$ not spelled out explicitly in \eqref{eq:Sas} are set to $\one^{\otimes 3}$).
Then \eqref{eq:transport-assoc-diag} for $\mathcal{D}$ and $\Delta_{U,V}$ is equivalent to the eight conditions in Table~\ref{tab:Lam-eqn-sectors}.

To verify the eight identities in Table~\ref{tab:Lam-eqn-sectors}, we reduce them to the $\np=1$ case which has been checked in \cite[Prop.\,7.3]{Gainutdinov:2015lja}.
To do so, define $\Salg^{(k)}$ to be the subalgebra of $\Salg$ generated by $\xp_k$, $\xm_k$ and $\LL$. By definition of $\Salg$, for $k \neq l$ elements from $\Salg^{(k)}$ and $\Salg^{(l)}$ super-commute. 

The aim is now to rewrite each condition in Table~\ref{tab:Lam-eqn-sectors} as a product of terms $k=1,\dots,\np$ which use only elements in $\Salg^{(k)}$ and note that since by \cite[Prop.\,7.3]{Gainutdinov:2015lja} the corresponding identity holds for each $k$ separately, it holds for the product.

For $\underline\Sas^{abc}$ the product decomposition is given in
\eqref{eq:Lambda-abc}.
The constants $\Bel$, $\delta_1$ and $\delta_0$ were already defined as products over elements in $\Salg^{(k)}$ and $\Salg^{(k)}\ot\Salg^{(k)}$ in \eqref{eq:Bel} and \eqref{eq:delta0-delta1}. For $\gamma$ in Table~\ref{tab:Lam-eqn-sectors} set $C^{(i)}=\xm_i\tensor\xp_i-\xp_i\tensor\xm_i$ and use
\be
	\gamma= \exp{\Big(\sum_{i=1}^\np C^{(i)}\Big)}=\prod_{i=1}^\np \gamma^{(i)} 
	\qquad \text{with} \quad
	\gamma^{(i)} := \exp C^{(i)} \ .
\ee

Now sector {\bf 000} directly decomposes into products over elements in 
$(\Salg^{(k)})^{\ot 3}$
 and hence holds by \cite[Prop.\,7.3]{Gainutdinov:2015lja}.
For sector {\bf 001} note that, for
 $c^{(k)},d^{(k)}\in \Salg^{(k)}$ even and $v^{(k)}\in \B^{(k)}$ we have, for $i\neq j$,
 \begin{align}   
	&(\id\tensor\copS) \circ (R_{c^{(i)}c^{(j)}}\tensor R_{d^{(i)}d^{(j)}}) \circ (\copS (v^{(i)}v^{(j)})) \\ \nonumber
	&= (\id\tensor\copS) \circ (R_{c^{(i)}}\tensor R_{d^{(i)}}) \circ (R_{c^{(j)}}\tensor R_{d^{(j)}}) \big(\copS (v^{(i)})\cdot\copS (v^{(j)})\big) \\ \nonumber
	&= (\id\tensor\copS) \circ (R_{c^{(i)}}\tensor R_{d^{(i)}}) \big(\copS (v^{(i)}) \cdot (R_{c^{(j)}}\tensor R_{d^{(j)}}) (\copS (v^{(j)}) )\big) \\ \nonumber
	&= (\id\tensor\copS) \big( (R_{c^{(i)}}\tensor R_{d^{(i)}}) (\copS (v^{(i)})) \cdot (R_{c^{(j)}}\tensor R_{d^{(j)}}) (\copS (v^{(j)}) )\big) \\ \nonumber
	&= (\id\tensor\copS) \big( (R_{c^{(i)}}\tensor R_{d^{(i)}}) (\copS (v^{(i)}))\big) \cdot  (\id\tensor\copS) \big( (R_{c^{(j)}}\tensor R_{d^{(j)}}) (\copS (v^{(j)}) )\big)  \ .
\end{align}
This allows one to write the LHS of the equation for sector {\bf 001} as a product over elements in 
$(\Salg^{(k)})^{\ot 3}$.
 For the RHS one uses a similar  statement where $\id\tensor\copS$ is replaced by $\copS\tensor\id$.

Analogous arguments work in all sectors, except that in sector {\bf 111} we need to deal with $\phi(h)$ separately.

The linear map $\phi\colon \Salg_0\to\Salg_0$ is defined in \eqref{eq:SF-phi-def}, where by abuse of notation we identify 
$\Salg_0$
 and $\algGr$ via 
 $\xp_k\idem_0 =a_k$, $\xm_k\idem_0 = b_k$.
Following \cite[Sec.~5.2]{Davydov:2012xg} $\phi$ can be written explicitly as, for $u \in \Salg_0$, 
\begin{align} \label{eq:phi}
	\phi(u)= \sum_{m=0}^{2\np} \rmi^{m(m+1)} \sum_{i_1\leq \ldots\leq i_m}
	\eps_{i_{1}}\cdots\eps_{i_{m}} ~\Lambda^\mathrm{co}_{\Salg_0} (\xx_{i_{1}}^{\eps_{i_{1}}}\cdots\xx_{i_{m}}^{\eps_{i_{m}}}\cdot u) 
	 \ \xx_{i_{1}}^{-\eps_{i_{1}}}\cdots \xx_{i_{m}}^{-\eps_{i_{m}}} \idem_0 \gc
\end{align}
where $\eps_{i_n}\in\{\pm 1\}$ such that $i_n=i_{n+1}\Rightarrow \eps_{i_n}>\eps_{i_n+1}$.
Furthermore, $\Lambda^\mathrm{co}_{\Salg_0} : \Salg_0 \to \oC$ is defined as in \eqref{eq:SF-coint-def}, again using the identification between $\Salg_0$ and $\algGr$. Explicitly, it is zero everywhere on $\Salg$ except in the top degree, where it takes the value
\be
	\Lambda^\mathrm{co}_{\Salg_0}\big( \xp_1\xm_1\cdots\xp_\np\xm_\np\idem_0
	\big) = 
	\beta^{-2} \ .
\ee
As in the definition of $\mathsf{A}$ in \eqref{eq:A-via-SN=1-def} choose
$\beta_1,\dots,\beta_\np$ such that $\beta_i^4 = -1$ and $\beta = \beta_1 \cdots \beta_\np$.
On each of the subalgebras $\Salg^{(k)}_0$ we can define the ``$\np=1$ version of $\phi$'' as follows:
\be
	\phi^{(k)}(\idem_0) =  \beta_k^2 \, \xp_k\xm_k\idem_0 \ , 
	\quad 
	\phi^{(k)}(\xx_k^{\pm}\idem_0)=   \beta_k^2 \,\xx_k^{\pm}\idem_0 \ , 
	\qquad 
	\phi^{(k)}(\xp_k\xm_k\idem_0)= - \beta_k^2 \, \idem_0 \ .
\ee
Indeed, for $\np=1$ we have $\phi = \phi^{(1)}$ (see also~\cite[Eqn.\,(33)]{Gainutdinov:2015lja}).

\begin{lemma}
We have, for $v^{(k)} \in \Salg^{(k)}_0$, 
\be
	\phi(v^{(1)} \cdots v^{(\np)}) ~=~  \phi^{(1)}(v^{(1)}) \cdots \phi^{(\np)}(v^{(\np)}) \ .
\ee
\end{lemma}

\begin{proof}
We start by describing $\phi$ in a different way: 
Let $T=\{1,\ldots ,\np\}\times \{-1,1\}$ be a totally ordered set whose ordering is given by, for $(i_1,\eps_1), (i_2,\eps_2)\in T$,
\be
(i_1,\eps_1) < (i_2,\eps_2) ~\iff ~ i_1<i_2 ~\lor~ (i_1=i_2 ~\land~ \eps_2<\eps_1) \gp
\ee
Let $K \subset T$ be a subset of the form $K=\{(j_1,\eps_{j_1}),\ldots ,(j_k,\eps_{j_k})\}$, 
	where the indexing 	agrees with the ordering in the sense that $(j_a,\eps_a) < (j_{a+1},\eps_{a+1})$.
We introduce the elements
\begin{align}
	u_K\coloneqq \xx_{j_1}^{\eps_{j_1}}\cdots\xx_{j_k}^{\eps_{j_k}} \idem_0 \qcq
		 u_{K}^-&\coloneqq \xx_{j_1}^{-\eps_{j_1}}\cdots\xx_{j_k}^{-\eps_{j_k}} \idem_0 \qcq u_\emptyset\coloneqq  \idem_0 \qcq 
		\beta_K=\beta_{j_1}\cdots\beta_{j_k} \ .
	\end{align}
For two subsets $K \subset L \subset T$ we define the linear map
	\begin{align}
		\phi^L(u_K)\coloneqq
		\beta_L
		\, \rmi^{|L|+2|K|} \, u^-_{L\setminus K} \gp  
	\end{align}
For $n=1,\dots,\np$ let
$L_n=\{(n,+),(n,-)\}$. We claim that
\begin{align*}
\text{(a)}\quad & \phi^{L_n}(u_K) = \phi^{(n)}(u_K) \quad \text{for $K \subset L_n$} \ , 
\\
\text{(b)}\quad & \phi^{\cup_{n=1}^r L_n}(u_K) = \phi^{\cup_{n=1}^{r-1} L_n}(u_{K\setminus L_r}) \cdot \phi^{L_r}(u_{K\cap L_r}) \ ,
\\
\text{(c)}\quad & \phi^T = \phi \ , 
\end{align*}
where for (b) we have $r = 1,\dots,\np$ and $K \subset \bigcup_{n=1}^r L_n$. Properties (a)--(c) together imply the statement of the lemma.

Properties (a) and (b) are immediate from the definition. 
It remains to show property (c).  Abbreviate $k=|K|$.
The cointegral
$\Lambda^\mathrm{co}_{\Salg_0}$ in \eqref{eq:phi} is non-zero only in top-degree, so that for a given $u_K$, the only non-zero summand in \eqref{eq:phi} is that with $\{(i_1,\eps_{i_1}),\ldots ,(i_m,\eps_{i_m})\} = T \setminus K$. Thus $\phi(u_K) = \pm u_{L\setminus K}^-$.
To determine the sign, we define
\begin{align}
	\eta_1(K)&= \rmi^{(2\np-k)(2\np-k+1)} 
	=  (-1)^\np \,\rmi^{(k-1)k}
	\gc \\ \nonumber
	 \eta_2(K) &=\prod_{(i,\eps_i )\in T\setminus K} \eps_i \qcq 	\eta_3(K) = \Big( \prod_{(n,\eps_n )\in K} -\eps_n \Big)\textbf{}
	 \cdot \Big( \prod_{m=1}^{k}(-1)^{ 2\np-k-m} \Big) \gp
\end{align}
The factors in  $\phi(u_K)$ from \eqref{eq:phi} correspond to $\eta_1,\eta_2,\eta_3$ as follows. 
We have $\beta^2 = \beta_T$ and
the product $\beta_T^{-1}\cdot\eta_1$ 
corresponds to the factor $\rmi^{m(m+1)}$, for $m=|T\setminus K| = 2\np-k$,
times the normalisation of $\Lambda^\mathrm{co}_{\Salg_0}$ from \eqref{eq:SF-coint-def}. The sign
$\eta_2$ corresponds to $\eps_{i_{1}}\cdots\eps_{i_{m}}$ where the $i_a$ run over values in the complement $T \setminus K$,
	and $\eta_3$ corresponds to the sign 
	we get when bringing $u_{T\setminus K}\cdot u_K$ in $\Lambda^\mathrm{co}_{\Salg_0}(\cdots)$ into the form $\xp_1\xm_1\cdots\xp_\np\xm_\np\idem_0$.
Altogether, the coefficient in front of $u_{L\setminus K}^-$ is
$\beta_T^{-1} \,\eta_1(K)\cdot \eta_2(K) \cdot\eta_3(K)$.
It is not hard to verify that
\begin{align} \eta_2(K)\cdot  \eta_3(K)
&= \Big(\prod_{(i,\eps_i )\in T\setminus K} \eps_i\Big) \, \Big(\prod_{(n,\eps_n )\in K} -\eps_n\Big)  \, \rmi^{(k-1)k} 
= \rmi^{(k-1)k} \cdot (-1)^{\np+k} \  .
\end{align}
and hence 
	$\eta_1(K)\cdot \eta_2(K) \cdot\eta_3(K) = (-1)^{k}$.
Thus $\phi(u_K) = \beta_T^{-1} (-1)^{k} u_{T\setminus K}^- =  \beta_T (-1)^{\np+k} u_{T\setminus K}^-  = \phi^T(u_K)$, as claimed.
\end{proof}

Using the above lemma, one can also write the two sides of sector {\bf 111} in Table \ref{tab:Lam-eqn-sectors} as products over elements in $(\Salg^{(k)})^{\ot 3}$ and conclude its validity from the $\np=1$ case.

Combining the above result with Propositions~\ref{prop:D-is-C-lin-equiv} and \ref{prop:functor-D-mult} we obtain:

\begin{proposition}\label{prop:functor-D-tensor}
The functor  $\mathcal D\colon \catSF \to \repS$, together 
with the isomorphisms $\Delta_{U,V}$ and $\id_{\oC^{1|0}} : \one \to \mathcal{D}(\one)$, is $\oC$-linear monoidal equivalence.
\end{proposition}

\subsection{Transporting the braiding}\label{sec:trans-braid}

Recall the definition of the element $r \in \Salg \ot \Salg$ in \eqref{eq:R-element-for-S}.
We  use $r$ to define a family of natural isomorphisms
\be
	\brS_{M,N} : M \ot N \to N \ot M
	\qquad , \quad M,N \in \repS 
\ee
in $\repS$
as a sum over sectors:\footnote{The appearance of the parity involution is the reason that $r$ is not a universal $R$-matrix for $\Salg$.}
\be\label{eq:RepS-braiding-sectors-def}
	\brS_{M,N}
	= \sum_{a,b \in \{0,1\}} \brS^{ab}_{M,N}
	\qquad , \quad
	\brS^{ab}_{M,N} = \sflip_{M,N} \circ 
	\big( r^{ab} \cdot \idem_a \ot \idem_b \big)
	\circ \big(\id_M \ot \omega_N^a \big) \ ,
\ee
where $\omega^0_N = \id_N$ and $\omega^1_N$ is the parity involution, and composition with $r^{ab} \cdot \idem_a \ot \idem_b$ denotes the action of this element of $\Salg \ot \Salg$ on $M \ot N$ (with the corresponding parity signs resulting from braiding one copy of $\Salg$ past $M$).

We will show that $\brS$ is the result of transporting the braiding from $\catSF$ to $\repS$ via 
the family of isomorphisms $\isoD_{U,V}$ introduced in Section~\ref{sec:funD},
 that is, it is the unique natural family of isomorphisms that makes the diagram \eqref{eq:transport-braiding-via-functorequiv} commute: 
\begin{equation}\label{eq:braiding-transport-comm-diag}
\xymatrix@R=22pt@C=42pt{
&\funD(U*V)\ar[r]^{\funD(c_{U,V})}\ar[d]^{\isoD_{U,V}}&\funD(V*U)\ar[d]^{\isoD_{V,U}}&\\
&\funD(U)\tensor\funD(V)\ar[r]^{\brS_{\funD(U),\funD(V)}}&\funD(V)\tensor\funD(U)&
}
\end{equation}
This then proves that $\brS$ is a braiding on $\repS$ and that 
$\funD$
 is a braided monoidal functor.

\begin{table}[bt]
\begin{align*}
&{\bf00}~:~&&  \sflip_{U,V}\Big( \sflip_{S,S}(\delta_0) \,.\, \gamma^{-1} \,.\, (u \tensor v) \Big)
=
\brS^{00}_{U,V}\Big( \delta_0 \,.\, (u\tensor v)\Big)
\\[.6em]
&{\bf01}~:~&&  
\sflip_{U,\B \ot V}\Big(
\big[\sflip_{\B,U}\big(
\copS(a) 
 \,.\, (\Bel \tensor \kappa)
  \,.\, (\one \tensor u)\big)\big]
  \tensor v
\Big)
\\
&&&=
\brS^{01}_{U,\B \ot V}\Big(
\big[\copS(a) 
 \,.\, (\delta_1 \tensor \Bel)
  \,.\,  (u \tensor \one) \big] \tensor v
\Big)
\\[.6em]
&{\bf10}~:~&&
\sflip_{\B \ot U, V} \circ (\id_\B \tensor \id_U \tensor \rho^V) 
\\
&&& \hspace{7em}
\circ (\id_\B \tensor \sflip_{S,U} \tensor \omega_V)  \Big(
\big[(\sflip_{S,S} \circ 
\copS(a) )
	\cdot (\Bel \tensor (\delta_1\kappa)) \big]
\tensor u \tensor v
\Big)
\\
&&&=
\brS^{10}_{\B \ot U,V}
\circ (\id_\B \tensor \id_U \tensor \rho^V) 
\circ (R_\Bel \tensor \sflip_{\Salg_0,U} \tensor \id_V)  \Big(
\copS(a) \tensor u \tensor v
\Big)
\\[.6em]
&{\bf11}~:~&&  
\beta \cdot 
	\sflip_{\B \ot U, \B \ot V} 
\circ (\id_\B \tensor \sflip_{\B,U} \tensor \omega_V) \Big(
\sflip_{\B,\B}\big[
\copS(h\kappa^{-1}) 
 \,\cdot\, (\Bel \tensor \Bel) \big]
  \tensor u \tensor v
\Big)
\\
&&&=
\brS^{11}_{\B \ot U,\B \ot V}
\circ(\id_\B \tensor \sflip_{\B,U} \tensor \id_V)
\Big(
\big[\copS(h) \,\cdot\, (\Bel \tensor \Bel)\big] \tensor u \tensor v
\Big)
\end{align*}
\caption{Conditions on  $\brS_{U,V}$ such that \eqref{eq:braiding-transport-comm-diag} commutes in each of the four sectors.
The conditions have to hold for all $h \in \Salg_0$, $a \in \B$, $u \in U$, $v \in V$ and all $U \in \catSF_i$, $V \in \catSF_j$.
The element $\gamma$ is defined as in Table~\ref{tab:Lam-eqn-sectors} (and its inverse is given in \eqref{gamma-inv}) and $\kappa$ is defined in \eqref{kappa}.
}\label{tab:R-eqn-sectors}
\end{table}

\newcommand{\W}{\Omega}

\begin{lemma}\label{lemma:psi-iso}
For all $M,N \in \repS$,
$\psi_{M,N}$ is a morphism in $\repS$. Furthermore it is invertible, natural in $M,N$ and  makes the diagram \eqref{eq:braiding-transport-comm-diag} commute.
\end{lemma}

\begin{proof}
Given morphisms $f : M \to M'$ and $g : N \to N'$ in $\repS$, it is immediate from the definition that $\psi_{M',N'} \circ (f \ot g) = (g \ot f) \circ \psi_{M,N}$. The lemma follows once we proved that \eqref{eq:braiding-transport-comm-diag} commutes. Indeed, since the top path in \eqref{eq:braiding-transport-comm-diag} is a morphism in $\repS$, so is $\psi_{\mathcal{D}(U),\mathcal{D}(V)}$. And since the top path is invertible, so is $\psi$.

We now show that \eqref{eq:braiding-transport-comm-diag} commutes. In Table~\ref{tab:R-eqn-sectors}, the conditions on $\psi_{M,N}$ are given in each of the four sectors. These conditions use the inverse of $\gamma$,
\begin{equation}\label{gamma-inv}
\gamma^{-1} = \exp(-C) = \prod_{i=1}^{\np}(\one\tensor\one - \xm_i\tensor \xp_i + \xp_i\tensor \xm_i - \xp_i \xm_i\tensor \xp_i \xm_i)\cdot\idem_0 \tensor \idem_0 \ ,
\end{equation}
and the constant
\begin{equation}\label{kappa}
\kappa := \exp\bigl(\ffrac12\hat C\bigr) =  \prod_{i=1}^{\np}(\one - \xp_i \xm_i ) \idem_0 \ .
\end{equation} 

Substituting the definition of $\psi$ in \eqref{eq:RepS-braiding-sectors-def}, the condition in sector {\bf 00} can be rewritten as
\begin{equation}\label{eq:S-braid-trans-sec00}
\RS^{00} \cdot \idem_0 \ot \idem_0 ~=~ \sflip_{S,S}(\delta_0) \cdot \gamma^{-1} \cdot  \delta_0^{-1} \ ,
\end{equation}
where $\delta_0$ and its inverse are given in \eqref{eq:delta0-delta1} and \eqref{eq:delta0-delta1-inv}.
Note that all ingredients in the above equality are written as products over elements in $\Salg^{(k)}$ for $k=1,\dots,\np$. This reduces the verification of \eqref{eq:S-braid-trans-sec00} to the case $\np=1$, in which case it is  a short calculation combining \eqref{eq:delta0-delta1}, \eqref{eq:delta0-delta1-inv}, \eqref{gamma-inv} and the definition of $\RS$ in \eqref{eq:R-element-for-S}.

In the remaining three sectors the strategy is the same: we show that the required identity can be written as a product over elements in $\Salg^{(k)}$. This implies that if the equality holds for $\np=1$, it holds for all $\np$. The case $\np=1$ can then be checked by hand or by computer algebra (which is what we did). Below we only explain the reduction to $\np=1$ and we omit the details of the verification for $\np=1$.

The condition on $\psi$ in the {\bf 01}-sector can be expressed via $r$ as the following equation in $\Salg_0\tensor\Salg_1$:
\begin{equation}
\sflip_{\Salg_0,\Salg_1}\bigl[\RS^{01} \cdot \copS(a) \cdot (\one\tensor\Bel) \bigr]= \copS(a) \cdot (\Bel\tensor\kappa\delta_1^{-1})\quad, \qquad a \in \B.
\end{equation}
The elements $\Bel$, $\delta_1$ and $\kappa$ are all given in product form, see~\eqref{eq:Bel}, \eqref{eq:delta0-delta1} and~\eqref{kappa}. This reduces the verification in sector {\bf 01} to the case $\np=1$.

We see from \eqref{eq:RepS-braiding-sectors-def} and Table~\ref{tab:R-eqn-sectors} that the situation in the {\bf 10}- and {\bf 11}-sectors is different because of the presence of the parity-involution $\omega$. 
In sector {\bf 10} the condition on $\psi$ is expressed in terms of $r$ as the following equation in $\Salg_1\tensor\Salg_0$:
\begin{equation}
\sflip_{\Salg_1,\Salg_0}\bigl[\RS^{10} \cdot (\id\tensor\omega) \bigl( \copS(a)\bigr) \cdot (\Bel\tensor\one) \bigr]= \copS(a) \cdot (\delta_1\kappa\tensor\Bel)\quad , \qquad a \in \B \ .
\end{equation}
As in sector {\bf 01} all elements are given in factorised form, and the above equality thus follows from the fact that it holds for $\np=1$.

To rewrite the condition on $\psi$ in sector {\bf 11} in terms of $r$, we have to relate 
$\omega_{\B \ot V}$,
 which appears in $\psi$ on the RHS of the condition in sector {\bf 11}, recall~\eqref{eq:RepS-braiding-sectors-def},
  and $\omega_V$, which appears on the LHS of the condition. This can be done as follows. Note that
\be
\omega_{\B}\otimes \omega_V = \omega_{\B\otimes V} \ .
\ee
We recall then the basis
\begin{align} 
	\B= \operatorname{span} \Bigl\{\Bel_{I=(i_1,\dots,i_{\np})}=\prod_{k=1}^{\np}(\xm_k)^{i_k} \xp_k \idem_1 \mid i_k\in\{0,1\}\Bigr\} \ ,
\end{align}
with $\omega(\Bel_I)=(-1)^{\np+\sum_k i_k}\Bel_I$ and so $\omega_{\B} = \W.(-)$ with
\begin{equation}\label{Omega-def}
\W = \prod_{i=1}^\np (\xm_i\xp_i-\xp_i\xm_i)\idem_1 \ .
\end{equation}
The braiding $\brS^{11}$ can be then expressed as 
\begin{align}\label{brS11}
\brS^{11}_{\B\tensor U,\B\tensor V} &= \sflip_{\B\tensor U,\B\tensor V} \circ \RS^{11}\circ\bigl(\id_{\B\tensor U}\tensor \omega_{\B \ot V} \bigr) 
\\ \nonumber
&= \sflip_{\B\tensor U,\B\tensor V}\circ (\RS^{11}\cdot \one\tensor\W )\circ (\id_{\B\tensor U}\tensor \id_\B \ot \omega_{V}) \ .
\end{align}
	In the last line, the element $\RS^{11}\cdot \one\tensor\W  \in \Salg \ot \Salg$ acts on the two tensor factors $\B$ in $\B \ot U \ot \B \ot V$.
Substituting this into the condition in Table~\ref{tab:R-eqn-sectors} gives the following condition on $\RS^{11}$:
\begin{equation}\label{eq:r11}
\sflip_{\Salg_1,\Salg_1}\bigl[(\RS^{11}\cdot \one\tensor\W ) \cdot  \copS(h) \cdot (\Bel\tensor\Bel) \bigr]= \beta \cdot \copS(h\kappa^{-1}) \cdot (\Bel\tensor\Bel)\quad , \qquad h\in \Salg_0 \ .
\end{equation}
Again one can write this equality as a product over elements in $\Salg^{(k)}$, $k=1,\dots,\np$, reducing the verification to $\np=1$.
\end{proof}

Since $\catSF$ is a braided monoidal category, the above lemma shows that $\psi$ defines a braiding on $\repS$. Altogether we have shown:

\begin{proposition}\label{prop:functor-D-tensor-br}
The functor $\mathcal D$ in Proposition~\ref{prop:functor-D-tensor} is braided monoidal.
\end{proposition}

%%%%%%%%%%%%%%%%%%%

\section{Equivalence between \texorpdfstring{$\catSF$}{SF} and \texorpdfstring{$\rep\Q$}{Rep(Q)}} \label{app:Q-S}

Here, we present the second part of the proof of Lemma~\ref{lem-trans-QSF}. We begin with introducing a quasi-bialgebra $\QQ$ in $\vect$ and show a braided monoidal equivalence between $\rep \Salg$ and $\rep \QQ$. Then we present a twisting of $\QQ$ into $\Q$, and therefore $\rep \Q$ is braided monoidally equivalent to $\rep \Salg$ and thus to $\catSF$.
Finally, we use this equivalence to transport the ribbon twist from $\catSF$ to $\rep \Q$.

\subsection{The quasi-bialgebra \texorpdfstring{$\QQ$}{Q hat}}

In this section we introduce the quasi-bialgebra
$\QQ=\QQ(\np,\beta)$ which is equal to $\Q$ as an algebra. It has a different quasi-bialgebra structure, namely
\begin{align} \label{def:co-prod-Qhat}
	\hat\Delta(\K)&=\K\tensor\K-(1+(-1)^\np) \idem_1\tensor\idem_1\cdot \K\tensor\K \gc & \eps(\K)&=1 \gc \\ \nonumber
	\hat\Delta(\fpm_i)&=\fpm_i\tensor\one + \K^{-1}\tensor\fpm_i-(1+(-1)^\np) \idem_1\tensor\idem_1\cdot \K^{-1}\tensor\fpm_i \gc  &  \eps(\fpm_i)&=0 \gc
\end{align}
where we use the central idempotents  $\idem_i\in\QQ$ ($i=0,1$)  defined as in~\eqref{def:idem Q}.
The coassociator for this coproduct is 
	(we will check the axioms below)
\begin{align}\label{eq:Qhat-coass-I}
\asQQ &= \idem_0\tensor\idem_0\tensor\idem_0 + \idem_0\tensor\idem_0\tensor\idem_1 + \idem_1\tensor\idem_0\tensor\idem_0
+ \idem_0\tensor\idem_1\tensor\idem_1\\ \nonumber
&\qquad + \idem_1\tensor\idem_1\tensor\idem_0 
+ \asQQ^{010}\idem_0\tensor\idem_1\tensor\idem_0
+ \asQQ^{101}\idem_1\tensor\idem_0\tensor\idem_1
+ \asQQ^{111}\idem_1\tensor\idem_1\tensor\idem_1\ ,
\end{align}
where the non-trivial components are given by
\begin{align}\label{eq:Qhat-coass-II}
\asQQ^{010}=&\prod_{k=1}^{\np}\hat\as^{010}_{(k)} \gc \\ \nonumber
\asQQ^{101}=& (-\K)^{\np-1}\tensor \K^{\np-1} \tensor\one \cdot
\Big(\prod_{k=1}^{\np} \hat\as^{101}_{(k)}  \Big) \cdot  \K^{\np-1} \tensor\K\tensor\one  \gc\\ \nonumber
\asQQ^{111}=&-\rmi^{\np}\beta^{2} \, \K^{\np-1}\tensor\one\tensor\one \cdot
\Big(\prod_{k=1}^{\np} \hat\as^{111}_{(k)} \Big) \cdot \K^{\np-1}\tensor\K^\np\tensor\one \gc
\end{align}
with
\begin{align}\label{eq:Qhat-coass-III}
\hat\as^{010}_{(k)} =~& \Bigl(\one\tensor\one\tensor\one
    + (1 + \rmi) \ff^+_k\K\tensor\K\tensor\ff^-_k\Bigr)\Bigl(\one\tensor\one\tensor\one + (1 - \rmi) \ff_k^-\K\tensor\K\tensor\ff_k^+\Bigr)\gc\\ \nonumber
\hat\as^{101}_{(k)} =~&\Bigl( \one\tensor\one\tensor\one
+ (1+\rmi) \one\tensor\ff_k^+\K\tensor\ff_k^-
+ (1 - \rmi) \ff_k^-\K\tensor\ff_k^+\tensor\one
\Bigr)\\ \nonumber
~&\times\Bigl( \one\tensor\one\tensor\one
 + (1 + \rmi)\ff_k^+\K\tensor\ff_k^-\tensor\one
 + (1 - \rmi)\one\tensor\ff_k^-\K\tensor\ff_k^+ \Bigr) \gc\\ \nonumber
\hat\as^{111}_{(k)} =~& 
\Bigl( \one\tensor\one\tensor\one
+(\rmi-1)\bigl(\one\tensor\ff_k^+\K\tensor\ff_k^-
 +  \ff_k^+\K\tensor\K\tensor\ff_k^-
 - \ff_k^+\K\tensor\ff_k^-\tensor\one
  + \one\tensor\ff_k^-\ff_k^+\tensor\one\bigr)
\Bigr)\\ \nonumber
~&\times\Bigl( \one\tensor\one\tensor\one
 - (\rmi-1)\bigl( \one\tensor\ff_k^-\K\tensor\ff_k^+
 + \ff_k^-\K\tensor\K\tensor\ff_k^+
   -  \ff_k^-\K\tensor\ff_k^+\tensor\one
   - \one\tensor\ff_k^-\ff_k^+\tensor\one\bigr)
  \Bigr)\\ \nonumber
  ~&\times\Bigl( \one\tensor\one\tensor\one
   - 2 \ \one\tensor\ff_k^-\ff_k^+\tensor\one
   \Bigr)
  \gp
\end{align}
The quasi-bialgebra $\QQ$ can also be equipped with an $R$-matrix which is defined as
\begin{equation} \label{eq:hatR}
\hat R = \hat R^{00}\idem_0\tensor\idem_0 + \hat R^{01}\idem_0\tensor\idem_1
+ \hat R^{10}\idem_1\tensor\idem_0 + \hat R^{11}\idem_1\tensor\idem_1 \ ,
\ee
where 
\begin{align}
\hat R^{0i} &= \rho(\K) \cdot \prod_{k=1}^{\np} \hat{R}^{0i}_{(k)}\ , \qquad i=0,1 \qc   \label{R0j-exp}
   \\   
\hat R^{10} &= \rho(\K)\cdot \prod_{k=1}^{\np} \hat{R}^{10}_{(k)} \cdot \one\tensor\K \gc \label{R10-exp}\\  
\hat R^{11} &= (-1)^\np \, \rmi \beta \, \rho(\K)\cdot \one\tensor \K^{\np-1}\cdot \Big(\prod_{k=1}^{\np} \hat{R}^{11}_{(k)}\Big) \cdot \K^\np\tensor\one  \label{R11-exp}
\end{align}
with the Cartan part
\begin{equation}\label{sflip-om}
\rho(\K)=  \half(\one\tensor\one + \omega\tensor\one + \one\tensor\omega - \omega\tensor\omega) \qcq \omega =  (\idem_0-\rmi\idem_1)\K \gc
\end{equation}
and
\begin{align} \label{eq:hatRk}
\hat{R}^{00}_{(k)} &= \one\tensor\one - 2\ff^-_k\K\tensor\ff^+_k \gc\\ \nonumber
\hat{R}^{01}_{(k)} &=\one\tensor\one - (1+\rmi)\ff^-_k\K\tensor\ff^+_k -(1+\rmi)\ff^+_k\K\tensor\ff^-_k  +(1-\rmi)\ff^-_k\ff^+_k\tensor\one + 2\rmi\ff^-_k\ff^+_k\tensor\ff^-_k\ff^+_k\gc \\\nonumber
  \hat{R}^{10}_{(k)} &= \one\tensor\one + (1 + \rmi) \ff^-_k\K\tensor\ff^+_k + (1 + \rmi) \ff^+_k\K\tensor\ff^-_k  
   + (1 + \rmi) \one\tensor\ff^-_k\ff^+_k - 2 \rmi \ff^-_k\ff^+_k\tensor\ff^-_k\ff^+_k \gc \\ \nonumber
     \hat{R}^{11}_{(k)} &= \one\tensor\one - 2\rmi  \ff^-_k\K\tensor\ff^+_k
+ (\rmi-1)\one\tensor\ff^-_k\ff^+_k 
- (1 + \rmi)\ff^-_k\ff^+_k\tensor\one +
 2  \ff^-_k\ff^+_k\tensor\ff^-_k\ff^+_k \gp
\end{align}

\begin{remark}\label{rem:Phi-R-N1}
For $\np=1$ we have
\begin{align}
\asQQ^{010} =\hat\as^{010}_{(k=1)} \qcq
\asQQ^{101} =\hat\as^{101}_{(k=1)} \cdot \one\tensor\K\tensor\one  \qcq
\asQQ^{111} = \frac{\beta^{2}}{\rmi} \hat\as^{111}_{(k=1)} \cdot \one\tensor\K\tensor\one   
\end{align}
 and they coincide with the components of the associator 
 in~\cite[Sec.\,7.4]{Gainutdinov:2015lja}. 
So, we have the simple factorisation~\eqref{eq:Qhat-coass-II} of the nilpotent (off-diagonal) part of the associator  into the product of $\np=1$ components. 
The Cartan part depends only on the parity of $\np$.
We  also note  that the components $\hat\as^{abc}_{(i)}$ commute with each other: $\hat\as^{abc}_{(i)} \cdot \hat\as^{abc}_{(j)}=\hat\as^{abc}_{(j)} \cdot \hat\as^{abc}_{(i)}$.
\\
For the $\hat R$ element, we note that~\eqref{R0j-exp} for $\np=1$ is 
$\hat R^{0i} =  \rho(\K)  \hat{R}^{0i}_{(k=1)}$, 
and it is the {\bf 00} and  {\bf 01} components of the universal  $R$-matrix 
in~\cite[Sec.\,7.7]{Gainutdinov:2015lja}. The Cartan part thus does not change with $\np$ while the ``off-diagonal'' part is just the product of the $\np=1$ components.
Also, $\hat{R}^{1i}_{(k=1)}$, for $i\in\{0,1\}$, corresponds to $\np=1$ case: the {\bf 10} component of the universal 
  $R$-matrix
in~\cite[Sec.\,7.7]{Gainutdinov:2015lja} is expressed as $\hat R^{10} =  \rho(\K)\cdot \hat{R}^{10}_{(k=1)} \cdot \one\tensor\K$ while  the {\bf 11} component is $\hat R^{11} =  \frac{\beta}{\rmi}\rho(\K)\cdot \hat{R}^{11}_{(k=1)} \cdot \K\tensor\one$ that agrees with~\eqref{R10-exp}-\eqref{R11-exp}. Hence, we have again the simple factorisation~\eqref{R10-exp}-\eqref{R11-exp}  of the nilpotent (off-diagonal) part of $\hat R$  into the product of $\np=1$ components. 

\end{remark}

The following lemma can be easily checked by a direct calculation. 
\begin{lemma}\label{lem:DeltaQhat-alg-map}
The map $\hat\Delta\colon \QQ\to \QQ\otimes \QQ$ defined in \eqref{def:co-prod-Qhat} is an algebra homomorphism.
\end{lemma}

In order to show that $(\QQ, \cdot,\one,\hat\Delta, \varepsilon, \hat\Phi, \hat R)$ is a quasi-triangular quasi-bialgebra we start with the braided monoidal category $\repS$ described in  Appendix~\ref{app:SF-S} and
	verify that
$\hat\Phi$ and $\hat R$ 
	can be obtained via
transport along a certain multiplicative functor from $\repS$ to $\repQQ$, see Sections~\ref{sec:RepS-RepQ}--\ref{sec:trans-braid-SQ}.
From this it follows that $\QQ$ is indeed a quasi-triangular quasi-bialgebra and that $\repQQ$ is braided equivalent to $\repS$.

%%%%%%%%%%%%%%%%%%%
\subsection{A $\oC$-linear equivalence from \texorpdfstring{$\repS$}{RepS} to \texorpdfstring{$\repQQ$}{Rep(Q-hat)}}\label{sec:RepS-RepQ}

In this section we present a $\oC$-linear functor $\funSQ: \repS \to \repQQ$. 
Recall that $\omega_U$ denotes the parity involution on the super-vector space $U$.
For a given $U\in\repS$, $\funSQ(U)$ is the underlying vector space $U$ with $\QQ$-action given by, for $u \in U$,
\begin{align}\label{eq:funSQ}
\K. u &~:=~ \z . \omega_U(u)  = \omega_U(\z . u) ~~ , \quad \text{where} \quad \z=\idem_0+\rmi\idem_1 \gc
\\
\nonumber
\ff^{\pm}_i. u &~:=~ \xpm_i . u \gp
\end{align}
Note that $\K^2$ acts as $\LL$, that is, $\idem_i\in\QQ$ ($i=0,1$) acts as $\idem_i\in\Salg$.
For a morphism $f : U \to V$ in $\repS$ we set $\funSQ(f) = f$.

\begin{proposition}\label{prop:G-equiv}
The functor $\funSQ$ is an equivalence of $\oC$-linear categories.
\end{proposition}
\begin{proof}
Since the proof is very similar to the proof of \cite[Prop.~5.2]{Gainutdinov:2015lja} we omit the details here.
However, for later reference we introduce 
a functor $\funQS:\repQQ\to\repS$ which is inverse to~$\funSQ$.
By inverting the relation in \eqref{eq:funSQ} we define an involution map on an object $V\in\repQQ$ as
\be\label{Z2-repQ}
\omega_V(v) := (\idem_0-\rmi\idem_1)\K. v \qcq v\in V \gp
\ee
Then the $\Salg$-module $\funQS(V)$ has $V$ as the underlying super-vector space with the $\oZ_2$-grading defined by the eigenvalues of $\omega_V$ as $\omega_V(v)=(-1)^{|v|}v$, 
for an eigenvector $v$ of $\omega_V$.
Moreover, $\LL$ acts on $\funQS(V)$ by $\idem_0-\idem_1\in\QQ$ and $\xpm_i$ acts on $\funQS(V)$ by $\ff^{\pm}_i$.
\end{proof}

Since $\Q$ and $\QQ$ have the same algebra structure we in fact have shown a $\oC$-linear equivalence of $\repS$ and $\repQ$.

\subsection{\texorpdfstring{$\funSQ$}{G} as multiplicative functor}

In order to show that $\funSQ$ is multiplicative, we define the family of isomorphisms 
\begin{align} 
\begin{split}
\isoG_{U,V} \colon \; \funSQ(U\Stensor V) &\to \funSQ(U) \tensor_{\repQQ} \funSQ(V) \gc \\ 
u\tensor v \label{eq:isoG}
&\mapsto u \tensor v + \idem_1. u \tensor (\xi-1)\idem_1. v \gc 
\end{split}
\end{align}
where $\xi$ is defined as
\begin{equation}\label{xi-def}
\xi =	\rmi^{\np(\np-1)/2}  
\prod_{k=1}^\np 
\xi_k \quad \text{with}\quad  \xi_k = \xp_k + \xm_k\  .
\end{equation}
Invertibility is easy to see since $(\isoG_{U,V})^2 = \id_{U \otimes V}$, which follows from $\xi^2 = \idem_1$.
Naturality is also clear. It remains to prove the following lemma.

\begin{lemma}
$\isoG_{U,V}$ is an intertwiner of $\QQ$-modules.
\end{lemma}

\begin{proof}

We need to show that for all $a \in \QQ$, $u \in U$, $v \in V$ we have
\begin{equation}\label{eq:GammaUV-Q-intertw}
	\isoG_{U,V}(a\,\hat.\,(u \otimes v)) = a.\isoG_{U,V}(u \otimes v) \ ,
\end{equation}
where the notation $\,\hat.\,$ emphasises that the action of $\Delta(a) \in \Salg \ot \Salg$ on $U \ot V$ is in $\svect$ and involves parity signs.

Since $\Delta^{\Salg}$ and $\hat\Delta$ are algebra maps, it is enough to verify this on the generators $\K$, $\ff^\pm_k$. 
If $U\notin\repS_1$ or $V\notin\repS_1$, $\isoG_{U,V}$ is just the identity, and the 
	verification is straightforward.
As an example for the sector \textbf{11} case let $a=\ff^+_k$ and assume $\np$ to be even. 
Then the two sides of the above identity are
\begin{align} 
\ff^+_k . \isoG_{U,V}(u\tensor v) 
&= (\ff^+_k \tensor \one + \K^{-1} \tensor \ff^+_k -2\K^{-1}\idem_1\tensor\ff^+_k\idem_1 ) 
	.
(\idem_1 . u  \tensor \xi \idem_1 . v ) \\ \nonumber
&= \xp_k  \idem_1 . u \tensor \xi \idem_1 . v + \rmi (-1)^{|u|}  \idem_1 . u \tensor \xp_k \xi \idem_1 . v \ , \\ \nonumber
\isoG_{U,V}\big(\ff^+_k \,\hat.\, (u \otimes v)\big)
  &=\isoG_{U,V}\big(\xp_k 
  \,\hat.\, 
  (u \otimes v)\big)
   = \isoG_{U,V}\big((\xp_k \otimes\one - \rmi\one\otimes \xm_k) 
   \,\hat.\, 
   (u \otimes v)\big) \\ \nonumber
	&= \xp_k \idem_1. u \otimes\xi\idem_1. v - \rmi (-1)^{|u|} \idem_1. u\tensor \xi\xm_k \idem_1. v \ .
\end{align}
Since $ \xpm_k \xi \idem_1 = (-1)^{\np-1} \xi \xmp_k \idem_1$ both sides are equal.
The calculations for the other generators and for odd $\np$ are equally straightforward.
\end{proof}

Altogether, we have shown:

\begin{proposition}\label{prop:funSQ-mult}
	With the isomorphisms $\isoG_{U,V}$ as in~\eqref{eq:isoG}, the functor $\funSQ: \repS \to \repQQ$ is multiplicative.
\end{proposition}

%%%%%%%
\subsection{Transporting the associator}
In this section we 
transport the associator from $\repS$ to $\repQQ$ 
along the lines explained around the diagram \eqref{eq:transport-assoc-diag}.
As $\repQQ$ is the category of (finite-dimensional) $\Qhat$-modules in vector spaces, the associator on $\repQQ$ takes the form
\begin{equation}\label{eq:assoc-RepQ-via-Phi}
 \assocQQ_{U,V,W}(u\tensor v\tensor w)
 = 
 \asQQ. (u\tensor v\tensor w)  \gc
\end{equation}
where $u,v,w$ are elements of $U,V,W\in\repQQ$ and for some $\asQQ \in \QQ \tensor \QQ \tensor \QQ$. 
In order to compute~$\hat \Phi$
and to see that it agrees with~\eqref{eq:Qhat-coass-I}-\eqref{eq:Qhat-coass-III},
 we  choose $U=V=W=\Qhat$
and  solve the diagram \eqref{eq:transport-assoc-diag}  with $\fun$ replaced by $\funSQ$.

Recall the functor $\funQS$ inverse to $\funSQ$ from the proof of Proposition~\ref{prop:G-equiv}. 
Let us abbreviate $\Qhat_{\funQS} := \funQS(\Qhat)$. 
The $\Salg$-module structure on $\Qhat_{\funQS}$ is 
as explained in the proof of Proposition~\ref{prop:G-equiv}.
Commutativity of the diagram \eqref{eq:transport-assoc-diag} 
applied for $\fun=\funSQ$ and $\Theta = \isoG$
then reads 
for  $q\in\Qhat_{\funQS}^{\tensor3}$:
\begin{equation}\label{eq:Phi-from-Lambda}
\big( (\isoG_{\Qhat_{\funQS},\Qhat_{\funQS}} \tensor \id) \circ \isoG_{\Qhat_{\funQS} \tensor\Qhat_{\funQS},\Qhat_{\funQS}}\big)\big(\Sas \,\hat. \, q \big)
=
\asQQ \cdot \big[ (\id\tensor\isoG_{\Qhat_{\funQS},\Qhat_{\funQS}}) \circ \isoG_{\Qhat_{\funQS},\Qhat_{\funQS} \tensor\Qhat_{\funQS}}
(q) \big] \gc
\end{equation}
where $\isoG_{U,V}$ is defined in \eqref{eq:isoG} and the notation 
$\hat.$ emphasises that the action of $\Sas$ on $\Qhat_{\funQS} \otimes \Qhat_{\funQS} \otimes \Qhat_{\funQS}$ is as in $\svect$, 
	i.e.\ involves parity signs.
During the calculation it is important to be careful with the parity signs. 
For example, $\Sas \,\hat. \,(\one \tensor \one \tensor \one)$ can not be simplified to~$\Sas$ since $\one \in \Qhat_{\funQS}$ is not of definite parity.
Other examples are the actions of $\isoG_{\Qhat_{\funQS},\Qhat_{\funQS} \tensor\Qhat_{\funQS}}$ and
$\isoG_{\Qhat_{\funQS}\ot\Qhat_{\funQS} ,\Qhat_{\funQS}}$, which are, for $a,b,c \in \Qhat_{\funQS}$,
\begin{align}\label{eq:isoGabc}
\isoG_{\Qhat_{\funQS},\Qhat_{\funQS} \tensor\Qhat_{\funQS}}(a \tensor b \tensor c) &= 
a \tensor b \tensor c + 
\idem_1. a \tensor \big[ \copS((\xi-1)\idem_1) \,\hat.\, (b \otimes c) \big] \gc \\ \nonumber
\isoG_{\Qhat_{\funQS}\ot\Qhat_{\funQS} ,\Qhat_{\funQS}}(a \tensor b \tensor c) &= 
a \tensor b \tensor c + 
\copS(\idem_1). (a\ot b) \tensor \big[ (\xi-1)\idem_1
\,.\,c \big] \gp
\end{align}

\newcommand{\tPhi}{\underline{\hat\Phi}}
\newcommand{\tasQQ}{\underline{\hat\Phi}}
\newcommand{\tLambda}{\underline\Lambda}
We  use below the notations (in the spirit of~\eqref{eq:tilde-Lambda}) 
\be\label{eq:tilde-Phi}
\tPhi^{abc}=\hat\Phi^{abc} \cdot (\idem_a\tensor\idem_b\tensor\idem_c)
\ee
 with $\hat\Phi^{abc}$ as in \eqref{eq:Qhat-coass-I}
	(it is understood that the $\hat\Phi^{abc}$ not spelled out explicitly in \eqref{eq:Qhat-coass-I} are set to $\one^{\otimes 3}$). 
	We will also use the similar convention for $\tPhi^{abc}_{(k)}$  and $\tLambda^{abc}_{(k)}$.

In the following we will present the calculation of $\hat\Phi$
 in sector \textbf{101}. The other cases are similar. 
For brevity, we write $\omega$ instead of $\omega_\Salg$ for the parity involution in $\Salg$. 
Using~\eqref{eq:isoGabc} and~\eqref{eq:tilde-Phi},
	the equality in \eqref{eq:Phi-from-Lambda} 
then reduces to
\begin{align}
\tPhi^{101} \, .\, \big(\oneS\tensor\copS(\xi) \,\hat.\, (\omega^\np\tensor\id\tensor\id)(q)\big) &= \oneS\tensor\oneS\tensor\xi \, .\, (\underline\Lambda^{101} \,\hat.\, q)\\ \nonumber
&=(\omega^\np\tensor\omega^\np\tensor\id)\big(\oneS\tensor\oneS\tensor\xi \,\hat.\, (\underline\Lambda^{101} \,\hat.\, q)\big) \gc
\end{align}
where the parity involutions appear upon converting the action ``$\,.\,$'' in $\vect$ to the action ``$\,\hat.\,$'' in $\svect$, using that $\xi$ is odd if and only if $\np$ is odd. 
By setting
\be
q = \oneS\tensor\copS(\xi) \,\hat.\, (\omega^\np\tensor\id\tensor\id)(p)
\ee
for $p\in\QQ_{\funQS}^{\tensor 3}$
	(note that since $\xi^2=\idem_1$ the above map is a bijection between $p$'s and $q$'s), 
we get 
\begin{align}\label{eq:trans-assoc-sec-101}
\tasQQ^{101} .\,
p &= (\omega^\np\tensor\omega^\np\tensor\id)\Big(\oneS\tensor\oneS\tensor\xi \,\hat\cdot\, 
\tLambda^{101} \,\hat\cdot\, \oneS\tensor\copS(\xi) \,\hat.\, (\omega^\np\tensor\id\tensor\id)(p) \Big)   \\ \nonumber
	&= (\omega^\np\tensor\omega^\np\tensor\id)\Big(\prod_{k=1}^\np \big(\oneS\tensor\oneS\tensor\xi_k \,\hat\cdot\, 
	\tLambda^{101}_{(k)} \,\hat\cdot\, \oneS\tensor\copS(\xi_k)\big)\,\hat.\, (\omega^\np\ot\id\tensor\id)  (p)\Big)  \gc
\end{align}
where for the second equality we used  the factorisation in~\eqref{eq:Lambda-abc} and~\eqref{xi-def} and the fact that $\tLambda^{101}_{(k)}$ are even elements, as well as the equality 
\be
(\oneS\tensor\oneS\tensor\xi) \,\hat\cdot\,  (\oneS\tensor\copS(\xi)) = \prod_{k=1}^\np \big(\oneS\tensor\oneS\tensor\xi_k \,\hat\cdot\,  \oneS\tensor\copS(\xi_k) \big)
\gc
\ee
which follows from reordering the parity-odd elements $\xi_k$ and $\copS(\xi_k)$.
We now take $\asQQ^{101}$ in the form~\eqref{eq:Qhat-coass-II} and check the above equality.
First, we know that this equality holds in the $\np=1$ case~\cite[Sec.~7.4]{Gainutdinov:2015lja}, which takes the form
\be\label{eq:Phik-equality}
	\tasQQ^{101}_{(k)} \cdot \one\ot\K\ot\one 
	\,.\, 
	p 
	= (\omega\tensor\omega\tensor\id)\big(\oneS\tensor\oneS\tensor\xi_k \,\hat\cdot\,\tLambda^{101}_{(k)} \,
\hat\cdot\,
 \oneS\tensor\copS(\xi_k)\big)   
	\,\hat.\, 
(\id\tensor\omega\tensor\id)(p) 
\ee
where $k=1$ and  one has to use~\eqref{eq:Qhat-coass-III} and the convention in~\eqref{eq:tilde-Phi}.
We also note that the equality~\eqref{eq:Phik-equality} holds for general $\np$ and $k$.
 It can be rewritten as
\be\label{eq:trans-assoc-sec-101-2}
	\tasQQ^{101}_{(k)} \cdot \one\ot\K\ot\one
	\,.
	(\id\tensor\omega\tensor\id) (p) 
= (\omega\tensor\omega\tensor\id)\big(\oneS\tensor\oneS\tensor\xi_k \,\hat\cdot\,\tLambda^{101}_{(k)} \,
\hat\cdot\,
 \oneS\tensor\copS(\xi_k)\big)  \,\hat.\, p   \gp 
\ee
	By choosing $p=\one^{\ot 3}$ and applying \eqref{eq:trans-assoc-sec-101-2} multiple times on   \eqref{eq:trans-assoc-sec-101} it follows that
(recall that $\one \in \Qhat_{\funQS}$ is not of definite parity)
\begin{align}
\tasQQ^{101}=&\left(\omega^{\np-1}\tensor\omega^{\np-1}\tensor\id\right) \Big(\tasQQ^{101}_{(1)}  \cdot \one\ot\K\ot\one 
\,.\,
 (\id\tensor\omega\tensor\id) \Big(\tasQQ^{101}_{(2)}  \cdot \one\ot\K\ot\one 
\,.\, 
(\id\tensor\omega\tensor\id)  \\ \nonumber
&\cdots \Big(\tasQQ^{101}_{(\np-1)}\cdot \one\ot\K\ot\one
\,.\,
 (\id\tensor\omega\tensor\id) \left(\tasQQ^{101}_{(\np)}\cdot \one\ot\K\ot\one 
	\,.\,
	\omega^{\np-1}(\one)\tensor\one\tensor\one\right) \Big)\cdots\Big) \gp
	\end{align}
By identifying $\omega$ with the element $(\idem_0-\rmi\idem_1)\K$, 
the above expression simplifies to~\eqref{eq:Qhat-coass-II}.
Note that $\omega\cdot\idem_0=\K\cdot\idem_0$ and hence $\one\ot\K\ot\one\cdot\one\ot\omega\ot\one\cdot \idem_1\ot\idem_0\ot\idem_1=\idem_1\ot\idem_0\ot\idem_1$.

Altogether, we have shown:
\begin{proposition}\label{prop:repS-repQ-monoid}
The natural isomorphism $\assocQQ$ from~\eqref{eq:assoc-RepQ-via-Phi} with $\hat\Phi$ as in \eqref{eq:Qhat-coass-I}--\eqref{eq:Qhat-coass-III} defines an associator on $\repQQ$. With respect to this associator, 
the equivalence $\funSQ: \repS \to \repQQ$ with multiplicative structure $\isoG_{U,V}$ defined in~\eqref{eq:isoG} 
and $\id_{\oC} : \one \to \funSQ(\one)$, is $\oC$-linear monoidal equivalence.
\end{proposition}

\subsection{Transporting the braiding} \label{sec:trans-braid-SQ}
We now similarly transport the braiding along the monoidal equivalence $\funSQ: \repS\to\repQQ$ from Proposition~\ref{prop:repS-repQ-monoid},
recall the discussion above~\eqref{eq:transport-braiding-via-functorequiv}.

\begin{lemma}
For all $U,V \in \repQQ$, the isomorphisms
$\tau_{U,V}\circ \hat R$ with $\hat R$ from~\eqref{eq:hatR}
  are natural in $U,V$ and  make 
the diagram \eqref{eq:transport-braiding-via-functorequiv} commute for all $M,N\in\repS$: 
	\begin{equation}\label{eq:braiding-transport-comm-diag-SQ}
	\xymatrix@R=32pt@C=55pt{
	&\funSQ(M\ot_{\repS} N)\ar[r]^{\funSQ(\brS_{M,N})}\ar[d]^{\isoG_{M,N}}&\funSQ(N\ot_{\repS}M)\ar[d]^{\isoG_{N,M}}&\\
	&\funSQ(M)\ot_{\repQQ}\funSQ(N)\ar[r]^{\tau_{\funSQ(M),\funSQ(N)}\circ \hat R}&\funSQ(N)\ot_{\repQQ}\funSQ(M)&
	}
	\end{equation}
\end{lemma}
\begin{proof}
It is clear that for morphisms $f : U \to U'$ and $g : V \to V'$ in $\repQQ$ 
	we have
$\tau_{U',V'}\circ \hat R \circ (f \ot g) = (g \ot f) \circ \tau_{U,V}\circ \hat R$.
The lemma follows once we proved that \eqref{eq:braiding-transport-comm-diag-SQ} commutes, as it was already argued in the proof of 
Lemma~\ref{lemma:psi-iso}.

Though the monoidal isomorphisms $\isoG$ are non-trivial only in the {\bf 11}-sector the transport is non-trivial in all the four sectors due to the passage from $\svect$ to $\vect$. 
We recall first the braiding~\eqref{eq:RepS-braiding-sectors-def} in $\repS$. The commutativity of the diagram~\eqref{eq:braiding-transport-comm-diag-SQ} in the {\bf 00} and  {\bf 01} sectors then corresponds to the equation
\begin{gather}\label{eq:R0j}
\flip_{U,V}\circ \hat R^{0i} \,.\, (u\tensor v) = \sflip_{U,V}\circ \RS^{0i} \smult (u\tensor v),\quad \qquad i\in\{0,1\},\quad u\in U,\quad v\in V,
\end{gather}
where $\tau$ is the symmetric braiding in vector spaces.
This equation for $\np=1$ case holds due to~\cite{Gainutdinov:2015lja}, see also Remark~\ref{rem:Phi-R-N1}. Therefore, using the factorised expression for $\hat R^{0i}$ and $\RS^{0i}$, 
the equation~\eqref{eq:R0j} has  the unique solution 
as in~\eqref{R0j-exp},
where the first factor $\rho(\K)$ defined in \eqref{sflip-om} is due to  the braiding in super-vector spaces.

Recall then the expression of the braiding in the {\bf 10} and  {\bf 11} sectors of~$\repS$
in~\eqref{eq:RepS-braiding-sectors-def}.
The commutativity of the diagram~\eqref{eq:braiding-transport-comm-diag-SQ} corresponds thus to the equations
\begin{align}
\flip_{U,V}\circ \hat R^{10} \,.\, (u\tensor v) &= \sflip_{U,V}\circ \RS^{10} \smult (\id\tensor\omega) \bigl[u\tensor v\bigr],\\ \nonumber
\flip_{U,V}\circ \hat R^{11} \,.\, (u\tensor v) &=(\id\tensor L_{\xi}) \circ \sflip_{U,V}\circ \RS^{11}\smult(\id\tensor\omega) \bigl[u\tensor \xi \,.\, v\bigr],
\end{align}
 with the solutions
given in~\eqref{R10-exp} and~\eqref{R11-exp}, correspondingly.
The derivation of~\eqref{R10-exp} is obvious. To derive~\eqref{R11-exp}, we note that the equation on $\hat R^{11}$ can be rewritten as
\begin{equation}
 \hat R^{11} \,.\, (u\tensor v) =(-1)^\np\rho(\K)\,. (\id\tensor\omega^{\np-1})\Bigl[ (\id\tensor\omega)\bigl(\xi\tensor\one\cdot\RS^{11}\cdot \one\tensor\xi \bigr)\smult (\omega^{\np}(u)\tensor  v) \Bigr] \gp
\end{equation}
We use then the factorised expressions for $\xi$ in~\eqref{xi-def} 
and for $\RS^{11}$ in~\eqref{eq:R-element-for-S} together with the known solution for $\np=1$, 
recall Remark~\ref{rem:Phi-R-N1}, and it finally gives the expression in~\eqref{R11-exp}. By the construction of $\hat R$, the diagram~\eqref{eq:braiding-transport-comm-diag-SQ} commutes and this finishes the proof.
\end{proof}

Since $\repS$ is a braided monoidal category, the above lemma shows that
 $\hat R$ from~\eqref{eq:hatR} is the R-matrix 
 of $\QQ$  and by the construction of the transport of the braiding the family $\tau_{U,V}\circ \hat R$ defines the braiding in $\repQQ$.
Altogether we have shown:

\begin{proposition}\label{prop:functor-G-tensor-br}
The functor $\mathcal G$ in Proposition~\ref{prop:repS-repQ-monoid} is braided monoidal.
\end{proposition}

\subsection{The quasi-bialgebra 
\texorpdfstring{$\Q$}{Q} is the twisting of \texorpdfstring{$\QQ$}{Q-hat}} \label{sec:twist-Q-Qhat}

We define the twist $\zeta$ --
 an invertible element in $\QQ\tensor\QQ$: 
\be \label{zeta-def}
	\zeta=\idem_0\tensor\one +\zeta^{10} \cdot \idem_1\tensor\idem_0+ \zeta^{11} \cdot\idem_1\tensor\idem_1\gc
\ee
where
\be
	 \zeta^{10}= \left(\prod_{k=1}^\np \zeta^{10}_{(k)}\right) \cdot \one\tensor\K 	\qcq 
	 \zeta^{11}= \left(\prod_{k=1}^\np \zeta^{11}_{(k)}\right) \cdot \one\tensor\K^{\np-1} 
\ee
and
\begin{align}
	\zeta^{10}_{(k)} &=\one\tensor\one + (1-\rmi) \one\tensor\fp_k\fm_k +(1-\rmi)\fm_k\K\tensor\fp_k + (1+\rmi)\fp_k\K\tensor\fm_k-2\fp_k\fm_k\tensor\fp_k\fm_k  \gc  \\ \nonumber
	\zeta^{11}_{(k)} &= \one\tensor\one-(1+\rmi)\one\tensor\fp_k\fm_k  \gp 
\end{align}
Its  inverse is
\be 
	\zeta^{-1}=\idem_0\tensor\one + (\zeta^{10})^{-1} \cdot \idem_1\tensor\idem_0 + (\zeta^{11})^{-1} \cdot \idem_1\tensor\idem_1
\ee
with
\begin{align*} 
	 (\zeta^{10})^{-1}&= \one\tensor\K^{-1} \cdot \left(\prod_{k=1}^\np (\zeta^{10}_{(k)})^{-1}\right) 	\qcq 
	 (\zeta^{11})^{-1}= \one\tensor\K^{-(\np-1)} \cdot \left(\prod_{k=1}^\np (\zeta^{11}_{(k)})^{-1}\right)  \gc \\ \nonumber
	(\zeta^{10}_{(k)})^{-1} &=\one\tensor\one +(1+\rmi) \one\tensor\fp_k\fm_k -(1-\rmi)\fm_k\K\tensor\fp_k -(1+\rmi)\fp_k\K\tensor\fm_k -2\fp_k\fm_k\tensor\fp_k\fm_k \gc \\ \nonumber
	(\zeta^{11}_{(k)})^{-1} &= \one\tensor\one -(1-\rmi)\one\tensor\fp_k\fm_k \gp
\end{align*}
Since $(\eps\tensor\id)(\zeta)=\one=(\id\tensor\eps)(\zeta)$
and $\zeta$ is invertible, $\zeta$ is indeed a twist. 
This means (see e.g. \cite{Dr-quasi} or \cite{ChPr}) that it defines another quasi-triangular quasi-bialgebra 
$\QQ_\zeta$ 
with the same algebra structure and counit 
as in $\QQ$, 
while the new coproduct $\Delta_\zeta$, $R$-matrix $R_\zeta$ and coassociator $\Phi_\zeta$ are given by
\begin{align}  \label{twist:copord}
	\Delta_\zeta(x) &= \zeta \, \hat\Delta(x) \, \zeta^{-1} \gc  \\ 
	R_\zeta &= \zeta_{21} \, \hat R \, \zeta^{-1} \gc \\
	\Phi_\zeta  &= (\zeta\tensor\one) \cdot (\hat\Delta\tensor\id)(\zeta) \cdot \hat\Phi \cdot (\id\tensor\hat\Delta)(\zeta^{-1}) \cdot (\one\tensor\zeta^{-1})    \gp\label{twist:coass}
\end{align}
The action of the twist defines a multiplicative structure on the identity functor between the representation categories of both quasi-bialgebras. 
In particular, the categories are braided monoidally
 equivalent. 

\begin{proposition}\label{prop-twist-Q-Qhat}
We have $\QQ_\zeta = \Q$, that is,
$\Q$ as a quasi-triangular quasi-bialgebra
defined in Section~\ref{sec:Q-def}
 is obtained from $\QQ$ by twisting via~$\zeta$.
\end{proposition}

	To prove the proposition, we will need the following lemma.

\begin{lemma}\label{lem-twist-Q-Qhat} 
We  have for  $a,b,c\in\{0,1\}$ and $1\leq i\neq j\leq \np$ the equalities
\begin{align} \label{lem-twist-Q-Qhat-1}
 \left[\hat\Phi_{(i)}^{abc} , \hat\Phi_{(j)}^{abc} \right] &=0 \gcg & 
 \left[(\Delta\ot\id)(\zeta^{11}_{(i)}),\hat\Phi^{101}_{(j)}\right] &=0 \gcg  \\ \nonumber
 \left[\zeta^{ab}_{(i)},\zeta^{ab}_{(j)}\right]&=0 \gcg &
 \left[ \hat\Phi_{(i)}^{101}  , \K\tensor\K\tensor\K \right]&=0
 \end{align}
 and
 \be \label{lem-twist-Q-Qhat-2}
\left[\zeta^{10}_{(i)}\tensor\one,(\id\tensor\Delta)(\zeta^{11}_{(j)})\right]=0 \gp
\ee
\end{lemma}
\begin{proof}  
	This can be easily checked using the anti-commutator relations of $\Q$ from \eqref{def:Q} and by recalling that $\Q$ and $\QQ$ have the same algebra structure. 
\end{proof}

	\newcommand{\tzeta}{\underline{\zeta}}
	\newcommand{\uPhi}{\underline{\Phi}}

\begin{proof}[Proof of Proposition~\ref{prop-twist-Q-Qhat}]
Using~\eqref{twist:copord}-\eqref{twist:coass}, we have to show that  $R_\zeta=R$, $\Phi_\zeta=\Phi$ and $\Delta_\zeta=\Delta$
	with $R$, $\Phi$ and $\Delta$ defined in \eqref{R+Riv}, \eqref{Phi+Phi-inv} and \eqref{def:delta}, respectively.  
For $\np=1$ we verified these equalities  using computer algebra. 
The proof for general $\np$ will rely on the $\np=1$ case.  

	We start with the coproduct and show the statement sector by sector.  
	Clearly,  $\Delta_\zeta$ in \eqref{twist:copord} agrees with the coproduct in $\Q$ from \eqref{def:delta} in the sectors {\bf00} and {\bf01}.
Since $\Delta_\zeta$ is an algebra map, it is enough to show the equality
$\Delta_\zeta = \Delta$
 on the generators.  
	Below we will use the  abbreviation
\be\label{def:tzeta+uphi}
		\tzeta^{ab} = \zeta^{ab}\cdot\idem_a\ot\idem_b \ , \qquad a,b\in\{0,1\}\ ,
\ee
	and similarly for $\tzeta^{ab}_{(k)}$.
By applying Lemma~\ref{lem-twist-Q-Qhat} and using that for $\np=1$ the statement is true we get in sector {\bf10}, for 
$x_i\in\{\K, \fp_i, \fm_i\}$,
\begin{align}  
	\tzeta^{10} \, \hat\Delta(x_i) \, (\tzeta^{10})^{-1} 
	&= \left(\prod_{k=1}^\np \tzeta^{10}_{(k)}\right) \one\tensor\K \cdot \hat\Delta(x_i) \cdot \one\tensor\K^{-1} \left(\prod_{k=1}^\np (\tzeta^{10}_{(k)})^{-1}\right) \\ \nonumber
	&= \left(\prod_{k\neq i} \tzeta^{10}_{(k)}\right) \cdot \Delta(x_i) \cdot \left(\prod_{k\neq i} (\tzeta^{10}_{(k)})^{-1}\right)
	= \Delta(x_i) \gc
\end{align}
where we also used that $\tzeta^{10}_{(k)}$ commutes with  $\Delta(x_i)$ for $i\ne k$.
For the sector {\bf11}, we prove it similarly.

The twisted $R$-matrix is 
$R_\zeta = 
\zeta_{21} \, \hat R \, \zeta^{-1}$
 and we begin with the case $\np=1$. 
 The direct calculation gives (for $\beta^4=-1$)
\begin{multline}
R_\zeta  =  \half\sum_{i,j,k=0}^1 2^k \rmi^{2k(i-j+1)-2ij} (\ff^-_1)^k \K^{k+i}\tensor (\ff^+_1)^k\K^j\\
\times (\idem_0\tensor\idem_0 + \rmi^{-i-k}\idem_1\tensor\idem_0 + \rmi^j\idem_0\tensor\idem_1 +  \rmi^{-i-k+j}\beta\idem_1\tensor\idem_1) \ .
\end{multline}
 We can rewrite $R_\zeta$ in a more compact  form involving the $\oZ_2$-parity $\omega=\K(\idem_0-\rmi\idem_1)$:
\begin{equation}
R_\zeta  = \sum_{n,m=0}^1 \beta^{nm} \rho_{n,m} \cdot (\one\tensor\one - 2 \ff^-_1 \omega\tensor \ff^+_1)\cdot\idem_n\tensor\idem_m\ ,
\end{equation}
with the Cartan factor
$\rho_{n,m}$ as in~\eqref{eq:car-fac}.
In the above calculation we used the identities $\rho_{n,0} = \rho(\K)\cdot\idem_n\tensor\idem_0$ and $\rho_{n,1} =(-\rmi)^{n}\K\tensor\one\cdot \rho(\K)\cdot\idem_n\tensor\idem_1$, with $\rho(\K)$ introduced in~\eqref{sflip-om}.
For general $\np$, 
using the factorised form of $\zeta$  we obtain the expression in terms of  $\np=1$ terms:
\begin{equation}
R_\zeta  \cdot\idem_n\tensor\idem_m =  \beta^{nm} \rho_{n,m} \cdot \prod_{k=1}^{\np}(\one\tensor\one - 2 \ff^-_k \omega\tensor \ff^+_k)\cdot\idem_n\tensor\idem_m,
\qquad n,m\in\oZ_2,
\end{equation}
where now $\beta^4=(-1)^{\np}$.
This proves $R=R_\zeta$ with the $R$-matrix for $\Q$ given in \eqref{R+Riv}
where one has to use $\omega_-=\omega$.

We now turn to the calculation of the twisted coassociator.
We prove the equality 
$\Phi_\zeta=\Phi$  
 in the sector {\bf101},
the proof for 	the
other sectors is similar. 
Recall the definition of the coassociators $\Phi$ and $\hat\Phi$ in \eqref{Phi+Phi-inv} and \eqref{eq:Qhat-coass-I}--\eqref{eq:Qhat-coass-III}, respectively.   
We have to show that $\Phi_\zeta^{101}= \one^{\otimes 3}$.
The equation \eqref{twist:coass} reduces to 
\begin{align}  \label{eq:Phi-twist-101}
	\Phi^{101}_{\zeta}\cdot\idem_1\ot\idem_0\ot\idem_1= \tzeta^{10}\tensor\one\cdot(\hat\Delta\tensor\id)(\tzeta^{11})\cdot\hat\uPhi^{101}\cdot (\id\tensor\hat\Delta)\bigl((\tzeta^{11})^{-1}\bigr) \gc
\end{align}
where we used the notations in \eqref{eq:tilde-Phi}
and~\eqref{def:tzeta+uphi},
and  that $\zeta=\one\ot\one$ in sector {\bf01}.
Using computer algebra
for $\np=1$
 it turned out that $\Phi^{101}_\zeta=\one^{\ot3}$. 
Taking this into account, the equality~\eqref{eq:Phi-twist-101} 
for $\np=1$
 is then equivalent to
\begin{align}  \label{eq:twist-assoc-1}
(\id\tensor\hat\Delta)(\tzeta^{11}_{(k)})\cdot
\one\tensor\K\tensor\one= \tzeta^{10}_{(k)}\tensor\one \cdot\one\tensor\K\tensor\one\cdot(\hat\Delta\tensor\id)(\tzeta^{11}_{(k)})\cdot\hat\uPhi^{101}_{(k)} \gc
\end{align}
where $k=1$ and we used that $\K^{-1}\idem_0 = \K \idem_0$. We note that the above equation also holds for general $k$.
Together with
 Lemma~\ref{lem-twist-Q-Qhat} we get for general $\np$:
\begin{align}  \label{eq:twist-calc-1}
	\tzeta^{10}\tensor\one\cdot(\hat\Delta\tensor\id)(\tzeta^{11})\cdot\hat\uPhi^{101}
	&= \big(\prod_{i=1}^\np\tzeta^{10}_{(i)}\tensor\one\big) \cdot \one\tensor\K\tensor\one
	\cdot \Big(\prod_{i=1}^\np(\hat\Delta\tensor\id) (\tzeta^{11}_{(i)}) \Big) \cdot \one\tensor\one\tensor\K^{\np-1}  \\ \nonumber
	&\quad
\times 
 (-\K)^{\np-1}\tensor \K^{\np-1}\tensor\one \cdot \big(\prod_{i=1}^\np\hat\uPhi_{(i)}^{101} \big) \cdot \K^{\np-1}\tensor\K\tensor\one \\ \nonumber
	&= \big(\prod_{i=1}^{\np-1}\tzeta^{10}_{(i)}\tensor\one\big) \cdot \Big(\tzeta^{10}_{(\np)}\tensor\one \cdot \one\tensor\K\tensor\one 
	\cdot (\hat\Delta\tensor\id) (\tzeta^{11}_{(\np)}) \cdot \hat\uPhi_{(\np)}^{101} \Big) \\ \nonumber
	&\quad
	\times
	 \Big(\prod_{i=1}^{\np-1}(\hat\Delta\tensor\id) (\tzeta^{11}_{(i)})  \cdot \hat\uPhi_{(i)}^{101}  \Big) \cdot \one\tensor\K^{\np} \tensor\K^{\np-1} \gc
\end{align}
where the first equality is by definition of $\zeta$ and $\hat\Phi^{101}$, while we used the relations in \eqref{lem-twist-Q-Qhat-1} at the second equality.
Next, by applying \eqref{eq:twist-assoc-1} for $k=\np$ to the previous expression we get
\begin{align}  \label{eq:twist-calc-2}
\tzeta^{10}\tensor\one\cdot(\hat\Delta\tensor\id)(\tzeta^{11})\cdot\hat\uPhi^{101}
	&= \big(\prod_{i=1}^{\np-1}\tzeta^{10}_{(i)}\tensor\one\big) \cdot \Big((\id\tensor\hat\Delta)(\tzeta^{11}_{(\np)})  \cdot \one\tensor\K\tensor\one \Big) \\ \nonumber
	&\quad 
	\times
	\Big(\prod_{i=1}^{\np-1}(\hat\Delta\tensor\id) (\tzeta^{11}_{(i)})  \cdot \hat\uPhi_{(i)}^{101}  \Big) \cdot \one\tensor\K^{\np} \tensor\K^{\np-1}\gp
\end{align}
Finally using first the relation~\eqref{lem-twist-Q-Qhat-2} and then doing the reordering of the terms in the products as in~\eqref{eq:twist-calc-1} and~\eqref{eq:twist-calc-2} for $i=\np-1, \dots, 1$,  we obtain
\begin{align}
	\text{RHS of}~\eqref{eq:twist-calc-2}&= \Big(\prod_{i=1}^\np (\id\tensor\hat\Delta)(\tzeta^{11}_{(i)}) \Big) \cdot \one\tensor\K\tensor\one \cdot \one\tensor\K^{\np} \tensor\K^{\np-1} \\ \nonumber
	&=(\id\tensor\hat\Delta)(\tzeta^{11}) 	\gp
\end{align}
Hence, $\Phi^{101}_\zeta=\one^{\ot3}=\Phi^{101}$ which completes the proof.
\end{proof}

\subsection{Transporting the ribbon twist}\label{sec:trans-twist-SQ}
The ribbon twist in $\catSF$ is described in Section~\ref{eq:SF-twist}. 
Following the same lines as in \cite[Sec.\,7.9]{Gainutdinov:2015lja}
and using that $\omega_{\B}$ is given by the left-action of 
$\W$ in~\eqref{Omega-def}
we can write, for $M\in\repQ$, $m\in M$,
\begin{equation}\label{eq:Q-twist}
\theta_M(m) = \Big( \idem_0 \prod_{k=1}^\np (\one + 2 \ff^+_k\ff^-_k) - \rmi \beta^{-1} \idem_1 \K \prod_{k=1}^\np (\one-2 \ff^+_k \ff^-_k) \Big).m \gp
\end{equation}
The action of the ribbon element $\ribbon\in\Q$ on $M$ is by convention
equals
 the inverse twist. It follows that  
\be
\ribbon =  (\idem_0-\beta\rmi\K\idem_1)\cdot\prod_{k=1}^\np(\one-2\fp_k\fm_k) \gp
\ee  

\subsection{Ribbon equivalence
\texorpdfstring{$\funQSF\colon\catSF\to\repQ$}{F:SF->Rep(Q)}}\label{app:rib-equiv-QSF}
Taking Appendices~\ref{app:SF-S} and \ref{app:Q-S} together 
(Propositions~\ref{prop:functor-D-tensor},~\ref{prop:functor-D-tensor-br} and~\ref{prop:repS-repQ-monoid},~\ref{prop:functor-G-tensor-br}, and~\ref{prop-twist-Q-Qhat}), we have established a ribbon equivalence
 $\funQSF\colon\catSF\to\repQ$ which is the composition 
 \be
 \funQSF=
\funSQ\circ \funCS
 \ee
 where the functor $\funCS$ is defined in~\eqref{eq:D-functor-sec0-def} and~\eqref{eq:D-functor-sec1-def}, and the functor $\funSQ$ is in Section~\ref{sec:RepS-RepQ}.
  The monoidal structure of $\funQSF$ is defined by
 \be
 \funQSF_{U,V} 
 \colon \funQSF(U*V) \xrightarrow{\sim} \funQSF(U)\tensor_{\repQ}\funQSF(V) 
 \ee
with 
 \be
 \funQSF_{U,V} = \big(\zeta.(-)\big) \circ \Gamma_{\funCS(U),\funCS(V)}\circ \funSQ(\Delta_{U,V})\gc
 \ee
 where  $\zeta$, $\Gamma_{U,V}$ and $\Delta_{U,V}$ are given in \eqref{zeta-def}, \eqref{eq:isoG} and in Section \ref{sec:funD}, respectively. 
  This monoidal equivalence is by construction braided and ribbon and thus finally proves Lemma~\ref{lem-trans-QSF}.

%%%%%%%%%%%%%%%%%%%%%%%%%

\section{Proof of Proposition~\ref{DH:main-prop}} \label{app:DD}

In this Appendix, we use for brevity $H$ instead of $H(\np)$, for the Hopf algebra introduced in~\eqref{H-alg1} and~\eqref{H-alg2}, and we fix the basis    in~$H$:
\be\label{eq:Hbasis}
H = \mathrm{span}_\oC \left\{f_{u_1}\cdots f_{u_m}k^v\,|\,
0\leq m \leq \np,
 1\leq u_1 < u_2 < \ldots < u_m\leq\np, 
v\in \mathbb{Z}_2
\right\}\ .
\ee
The coproduct $\Delta$ on the basis elements~\eqref{eq:Hbasis} is given by
\begin{align}\label{eq:Delta-H}
\Delta(f_{u_1}\cdots f_{u_m}k^v) &=
 \left(\prod_{u=u_1}^{u_m} (f_u\tensor k + \one\tensor f_u) \right)\cdot k^v\tensor k^v \\ \nonumber
&= \Big(\sum_{l=(l_1,\ldots,l_m)\in\oZ_2^m} f_{u_1}^{l_1}\cdots f_{u_m}^{l_m} \tensor f_{u_1}^{1-l_1} k^{l_1} \cdots f_{u_m}^{1-l_m} k^{l_m}\Big) \cdot k^v\tensor k^v \\ \nonumber
&= \Big(\sum_{l\in\oZ_2^m}
  \eps_l(m)
   \, f_{u_1}^{l_1}\cdots f_{u_m}^{l_m} k^v \tensor f_{u_1}^{1-l_1} \cdots f_{u_m}^{1-l_m} k^{v+|l|}\Big)  \gp
\end{align}
where $\eps_l(m)=(-1)^{\sum_{i=1}^{m-1} l_i((m-i)-\sum_{j=i+1}^m l_j)}  $ and $|l|=\sum_i l_i \in\oN$.

For the Drinfeld double construction we need  the dual Hopf algebra
\be
(H^{\rm{op}})^\ast = \left(H^\ast,\, 
\mu_{H^\ast}=\Delta^\ast,\, \one_{H^\ast}=
\eps^\ast,\,
 \Delta_{H^\ast}=\bigl(\mu^{\mathrm{op}}\bigr)^\ast,\, 
\eps_{H^\ast}=\eta^\ast,\,
 S_{H^\ast}=(S^{-1})^\ast\right)\gc
\ee
where
$\mu^{\mathrm{op}}$ is the opposite multiplication,
and we used the standard isomorphism of vector spaces $(H\tensor H)^\ast \cong H^\ast\tensor H^\ast$,
so $\Delta_{H^\ast}(\varphi) (a\tensor b) = \varphi(ba)$ and 
$(\varphi \cdot \psi)(a) = (\varphi \tensor \psi) \bigl(\Delta(a)\bigr)$.
Using~\eqref{eq:Delta-H}, we note that the canonical duals of the generators of $H$ do not generate $(H^{\rm{op}})^\ast$ since e.g.\ $f_i^\ast \cdot f_j^\ast = (f_i^\ast \tensor f_j^\ast)\circ \Delta=0$. 
 We then instead use  the linear forms
\begin{align}
\kappa = \one^\ast - k^\ast \gc \qquad \varphi_i = (f_ik)^\ast-f_i^\ast \gp
\end{align}
\begin{lemma} \label{lem:H-dual-Hopf-alg} 
The  algebra $(H^{\rm{op}})^\ast$ is generated by $\kappa$ and $\varphi_i$, $1\leq i\leq \np$, with the defining relations
\begin{align}\label{app:DH:Hs-alg}
\kappa^2 =\one_{H^\ast} \gc \qquad 
 \{\varphi_i,\varphi_j\}=0 \gc \qquad \{\varphi_i,\kappa\}=0\gc 
\end{align}
and	the set 
	\be\label{eq:Hast-basis}
\{\varphi_{i_1}\ldots \varphi_{i_m}\kappa^j \mid 0\leq m\leq\np,  1\leq i_1 < i_2 < \ldots < i_m\leq\np,  j\in\oZ_2\}
	\ee
	 forms a basis.
 The Hopf-algebra structure of $(H^{\rm{op}})^\ast$ is
\begin{align}\label{app:DH:Hs-hopf}
  \Delta_{H^\ast}(\kappa) 
	&= \kappa\tensor\kappa \gc \qquad  \Delta_{H^\ast}(\varphi_i)= \varphi_i\tensor\one_{H^\ast} + \kappa\tensor\varphi_i   \gc \\
	\eps_{H^\ast}(\kappa) &=1 \gc \qquad \qquad\; \eps_{H^\ast}(\varphi_i) = 0 \gc\nonumber \\
	S_{H^\ast}(\kappa) 
	&=\kappa \gc  \qquad \qquad  S_{H^\ast}(\varphi_i) = \varphi_i\kappa \gp\nonumber
\end{align}
\end{lemma}
\begin{proof} 
We begin with   the defining relations. The first one from~\eqref{app:DH:Hs-alg} is straightforward to check using  $\one_{H^\ast}=\eps= \one^\ast + k^\ast$. The next one follows from the calculation:
\begin{align}
\varphi_i \varphi_j &= \big(((f_ik)^\ast-f_i^\ast)\tensor ((f_jk)^\ast-f_j^\ast)\big)\circ \Delta 
= -((f_i k)^\ast\tensor f_j^\ast+f_i^\ast\tensor (f_j k)^\ast)\circ \Delta \\ \nonumber
&= (f_if_jk)^\ast+(f_if_j)^\ast =-((f_jf_ik)^\ast+(f_jf_i)^\ast)=  -\varphi_j \varphi_i \gc
\end{align}
where we used  the coproduct formula~\eqref{eq:Delta-H},
and similarly for the third relation.

Now, we construct a basis in $(H^{\rm{op}})^\ast$.
Using induction, one can check the relations
\begin{align} \label{eq:varphi-string}
(-1)^m\varphi_{i_1}\ldots \varphi_{i_m} 
&=  (f_{i_1}\ldots f_{i_m})^\ast  + (-1)^m(f_{i_1}\ldots f_{i_m}k)^\ast\\   \nonumber
\varphi_{i_1}\ldots \varphi_{i_m} \kappa &= (f_{i_1}\ldots f_{i_m})^\ast-(-1)^m (f_{i_1}\ldots f_{i_m}k)^\ast\gp
\end{align}
Since the set of elements 
$(f_{u_1}\ldots f_{u_m}k^v)^\ast$,
with indices as in~\eqref{eq:Hbasis}, 
forms a basis in $(H^{\rm{op}})^\ast$, we conclude from the above relations that the set~\eqref{eq:Hast-basis} is also a basis in $(H^{\rm{op}})^\ast$.

For the coproduct we have by definition $\Delta_{H^\ast}(\varphi)(a\otimes b) = \varphi(ba)$. This can be written as
\be\label{eq:phi-cop-proof}
\Delta_{H^\ast}(\varphi_i) = ((f_ik)^\ast-f_i^\ast)
	\circ\mu^{\mathrm{op}} \gp
\ee
For $f_i^\ast \circ\mu^{\mathrm{op}}(a\otimes b)$
we should find such pairs $(a,b)$ that $f_i^\ast(ba)$ is non-zero. There are four such pairs and we get
\be
f_i^\ast \circ\mu^{\mathrm{op}}
=  f_i^\ast\tensor \one^\ast + \one^\ast\tensor f_i^\ast - (f_ik)^\ast\tensor k^\ast + k^\ast\tensor f_ik^\ast \gc
\ee
and we get a similar expression for
$(f_ik)^\ast\circ\mu^{\mathrm{op}}$.
Combining the eight total terms in~\eqref{eq:phi-cop-proof} we get the coproduct in~\eqref{app:DH:Hs-hopf}. 
For the antipode we have
\begin{align} 
S_{H^\ast}(\varphi_i) 
&= ((f_ik)^\ast-f_i^\ast) \circ S^{-1} = f_i^\ast+(f_ik)^\ast  
= \varphi_i\kappa  \gp
\end{align}
Calculations for the coproduct and antipode  for $\kappa$, together with the counit, are straightforward.
This finally proves the lemma. 
\end{proof}

Next we present $D(H)$, the Drinfeld double of $H$, by following the conventions in \cite[Chapter IX]{Kassel-book}.
As a vector space $D(H)$ is $H^\ast\tensor H$. 
It has a Hopf algebra structure with unit $\one_{H^\ast}\tensor\one$ and  counit $\eps(\phi\tensor a) = \eps_{H^\ast}(\phi)\eps( a)$,  with the multiplication  defined as
\begin{equation}\label{app:DH:maps}
 (\phi\tensor a) \cdot (\psi\tensor b) = \sum_{(a)} 
\phi \cdot \psi (S^{-1}(a''')(-)a')\tensor a''b  \gc
\end{equation}
where $\psi(S^{-1}(a''')(-)a')$ stands for the map $(x\mapsto \psi(S^{-1}(a''')xa'))$.
The coproduct and the antipode are given by
\begin{equation}\label{app:DH:comaps}
\Delta (\phi\tensor a) = \sum_{(\phi),(a)} (\phi'\tensor a')\tensor (\phi''\tensor a'') \gc 
\quad S(\phi\tensor a) =\bigl(\one_{H^\ast}\tensor S(a)\bigr) \cdot \bigl(S_{H^\ast}(\phi)\tensor \one\bigr) \gp
\end{equation}
We will identify an element $a$ in $H$ with $\one_{H^\ast}\tensor a$ and an element $\phi$ in $(H^{\rm{op}})^\ast$ with $\phi\tensor \one$, 
in particular, we write $\phi a = (\phi\tensor\one)\cdot(\one_{H^\ast}\tensor a)$.
In this notation, the basis of $D(H)$ is
\be
\{\varphi_{i_1}\cdots\varphi_{i_m}\kappa^u f_{j_1}\cdots f_{j_n}k^v\mid0\leq m,n\leq\np, u,v\in\oZ_2\} \gc
\ee 
where as usual we assume that $1\leq i_1 < i_2 < \ldots < i_m\leq\np$ and similarly for the $j$'s indices.
It is well-known that $D(H)$ is quasi-triangular: for any basis $\{b_i| \, i\in I\}$ in $H$ the $R$-matrix is given by 
(using the convention above we interpret $b_i, b_i^*\in D(H)$)
\be\label{eq:DH-Rmatrix}
R_D=\sum_{i\in I} b_i\tensor b_i^\ast\gp
\ee

\begin{proposition} \label{lem:Dd-alg}
$D(H)$ is generated by $k,\kappa,f_i,\varphi_j$ with defining relations \eqref{H-alg1}, \eqref{app:DH:Hs-alg} and 
\begin{align} \label{app:DH:alg}
	k\kappa &=\kappa k \gc \qquad \varphi_i k=-k\varphi_i \gc \qquad f_i\kappa=-\kappa f_i \gc\qquad 
	[f_i,\varphi_j]=\delta_{i,j}(\kappa -k) \gp
\end{align}
The Hopf algebra structure is given by \eqref{H-alg2} and \eqref{app:DH:Hs-hopf}. 
The $R$-matrix for $D(H)$ is given by
\begin{align} \label{DH:R-mat}
R_D =\frac12 (\one\tensor\one+\one\tensor\kappa+k\tensor\one-k\tensor\kappa)\Big(\sum_{\substack{0\leq m\leq\np,\\i_1<\ldots < i_m}} 
(-1)^m  f_{i_1}\ldots f_{i_m} \tensor\varphi_{i_1}\ldots\varphi_{i_m} \Big) \gp
\end{align}
\end{proposition}
\begin{proof}
Using \eqref{app:DH:maps} it is straightforward  to show the relations in \eqref{app:DH:alg}. 
For example, we get the last equality 
since $\sum_{(f)} f_i'\tensor f_i'' \tensor f_i''' = f_i\tensor k\tensor k+ \one\tensor f_i\tensor k + \one\tensor \one\tensor f_i$ and then
\begin{align}
f_i\varphi_j  
&= \varphi_j(k-f_i)  k + \varphi_j(k-)  f_i + \varphi_j(f_ik-)  \one \\ \nonumber
&= -\delta_{i,j} (\one^\ast+k^\ast)  k + (-f_j^\ast+(f_jk)^\ast)  f_i + \delta_{i,j} (\one^\ast-k^\ast)  \one  \gp
\end{align}
The coproduct, counit, and antipode on the generators were already computed.

For the $R$-matrix in~\eqref{eq:DH-Rmatrix}, we fix the basis $\{b_i\}$ as in~\eqref{eq:Hbasis}. Then
\begin{align}
	R_D &= \sum_{\substack{0\leq m\leq\np,\\1\leq i_1<\ldots < i_m\leq\np}} \sum_{0\leq j\leq 1}  f_{i_1}\ldots f_{i_m}k^j \tensor (f_{i_1}\ldots f_{i_m}k^j)^\ast \gp
\end{align}
By applying \eqref{eq:varphi-string} 
we get
\begin{align}
(f_{i_1}\ldots f_{i_m})^\ast &= \ffrac12\varphi_{i_1}\ldots \varphi_{i_m}((-1)^m\one_{H^\ast}+\kappa) \label{eq:f-phi-1} \gc \\ \nonumber
(f_{i_1}\ldots f_{i_m}k)^\ast&= \ffrac12\varphi_{i_1}\ldots \varphi_{i_m}(\one_{H^\ast}-(-1)^m\kappa))  
\end{align}
and therefore
\be
	R_D = \Big(  \sum_{\substack{0\leq m\leq\np,\\i_1<\ldots < i_m}} f_{i_1}\ldots f_{i_m} \tensor  \varphi_{i_1} \ldots \varphi_{i_m} \Big) 
	    \ffrac12 \big(	(-1)^m \one\tensor\one+\one\tensor\kappa+k\tensor\one-(-1)^m k\tensor\kappa \big) 
	    \ee
which  gives~\eqref{DH:R-mat} using the relations between $f_i$ and $k$, and $\varphi_i$ and $\kappa$.
\end{proof}

\begin{proposition}
For any $\np\geq1$,
the linear map 
\begin{align}\label{iso-Dd-Q}
	\Psi (\varphi_{i_1}\cdots\varphi_{i_m}\kappa^u f_{j_1}\cdots f_{j_n}k^v) 
	= (-1)^{nu}\,\rmi^{n(n-1)}2^m \fp_{i_1}\cdots\fp_{i_m}\fm_{j_1}\cdots\fm_{j_n}\,\omega_+^u\omega_-^{v+n} \ ,
\end{align}
with $\omega_\pm$  from~\eqref{def:delta},
defines an isomorphism of  $\oC$-algebras  
$D\bigl(H(\np)\bigr)\xrightarrow{\sim}\Q(\np,\beta)$.\\
Furthermore, for even $\np$ this map
defines an isomorphism of Hopf algebras
 between $D\bigl(H(\np)\bigr)$ and	 $\Q(\np,\pm 1)$. \\
Moreover, for $\beta=1$ the $R$-matrix in $\Q$ defined in \eqref{R+Riv} and the image of $R_D$ under $\Psi\tensor\Psi$ coincide,
	i.e.\ $\Psi$ is an isomorphism of quasi-triangular Hopf algebras.
\end{proposition}
\begin{proof}
In order to show that $\Psi$ is an  algebra map it is enough to verify the relations in 
\eqref{H-alg1},~\eqref{app:DH:Hs-alg}  and~\eqref{app:DH:alg}
for the the image of the generators of $D(H)$ under $\Psi$. Since we have
\begin{align*}
	\Psi(\varphi_i) =	 2\fp_i \gc \qquad \Psi(\kappa) =
	  \omega_+ \gc \qquad \Psi(f_i) = \fm_i\omega_- \gc \qquad  \Psi(k) = \omega_- \gc 
\end{align*} 
this is an easy check. For example,
\begin{align}
	\Psi(f_i)\Psi(\varphi_j) &=
 \omega_-(\delta_{i,j}(\K^2-\one)+2\fp_j\fm_i) = \delta_{i,j}(\omega_+-\omega_-)+2\fp_j\fm_i\omega_- \\ \nonumber
	&=\delta_{i,j}(\Psi(\kappa)-\Psi(k))+\Psi(\varphi_j)\Psi(f_i) 
\end{align}
Since $\frac{1}{1+\rmi}\Psi(\kappa+\rmi k)=\K$ and $\Psi(f_j\kappa)=\fm_j$ the image of $\Psi$ clearly generates $\Q$.
Moreover, since the dimensions of $D(H)$ and $\Q$ agree $\Psi$ is  bijective.

To show that $\Psi$ is a coalgebra map we use~\eqref{H-alg2}  and~\eqref{app:DH:Hs-hopf}, and check
\begin{align}
		(\Psi\tensor\Psi)(\Delta(\varphi_i)) &=(\Psi\tensor\Psi)(\varphi_i\tensor\one + \kappa\tensor\varphi_i) 
	= 2\fp_i\tensor\one + 2 \omega_+\tensor\fp_i = \Delta(2\fp_i) = \Delta(\Psi(\varphi_i))\gc\\\nonumber
	(\Psi\tensor\Psi)(\Delta(f_i)) &=(\Psi\tensor\Psi)(f_i\tensor k+ \one\tensor f_i) 
	= \fm_i\omega_-\tensor\omega_- + \one\tensor\fm_i\omega_- = \Delta(\fm_i\omega_-) = \Delta(\Psi(f_i))\gc\\\nonumber
	(\Psi\tensor\Psi)(\Delta(\kappa)) &=(\Psi\tensor\Psi)(\kappa\tensor \kappa) 
	= \omega_+\tensor\omega_+  \stackrel{(*)}{=} \Delta(\omega_+) = \Delta(\Psi(\kappa))  \gp
\end{align}
Note, the equality $(*)$ is in $\Q$ and it holds only if $\np$ is even.
For the antipode, we have
\begin{align}
	\Psi(S(\varphi_i)) &=\Psi(\varphi_i\kappa) = 2\fp_i\omega_+ = S(2\fp_i) = S(\Psi(\varphi_i))\gc\\\nonumber
	\Psi(S(f_i)) &=-\Psi(f_ik) = -\fm_i = S(\fm_i\omega_-) = S(\Psi(f_i))\\\nonumber
	\Psi(S(\kappa)) &=\Psi(\kappa) = \omega_+ = S(\omega_+) = S(\Psi(\kappa))\gp
\end{align} 

Recall the $R$-matrices in $\Q$ and $D(H)$ defined in \eqref{R+Riv}  and \eqref{DH:R-mat}, respectively. 
The image of $R_D$ under $\Psi\tensor\Psi$ leads to 
(using $\beta=1$)
\begin{align*}
	(\Psi\tensor\Psi)(R_D) 
	&=  \frac12 (\one\tensor\one+\one\tensor\omega_++\omega_-\tensor\one-\omega_-\tensor\omega_-)
	\Big(\sum
	 (-2)^{m} \fm_{i_1}\omega_-\ldots \fm_{i_m}\omega_-\tensor \fp_{i_1}\ldots \fp_{i_m}  \Big) \\ 
	&=  \sum_{u,v=0}^1 \rho_{u,v} \cdot \prod_{k=1}^{\np}(\one\tensor\one - 2 \ff^-_k \omega_-\tensor \ff^+_k)\cdot\idem_u\tensor\idem_v  = R \gc
\end{align*}
where the sum is taken over $1\leq i_1<\ldots < i_m\leq\np$ with $0\leq m\leq\np$.
This calculation finishes the proof of the statement.
\end{proof}

%%%%%%%%%%%%%%%%%%%%%%%%%%%%%%%%%%%%     bibliography
\newcommand\arxiv[2]      {\href{http://arXiv.org/abs/#1}{#2}}
\newcommand\doi[2]        {\href{http://dx.doi.org/#1}{#2}}
\newcommand\httpurl[2]    {\href{http://#1}{#2}}

\end{document}